\documentclass[12pt,reqno]{amsart} 
\usepackage{amsmath}
\usepackage{amssymb, latexsym, amsfonts, amscd, amsthm, mathrsfs, enumerate, esint}
\usepackage[usenames,dvipsnames]{color}
\usepackage[all]{xy}
\usepackage{graphicx}
\usepackage{epsfig}
\usepackage{hyperref}
\usepackage{marginnote}
\usepackage{mathtools} 			
\usepackage{bbm}
\usepackage{amssymb}
\usepackage{amsmath}
\usepackage{amsfonts,latexsym,amstext}
\usepackage{marginnote}

\numberwithin{equation}{section}

\numberwithin{equation}{section}

\newcommand{\abbr}[1]{{\sc\lowercase{#1}}}

\newtheorem{thm}{Theorem}[section]
\newtheorem{cor}[thm]{Corollary}
\newtheorem{lem}[thm]{Lemma}
\newtheorem{ppn}[thm]{Proposition}
\newtheorem{rmk}[thm]{Remark}

\theoremstyle{definition}
\newtheorem{defn}[thm]{Definition}
\newtheorem{ex}[thm]{Example}
\newtheorem*{assumption}{Assumption}

\newtheorem{problem}[thm]{Problem}

\newcommand{\beq}{\begin{equation}}
\newcommand{\eeq}{\end{equation}}

\newcommand{\ep}{\epsilon}

\newcommand{\vphi}{\varphi}

\newcommand{\B}{\mathbb{B}}

\newcommand{\E}{\mathbb{E}}

\newcommand{\N}{\mathbb{N}}

	\renewcommand{\P}{\mathbb{P}}

\newcommand{\R}{\mathbb{R}}
\newcommand{\bS}{\mathbb{S}}

\newcommand{\Z}{\mathbb{Z}}

\newcommand{\ovl}{\overline}
\newcommand{\udl}{\underline}

\newcommand{\cA}{\mathcal{A}}
\newcommand{\cAe}{\mathcal{A}^{\sf e}}

\newcommand{\cC}{\mathcal{C}}
\newcommand{\cD}{\mathcal{D}}

\newcommand{\cF}{\mathcal{F}}

\newcommand{\cK}{\mathcal{K}}
\newcommand{\cL}{\mathcal{L}}

\newcommand{\cN}{\mathcal{N}}

\newcommand{\cT}{\mathcal{T}}

\newcommand{\ga}{\mathfrak{a}}

\newcommand{\wt}{\widetilde}
\newcommand{\wh}{\widehat}

\DeclareMathOperator{\argmin}{arg\,min}
\DeclareMathOperator{\argmax}{arg\,max}

\newcommand{\etas}{\eta_\ep}

\title[A random growth model with memory]{Averaging principle and shape theorem 
\\ for a growth model with memory}

\author[A. Dembo]{Amir Dembo}
\address{Department of Mathematics, Stanford University.}
\email{adembo@stanford.edu}
\author[P. Groisman]{Pablo Groisman}
\address{Departamento de Matemática, FCEN, Universidad de Buenos Aires,
IMAS-CONICET and NYU-ECNU Institute of Mathematical Sciences at
NYU Shanghai.}
\email{pgroisman@dm.uba.ar}
\author[R. Huang]{Ruojun Huang}
\address{Courant Institute.}
\author[V. Sidoravicius]{Vladas  Sidoravicius}
\address{Courant Institute and NYU-ECNU Institute of Mathematical Sciences at
NYU Shanghai.}

\begin{document}

\begin{abstract}
We present a general approach to study a class of random growth models in $n$-dimensional Euclidean space. These models are designed to capture basic growth features which are  expected to manifest at {\it the mesoscopic} level for several classical self-interacting processes originally defined at {\it the microscopic scale}. It includes once-reinforced random walk with strong reinforcement, origin-excited random walk, and few others, for which the set of visited vertices  is expected to form a ``limiting shape''. We prove an averaging principle that leads to such a shape theorem. The limiting shape can be computed in terms of the invariant measure of an associated Markov chain. 
\end{abstract}



\thanks{This research was supported in part by NSF grant DMS-1613091, UBACYT grant 20020160100147BA and PICT 2015-3154.}

\dedicatory{V. S. passed away in May 2019. The other authors dedicate this paper to his memory.}


\maketitle

\section{Introduction}


Random growth processes arise in great variety  in a large class of physical and biological phenomena, network dynamics, etc. Starting from seminal works  of Eden \cite{Ed} and Hammersley and Welsh \cite{HW}, a series of mathematical models have been developed to capture and understand the evolution and pattern formation
of growth processes. Our motivation stems from Laplacian growth models, which are characterized by the fact that the rate at which each portion of the boundary of the domain grows is determined by the harmonic measure of the domain from some given point, which we call {\em source}. The list includes Diffusion Limited Aggregation (\abbr{DLA}) \cite{WS}, its generalization -- Dielectric Breakdown Model (\abbr{DBM}) \cite{NPW}, Hastings-Levitov process \cite{HaLe}; Internal \abbr{DLA} (\abbr{IDLA})  \cite{DF,LBG}, abelian sandpiles and rotor aggregation \cite{LP}. It also includes {\em once-reinforced} random walk  with strong reinforcement (\abbr{ORRW}) \cite{Dav}, and 
{\em origin-excited} random walk (\abbr{OERW}) \cite{K2},  for which the set of visited vertices  is expected to form a limiting shape.  For models such as \abbr{DLA}, \abbr{DBM} or Hastings-Levitov, the source is at infinity, while 
in models such as \abbr{IDLA}, the source is at the origin. 
Whenever the source is fixed, the process of growing in 
time domains is Markovian. In contrast, the latter process is non-Markovian 
in \abbr{ORRW} or excited random walks, where the source is moving and 
depends strongly on the last hitting point of the boundary and current  
shape of the domain.

In general, lattice growth models of this type are
elusive, specially when the source is at infinity or when it is not fixed. A notable exception is \abbr{IDLA} for which
Lawler, Bramson and Griffeath
obtained a shape theorem (see \cite{LBG}).
Specifically, here particles are emitted in steps, one by one, from the source which is always located at the origin, and perform simple random walk until they visit an unvisited vertex. 
Each particle waits at the source until the previous one hits the external boundary, before being emitted. 
Gravner-Quastel \cite{GQ} and Levine-Peres \cite{LP2} generalize \cite{LBG} and relate \abbr{IDLA} under more general, albeit still fixed, source locations to \abbr{PDE} free boundary problems (a Stefan problem in \cite{GQ}, and an obstacle problem in \cite{LP2} who also obtain analogous shape theorems for rotor-router and divisible sandpile models).
An interesting variant is the Uniform \abbr{IDLA},
where upon hitting the boundary, the particle (source) is moved 
at a point chosen at random uniformly in the domain, and 
it is shown in \cite{BDKL} that the limiting 
shape of Uniform \abbr{IDLA} is the Euclidean ball. 

Beyond these two examples, there is little understanding of such growth processes, 
despite substantial recent advances for first passage percolation.
In particular, it is conjectured that for both \abbr{ORRW} and \abbr{OERW}
the evolution leads to the formation of an asymptotic shape as time goes to 
infinity (see \cite{K1, K2}), but there is no clear vision on how to attack the problem. Recall that in \abbr{ORRW} the particle performs random walk on $\Z^n$, but each edge (or vertex) increases its conductance by a fixed strength $a>0$ after the first time it is traversed. A phase transition is 
expected in terms of $a$, with a limiting shape conjectured for 
all $a$ large enough. In the \abbr{OERW} model, the particle receives 
a (one-time) small drift towards the origin whenever it reaches 
an unvisited vertex (instead of the conductance change of the \abbr{ORRW}), 
and a shape theorem is conjectured to hold, 
no matter how small this positive drift is.
We refer the reader to \cite{BW, KZ} for background on excited random walks, and to \cite{BDKL,K2} for discussions on various \abbr{IDLA} type processes and 
reinforced walks, all of whom share certain similar features.
In particular, heuristically, whenever the self-interaction tends to 
attract the walker towards the bulk of its existing range, the boundary 
of the latter should change at a much slower rate than that 
of the walker, providing a natural setting to witness {\em averaging}. 

While non-lattice isotropic models are more amenable to rigorous 
analysis (see \cite{STV,NT}), this typically requires having random 
conformal maps, hence restricted to dimension $n=2$. 
By focusing instead on the evolutions of {\em star-shaped domains} in 
$\R^n$, 
we are able to handle any $n \ge 2$, and mention in passing that,
on the deterministic side, the works \cite{CM2,CM1}
are close in spirit to our averaged equation \eqref{b-bar}. 

We consider here a general random growth model 
in $\R^n$ which is specified by two {\em rules} $F$, $H$ and a 
scaling parameter $\ep>0$. The rule $F$, which is allowed to 
depend on the whole geometry of the domain and the position of the source,
determines the (random) point at the boundary where the particle, upon 
starting at the {\it prescribed position}, called {\it source}, is going 
to hit the boundary of the domain. For example, $F$ may be the  
Harmonic measure at the boundary of the domain from the source. After the 
particle hits the boundary, the domain grows around the hitting point 
with a volume increase of $\ep$, followed by the particle jumping to the next source position, according to the rule $H$.

More precisely, fixing a small parameter $\ep>0$, we consider evolving domains $(D_t^\ep)_{t\ge 0}$ in $\R^n, n\ge2$, which form simply-connected star-shaped compact sets (i.e. they can be parametrized by a function $R_t^\ep$ defined on the sphere $\bS^{n-1}$). It is a pure jump Markov process that starts with an initial domain $D^\ep_0\ni0$ and {\em particle position} $x^\ep_0$ and evolves at a Poisson rate of $\ep^{-1}$ by increasing the domain around randomly chosen boundary points (or equivalently, spherical angles $\xi_t\in \bS^{n-1}$). The probability density for choosing boundary points to evolve is given by the hitting kernel $F(R^\ep_{t^-},x^\ep_{t^-},\cdot)$, which is a probability density on the sphere $\bS^{n-1}$. After each hitting at the boundary at a point $\xi_t$, the particle is instantaneously transported according to the specified rule $H(R^\ep_{t^-}, \xi_t)$ to a point in $\R^n$ that can depend on both the domain and the last hitting position. Contrary to the rule $F$, we note that $H$ is  deterministic in this article. In principle, one can consider a more comprehensive model where $H$ is a probability density function on $\R^n$, representing random transportation of the particle after each hitting of the current domain boundary.

The process $(R^\ep_t,x^\ep_t)_{t\ge 0}$ of evolving domains in $\R^n$ together with the position of the driving particle coupled to the former is, by construction, Markov (though each marginal is in general non-Markovian). The aim is to construct a continuum simplified model of ``random walk interacting with its range", allowing for general hitting kernel and non-trivial redistribution after each interaction, 
while inferring whether the evolving domain has an asymptotic shape.

\begin{figure}[h]
\[
\begin{array}{ccc}
  \includegraphics[width=3cm]{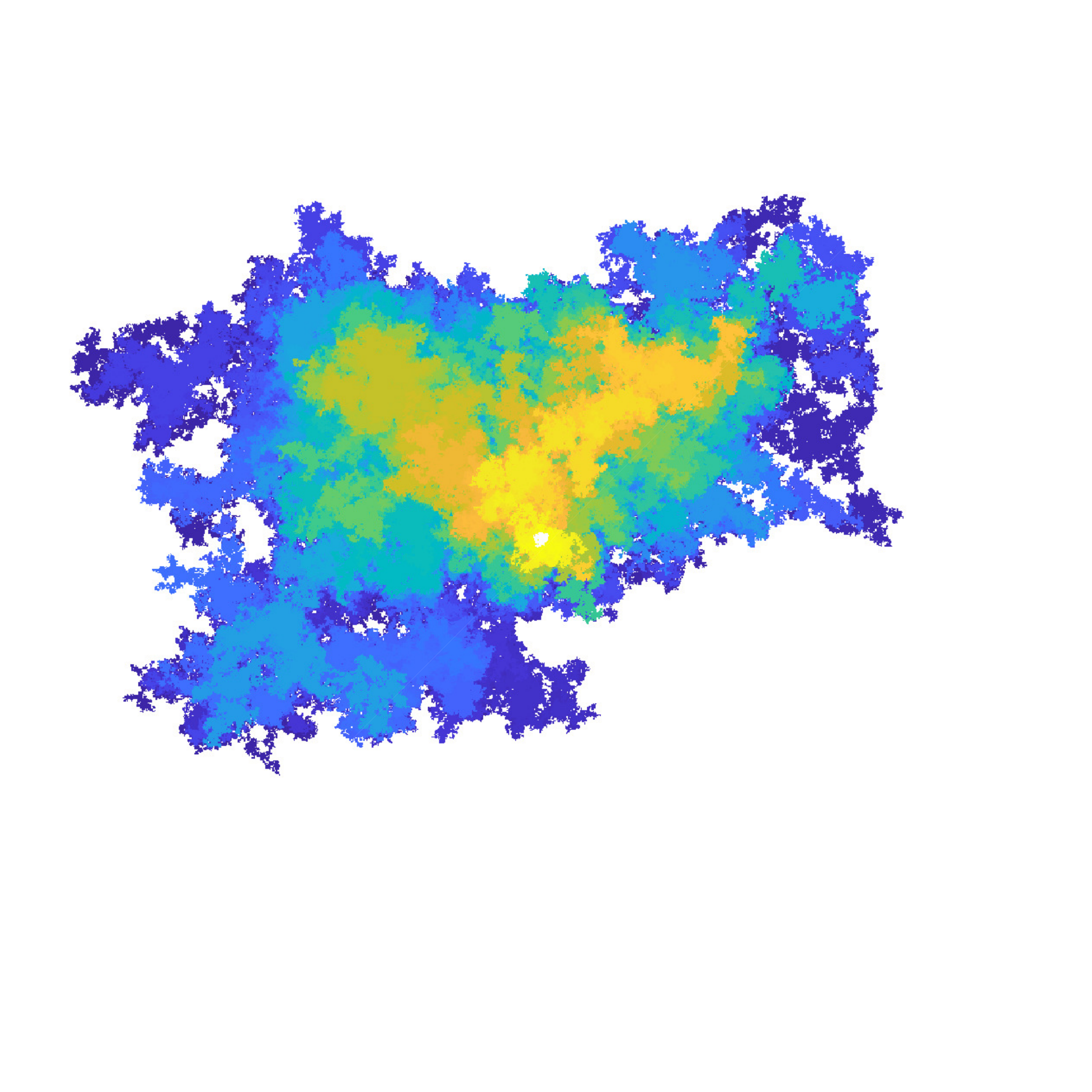}&  \includegraphics[width=3cm]{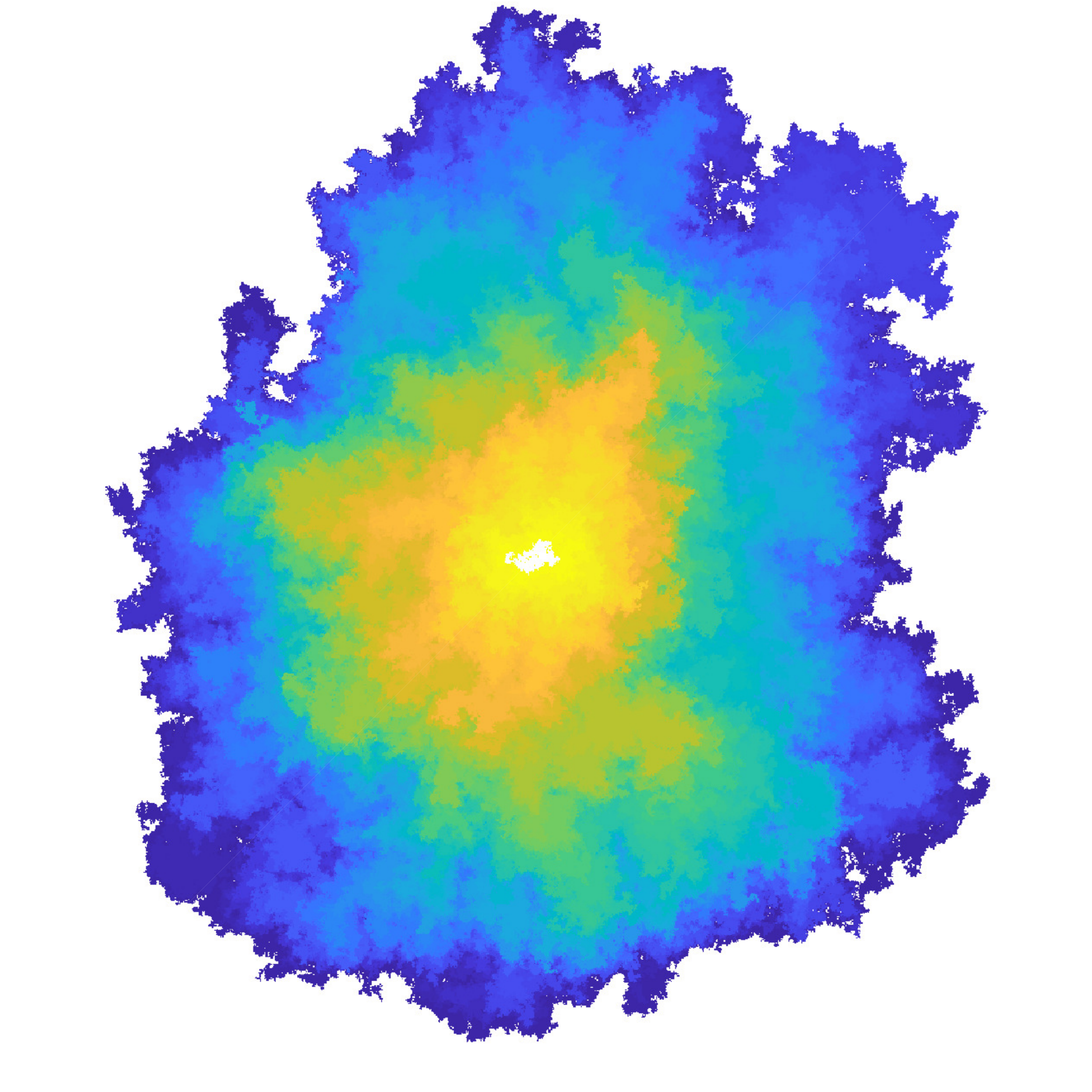}&  \includegraphics[width=3.5cm]{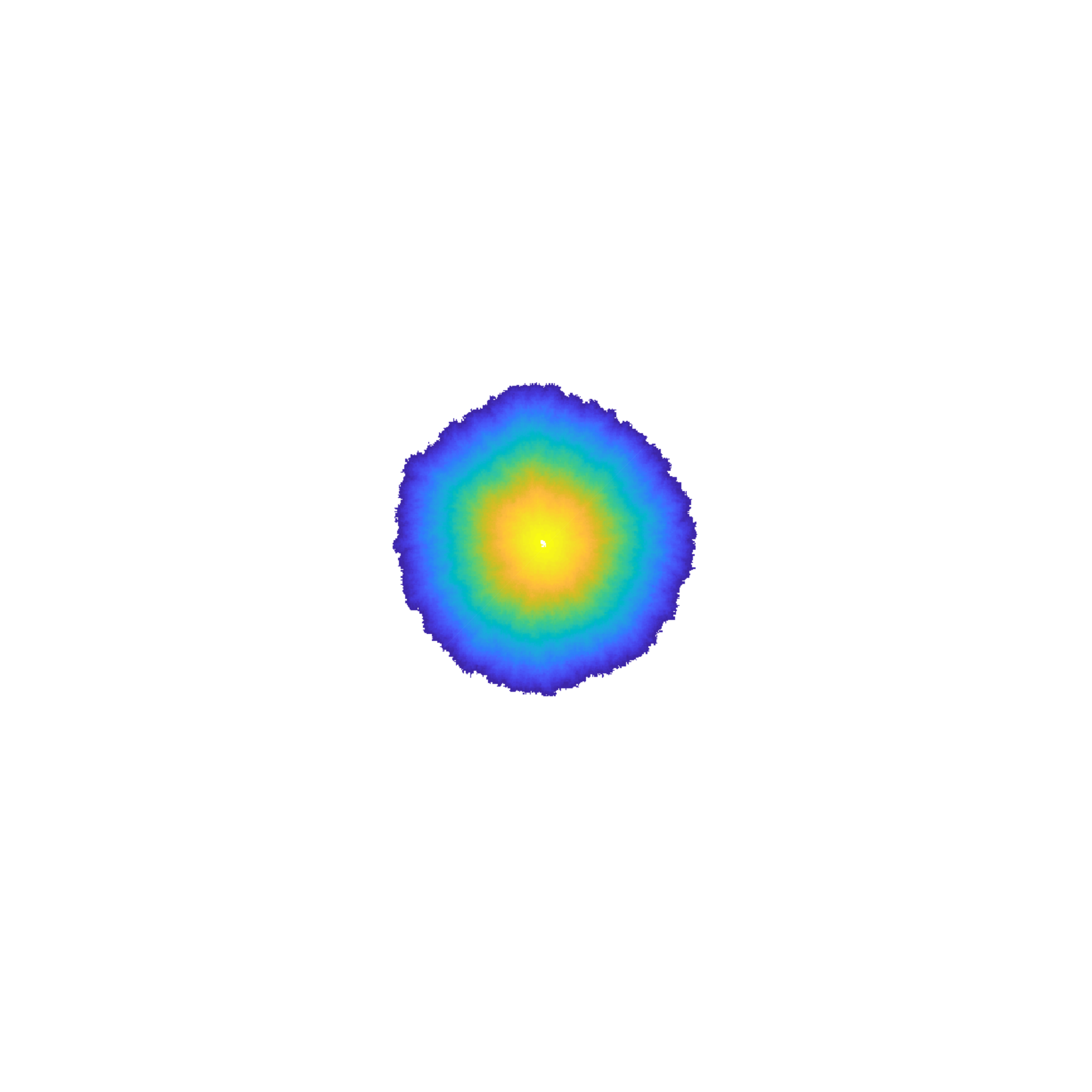}
  \end{array}
  \]
 \caption{Vertex once-reinforced random walk 
on $\Z^2$ with strength parameter $a=2$ (left), $a=3$ (middle) and $a=100$ (right) in a box of size 2000. The color of each vertex is proportional to the square root of its first visit time by the walk.}
\label{fig:1}
\end{figure}

\begin{figure}[h]
\[
 \begin{array}{ccc}
  \includegraphics[width=3cm]{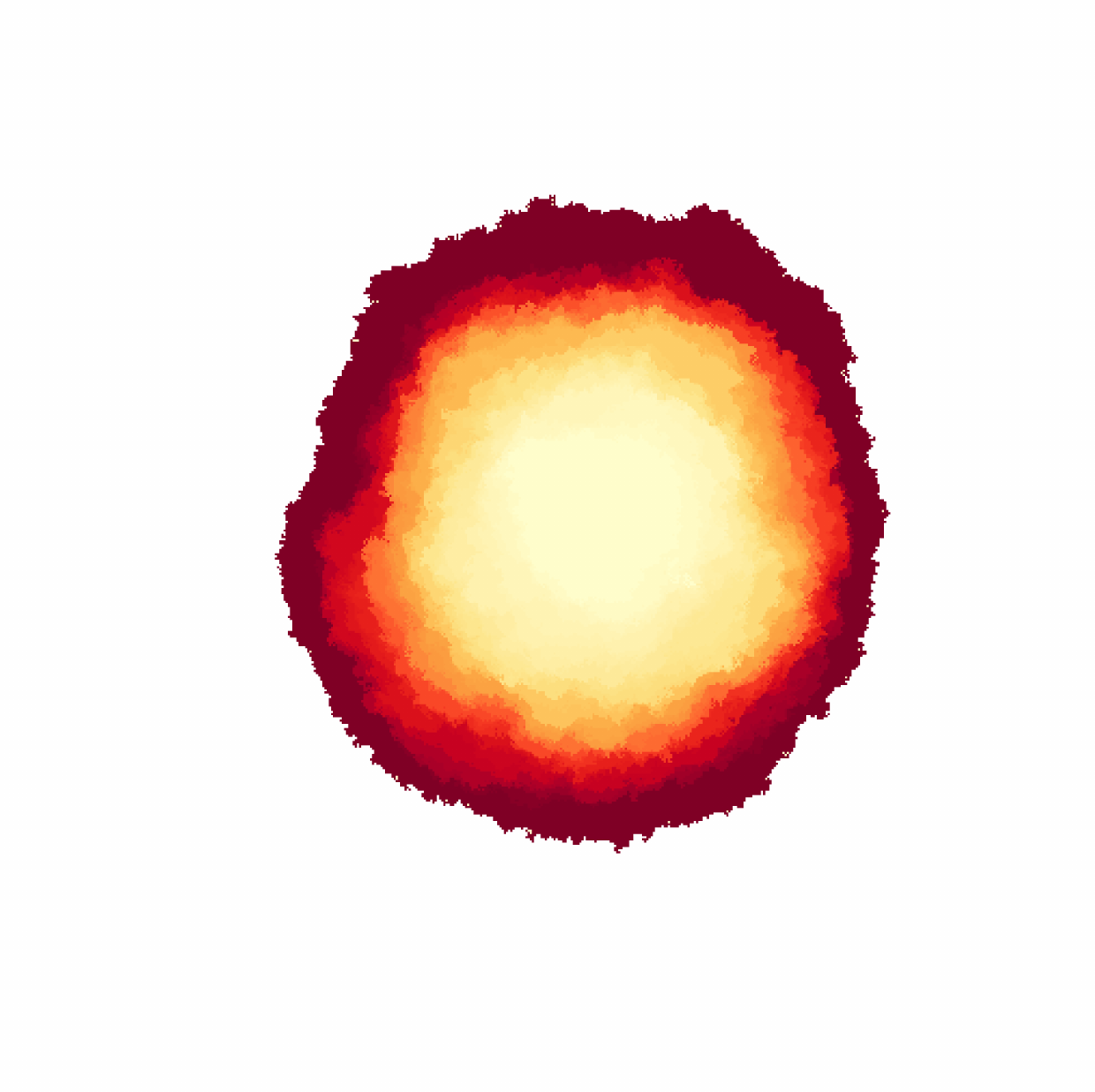}&  \includegraphics[width=3cm]{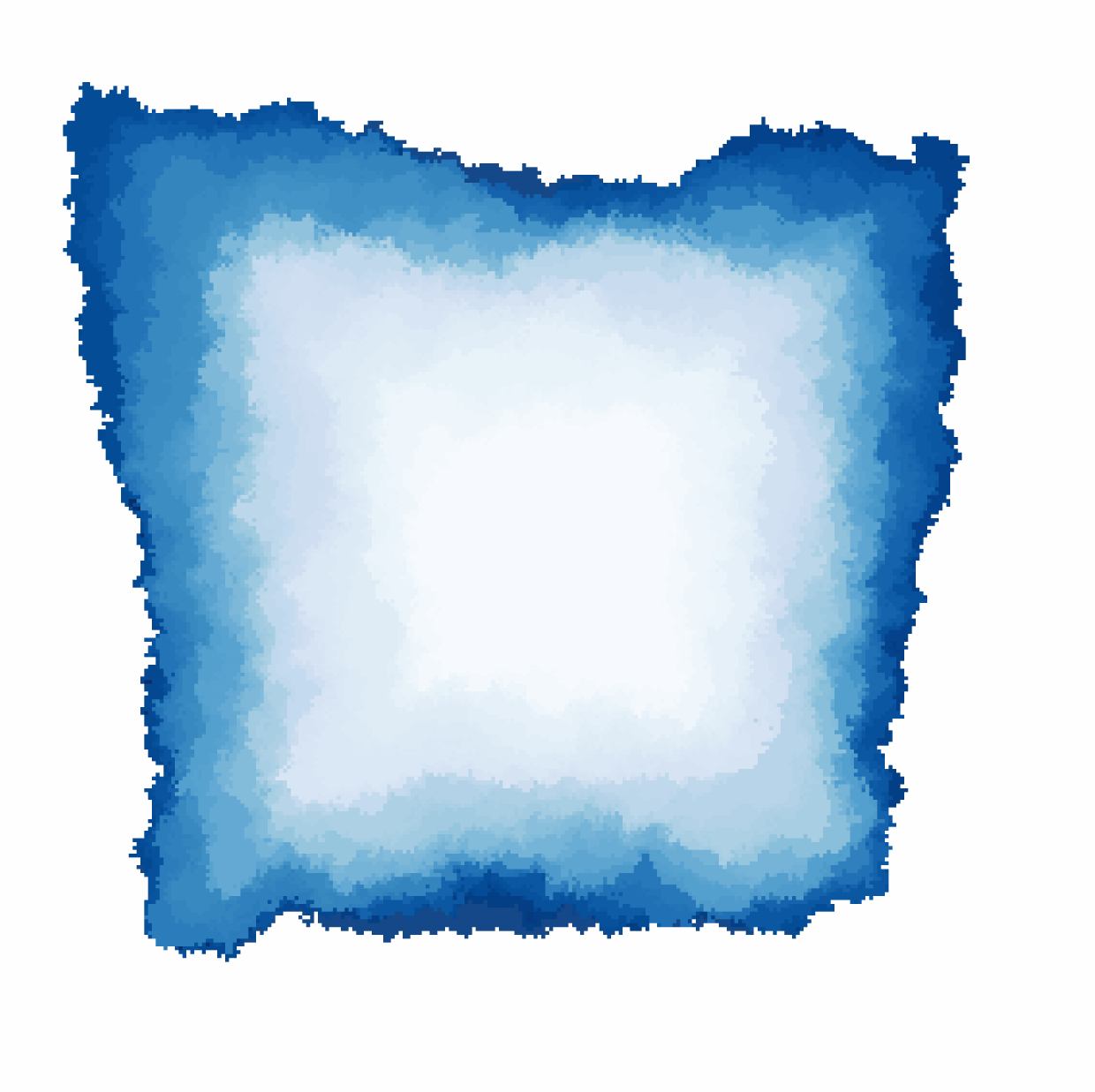}&  \includegraphics[width=3cm]{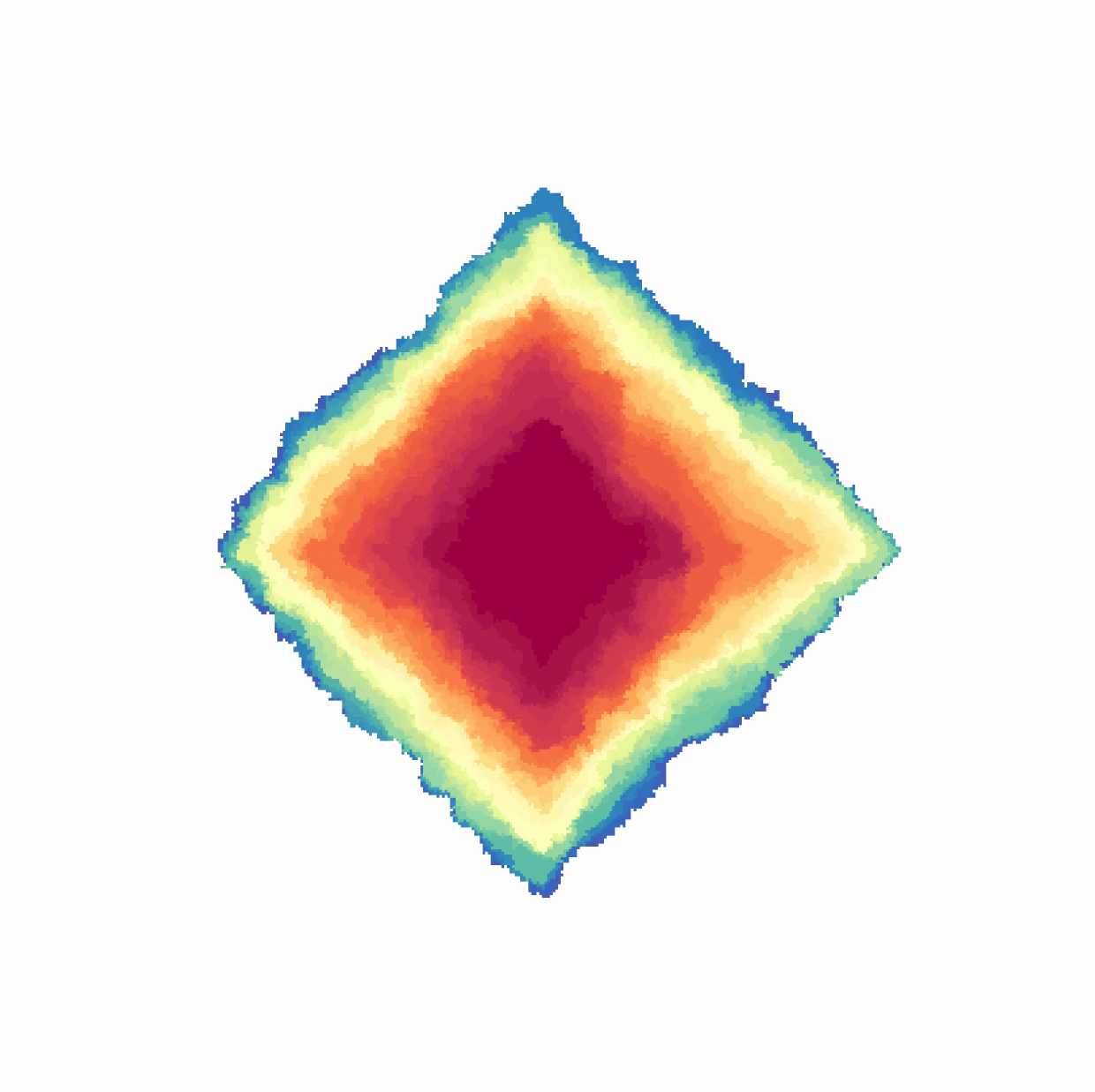}
 \end{array}
 \]
 \caption{
Origin-excited random walk on $\Z^2$ with three different excitation rules. Left: choose a coordinate with probability proportional to its absolute value and move one unit towards the origin in the chosen coordinate. Middle: move one unit towards the origin in the direction of the coordinate with the largest absolute value. Right: move one unit towards the origin in each coordinate. Each site is colored according to the first visit time.
}
  \label{fig:2}
\end{figure}

The {\it averaging principle} has been extensively studied in the theory of dynamical systems, see e.g. \cite{BM, Ce,FW, Kh, Ki, PSV, Ve} and references therein. Usually one identifies a {\em slow variable} and a {\em fast 
variable}. Under suitable conditions the fast variable achieves equilibrium in a time scale for which the slow variable does not evolve macroscopically. Hence, as the scale parameter $\ep \to 0$ one expects the slow variable to move according to a system where
the fast variable is integrated with respect to its invariant measure, which may depend on the slow variable as well. In our model, the  averaging property that one expects 
in models such as \abbr{ORRW} and \abbr{OERW} is explicitly shown in 
terms of the process $(R_t^\ep, x_t^\ep)$, where as $\ep \to 0$, the 
variable $R_t^\ep$ serves as the slow variable, while 
$x_t^\ep$ acts as the fast one
(and though the literature on averaging is large, we found no 
averaging principle that fits our case, involving a Markov jump process in infinite dimensions). The averaging principle is close in spirit to hydrodynamic limits, 
a standard tool in the study of interacting particle systems (see \cite{DP,KL,Va} and references therein). A hydrodynamic limit is proved  
for a continuous version of \abbr{IDLA} in \cite{GQ}, 
yielding in turn a shape theorem, thanks to 
the scale invariance of this model (as in Lemma \ref{lem:coup} below). 
As mentioned before, in this process particles are emitted from fixed sources.
One of our goals here is to derive similar results for
self-interacting random walks, where the source is clearly moving.

Under certain mild conditions on our model features (namely, the rules $F$ and $H$), 
we prove in Theorem \ref{thm1} an averaging principle. It
allows us to identify the limiting infinite-dimensional \abbr{ODE} governing the evolving domain as the slower dynamics of the pair, yielding  
in Theorem \ref{thm2} the limiting shape result as a
stationary solution of the limiting \abbr{ODE}. Then, in 
Theorem \ref{harm:measure} we verify our assumptions 
for a certain class of models, and
in some instances compute explicitly their limiting shape.

Let $\bS^{n-1}$ be the unit sphere in $\R^n, n\ge2$ equipped with its 
Euclidean surface area measure $\sigma(\cdot)$ and for 
any $1\le p\le\infty$ let $\|f\|_p$ denote the $L^p(\bS^{n-1})$ norm of $f$ with respect to $\sigma(\cdot)$. We denote by $C_+(\bS^{n-1})$ the 
space of {\it strictly positive} continuous functions on $\bS^{n-1}$.

\begin{defn}
A simply-connected compact set $D\subseteq\R^n$ is called {\it{star-shaped}} with respect to $0\in D$, if the line segment connecting $0$ and any $x\in \partial D$ is entirely contained in $D$. 
\end{defn}
\noindent Any star-shaped $D$ is uniquely represented by 
a non-negative function $r : \bS^{n-1} \to \R_+$ as
\begin{align*}
D=\big\{x\in \R^n:  x=\rho\theta, \,\theta\in \bS^{n-1}, \, 0\le \rho\le  r(\theta)\big\}\,.
\end{align*}
Hereafter, by a slight abuse of notation, we identify any
$r \in C_+(\bS^{n-1})$ with its graph, which encloses 
a star-shaped domain $D$ and denote by $Leb(r)$ the 
Lebesgue measure (or volume) of that domain $D$. Namely,
\begin{align*}
Leb(r)=n^{-1}\int_{\bS^{n-1}}r(\theta)^nd\sigma(\theta)=n^{-1}\|r\|_n^n \, .
\end{align*}
Let $\mathcal D(F)$ be an open subset of $C_+(\bS^{n-1}) \times \R^n$ 
such that 
$\{x : (r,x) \in \mathcal D(F)\}$ is non-empty
for any $r \in C_+(\bS^{n-1})$. The measurable map
\begin{align*}
F\colon \mathcal D(F) \subset C_+(\bS^{n-1})\times \R^n
\to C_+(\bS^{n-1}),
\end{align*}
assigns to each $(r,x) \in \mathcal D(F)$ a strictly positive, continuous
probability density function $F(r,x,\xi)$ with respect to $\sigma(\cdot)$. This function represents the rule whereby a particle starting from $x\in \R^n$ chooses a point 
$r(\xi)\xi$, $\xi \in \bS^{n-1}$ at the boundary of the domain 
enclosed by $r$, to be
the center of the (small) bump we add on the domain boundary $r$.
The measurable map
\begin{align*}
H(r,\xi):C_+(\bS^{n-1})\times \bS^{n-1}\to\R^n
\end{align*}
assigns for each $r\in C_+(\bS^{n-1})$ and $\xi \in \bS^{n-1}$ the 
transported (source) location $x=H(r,\xi)$ of a particle 
that hits the domain boundary $r$ at angle $\xi$. Assuming
that $(r',H(r,\xi)) \in \mathcal D(F)$
for any $r'\ge r$ and $F(r,x,\cdot) d\sigma$-a.e. $\xi$,
guarantees that a.s. the iterative composition of 
the rules $H$ and $F$ is well defined (per our dynamics \eqref{fast-component}). The small bump we add
is in the form of a suitable spherical approximate identity 
$g_\eta(\cdot)$, as defined next.

\begin{defn}\label{def:apx-iden}
A collection of continuous functions 
$g_\eta : [-1,1] \to \R_+$ is called a {\em spherical approximate identity} 
if $1 \star g_\eta = 1$,
$\|f \star g_\eta \|_2\le \|f\|_2$ and $\|f \star g_\eta -f\|_2\to 0$ as $\eta\to 0$, for every $f\in L^2(\bS^{n-1})$, where (see \cite[(2.1.1)]{DX}),
\begin{align}\label{integ-1}
(f \star g_\eta)(z):=\frac{1}{\omega_n}\int_{\bS^{n-1}}f(\theta) 
g_\eta(\langle z,\theta\rangle) d\sigma(\theta), \quad z\in \bS^{n-1},
\end{align} 
$\omega_n=\sigma(\bS^{n-1})=\frac{(2\pi)^{n/2}}{\Gamma(n/2)}$ 
is the surface area of $\bS^{n-1}$, 
and $\langle z,\theta\rangle$ denotes the scalar product 
associated with the Euclidean norm $|\cdot|$ in $\R^n$. We call such collection \emph{local} 
if in addition
$g_\eta(\langle z, \cdot \rangle)$ 
is supported on the spherical cap of (Euclidean) radius $2\eta$ centered 
at $z$ and $\eta^{n-1} \|g_\eta\|_\infty$ are uniformly 
bounded.
\end{defn}

Utilizing \cite[Section 2.1]{DX} we characterize in 
Lemma \ref{lem:apx-iden} the collections $g_\eta$ 
that form a local spherical approximate identity
(see also Figure \ref{fig:4}).
\begin{figure}[h]
\[
 \begin{array}{cc}
  \includegraphics[width=4cm]{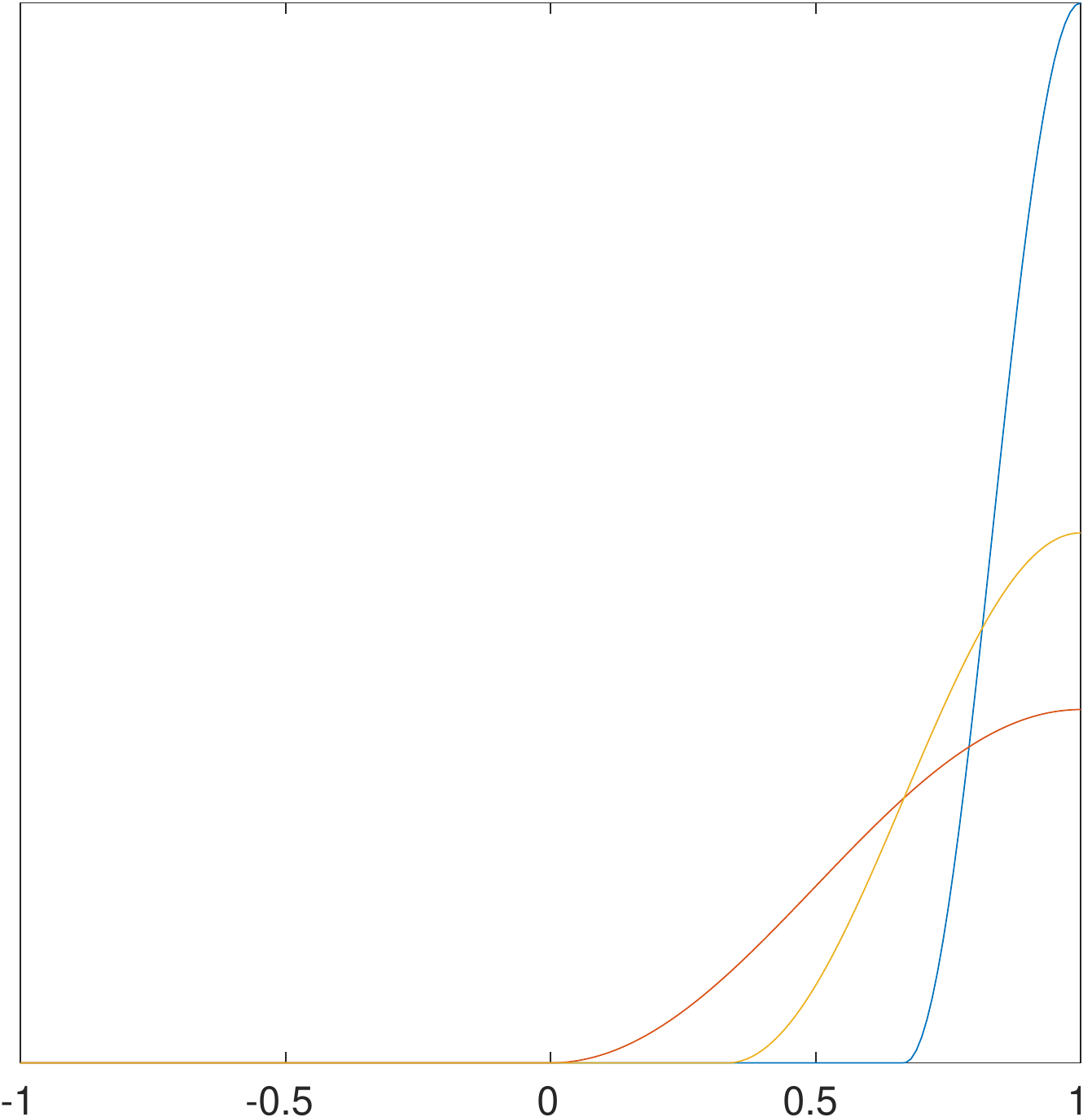}&  \includegraphics[width=4.5cm]{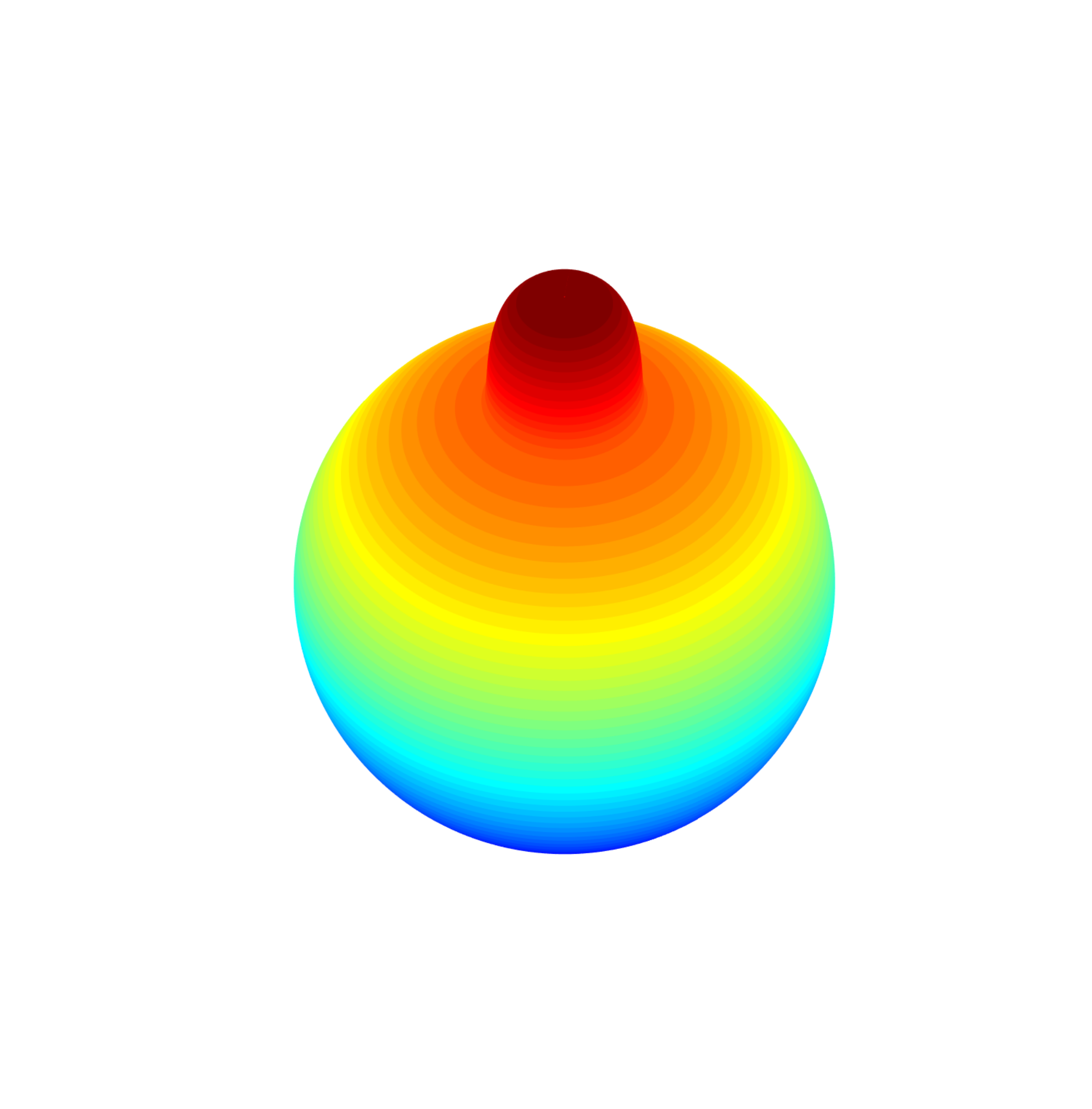}
  \end{array}
 \]
 \caption{Left: Functions $g_\eta$ for different values of $\eta$. Right: Function $g_\eta(\langle z,\cdot \rangle)$ defined on the sphere $\mathbb S^2$ with $z=(0,0,1)$.}
  \label{fig:4}
\end{figure}

Throughout we set the positive function on $\mathcal{D}(F)$,
\begin{align}\label{def:y}
y_{r,x}:= \omega_n \int_{\bS^{n-1}}r(\theta)^{n-1}F(r, x,\theta)d\sigma(\theta) \,.
\end{align}
Noting that for $\xi$ of density $F(r,x,\cdot) d\sigma$ 
\begin{equation}\label{eq:Leb-add}
\lim_{\ep,\eta \to 0} 
\ep^{-1} \E \big[Leb(r+\ep g_\eta(\langle \xi,\cdot\rangle))-Leb(r)) \big]
= y_{r,x} 
\end{equation}
(c.f. proof of Proposition \ref{vol-growth}),
we add at each update a bump 
$(\ep/y_{r,x}) \, g_\eta(\langle \xi,\cdot \rangle)$ 
on the current boundary $r$,
where $\{g_\eta\}$ is a \emph{local} spherical approximate identity,
so that for $\ep \ll 1$, the volume 
of $D_t^\ep$ should grow at a nearly constant, unit rate. 
Using the $\ep$-dependent 
\begin{align}\label{eta}
\eta(\ep,r,x):= \ep^{{1}/{n}} \, y_{r,x}^{-1/{(n-1)}} \,,
\end{align}
as our spherical-scale parameter yields 
a bump 
$\ep y^{-1}_{r,x}g_\eta$ on the boundary $r$ 
of about $\ep^{1/n}$ height (in the radial direction), 
uniformly in $(r,x)$,
which is supported 
in case of a Euclidean ball of unit surface area
(namely, $r \equiv \omega_n^{-1/(n-1)}$),
on spherical caps of radius $2 \ep^{1/n}$. Clearly, when 
adding such $\ep$-dependent bumps to our boundary function, 
the star-shaped domain evolves by a localized bump 
and the new domain remains star-shaped.
Specifically, fixing $\ep\in(0,1]$ and starting at some 
$(R_0^\ep,x_0^\ep)$ we construct the 
Markov jump process $(R^\ep_t,x^\ep_t)_{t\ge 0}$ of jump rate $\ep^{-1}$ 
and state space $C_+(\bS^{n-1})\times  \R^n$, as follows. 
For a sequence $\{T^\ep_i\}_{i\in\N}$ of auxiliary Poisson arrival 
times of rate $\ep^{-1}$, starting at $T_0^\ep=0$, we freeze
$(R^\ep_t,x^\ep_t)$ during each of the intervals $[T^\ep_i,T^\ep_{i+1})$,
while as each $t=T^\ep_i$, $i \ge 1$, 
conditional on the canonical filtration 
\begin{align*}
\cF_{t^-}:=\sigma\{R^\ep_s, x^\ep_s, \xi_s: \, s\le t^-\} \,,
\end{align*}
let 
\begin{equation}
\xi_t \overset{d}{\sim} F( R_{t^-}^\ep, x^\ep_{t^-},\cdot) \,.
\label{bump-loc}
\end{equation}
That is, given $\mathcal F_{t^-}$ the random
$\xi_t \in \bS^{n-1}$ has the density $F( R_{t^-}^\ep, x^\ep_{t^-},\cdot)$ 
with respect to $\sigma(\cdot)$.
Then, 
update $(R_{t^-}^\ep,x^\ep_{t^-})$ according to 
\begin{align} 
 R_t^\ep(\theta)&= R_{t^-}^\ep(\theta)+ \frac{\ep }{y_{ R_{t^-}^\ep, x^\ep_{t^-}}}g_{\eta(\ep,R^\ep_{t^-},x^\ep_{t^-})} (\langle\xi_t, \theta \rangle), \quad \theta\in \bS^{n-1},            \label{slow-component}\\
x^\ep_t&=H(R^\ep_{t^-},\xi_t)   \label{fast-component}
\end{align} 
(recall the definitions \eqref{def:y} of $y_{r,x}$ and \eqref{eta} of 
$\eta(\ep,r,x)$).
\begin{figure}[h]
\[
 \begin{array}{ccc}
  \includegraphics[width=2cm]{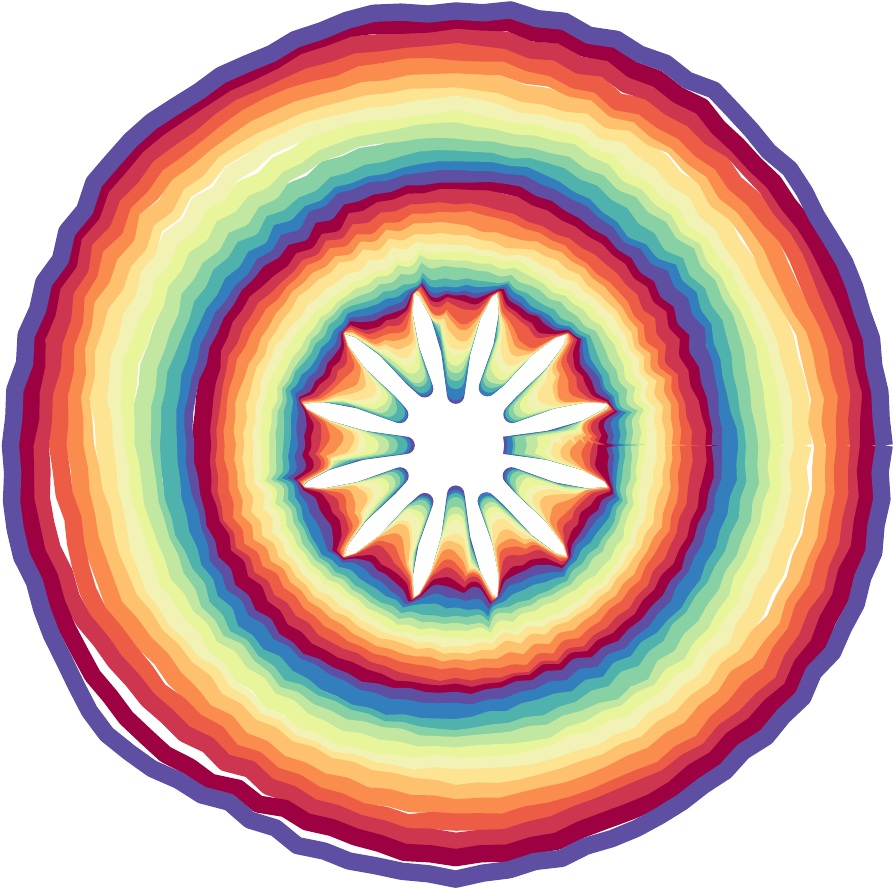}&  \includegraphics[width=2cm]{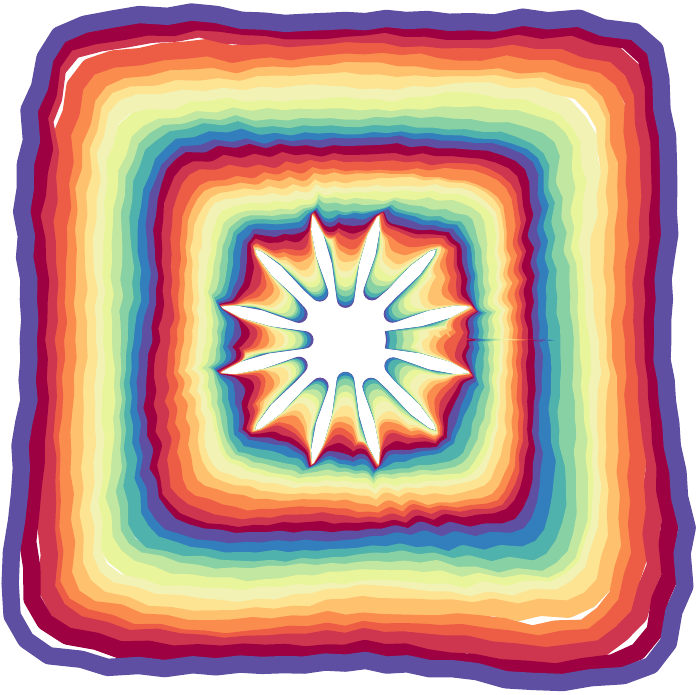}&  \includegraphics[width=2cm]{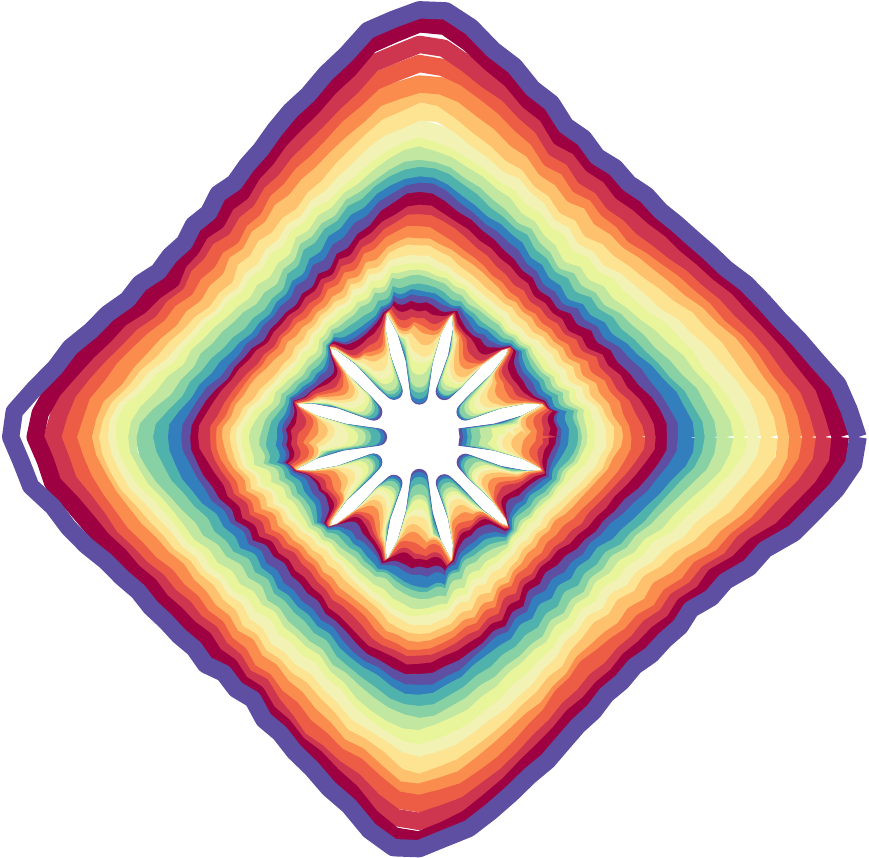}\\
  \includegraphics[width=3cm]{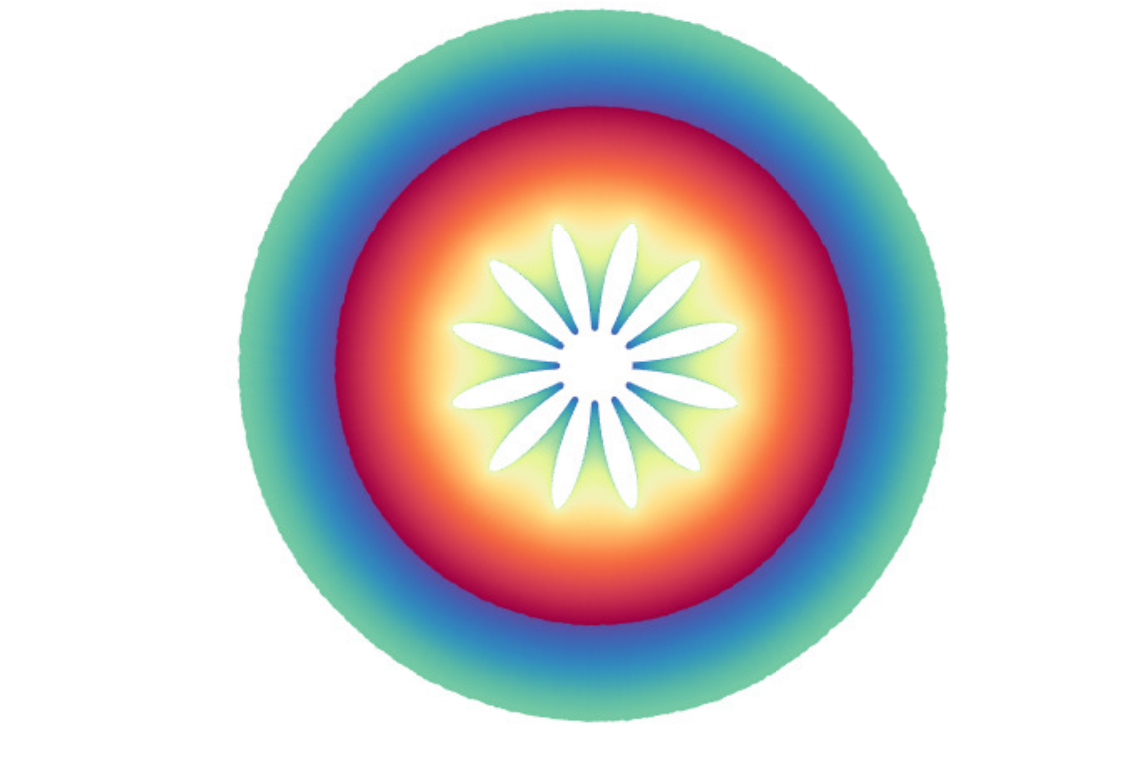}&  \includegraphics[width=3cm]{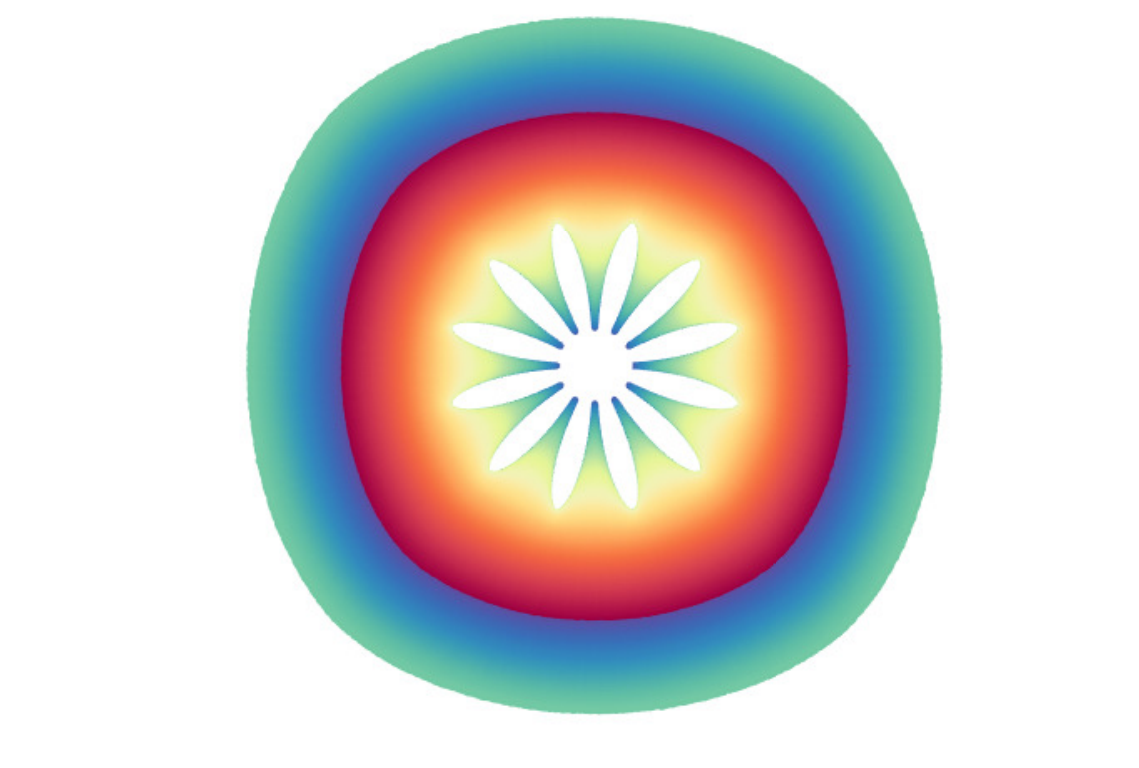}&  \includegraphics[width=3cm]{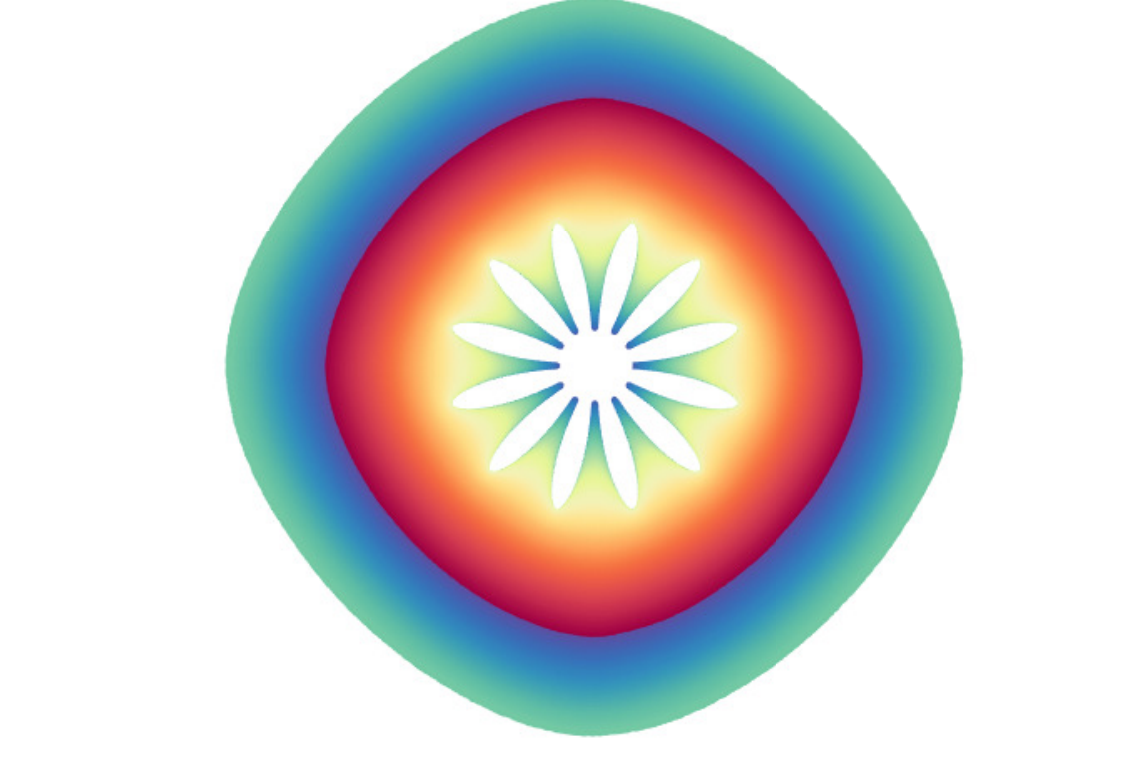}\\
  \includegraphics[width=3cm]{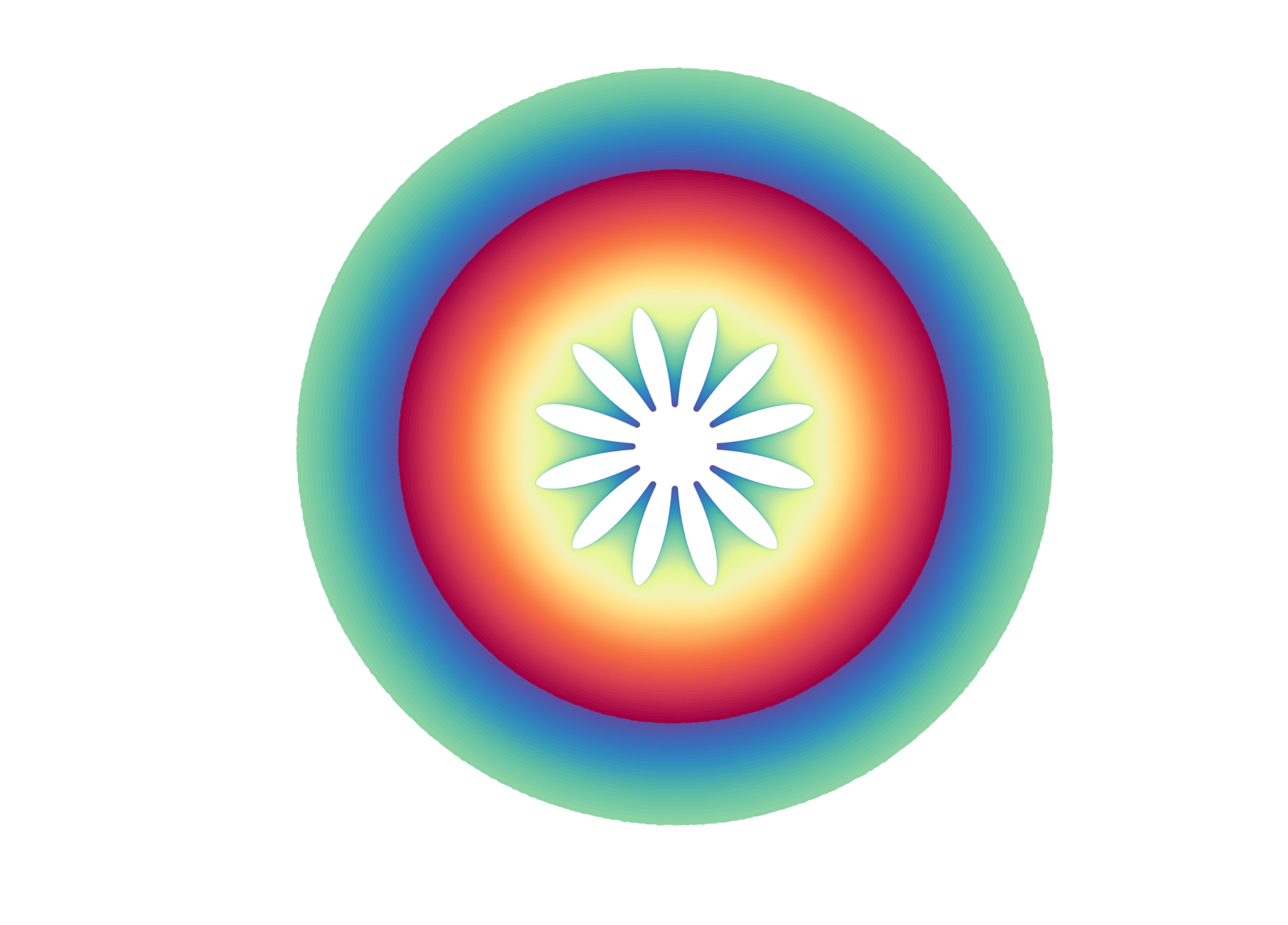}&  \includegraphics[width=3cm]{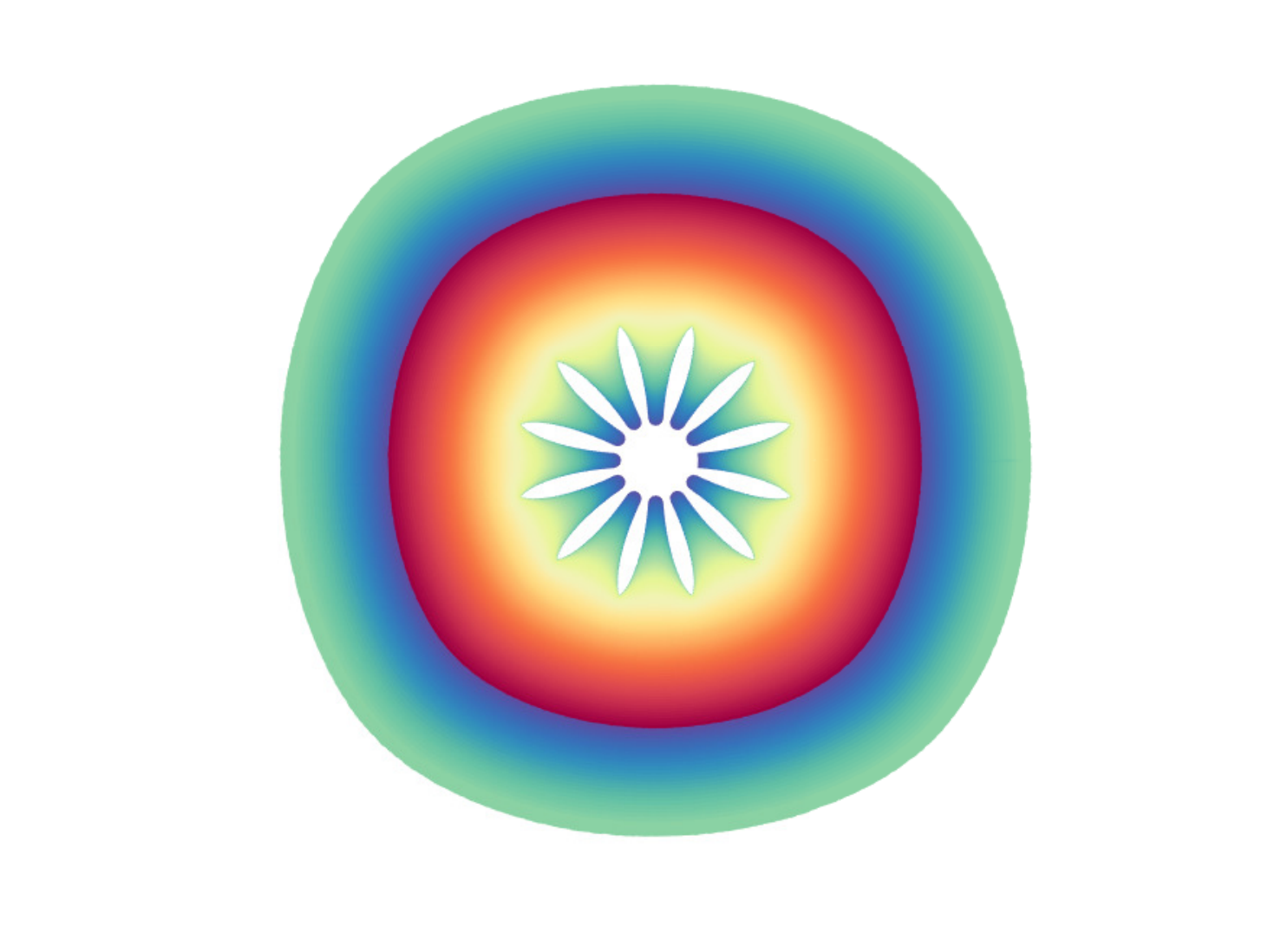}&  \includegraphics[width=3cm]{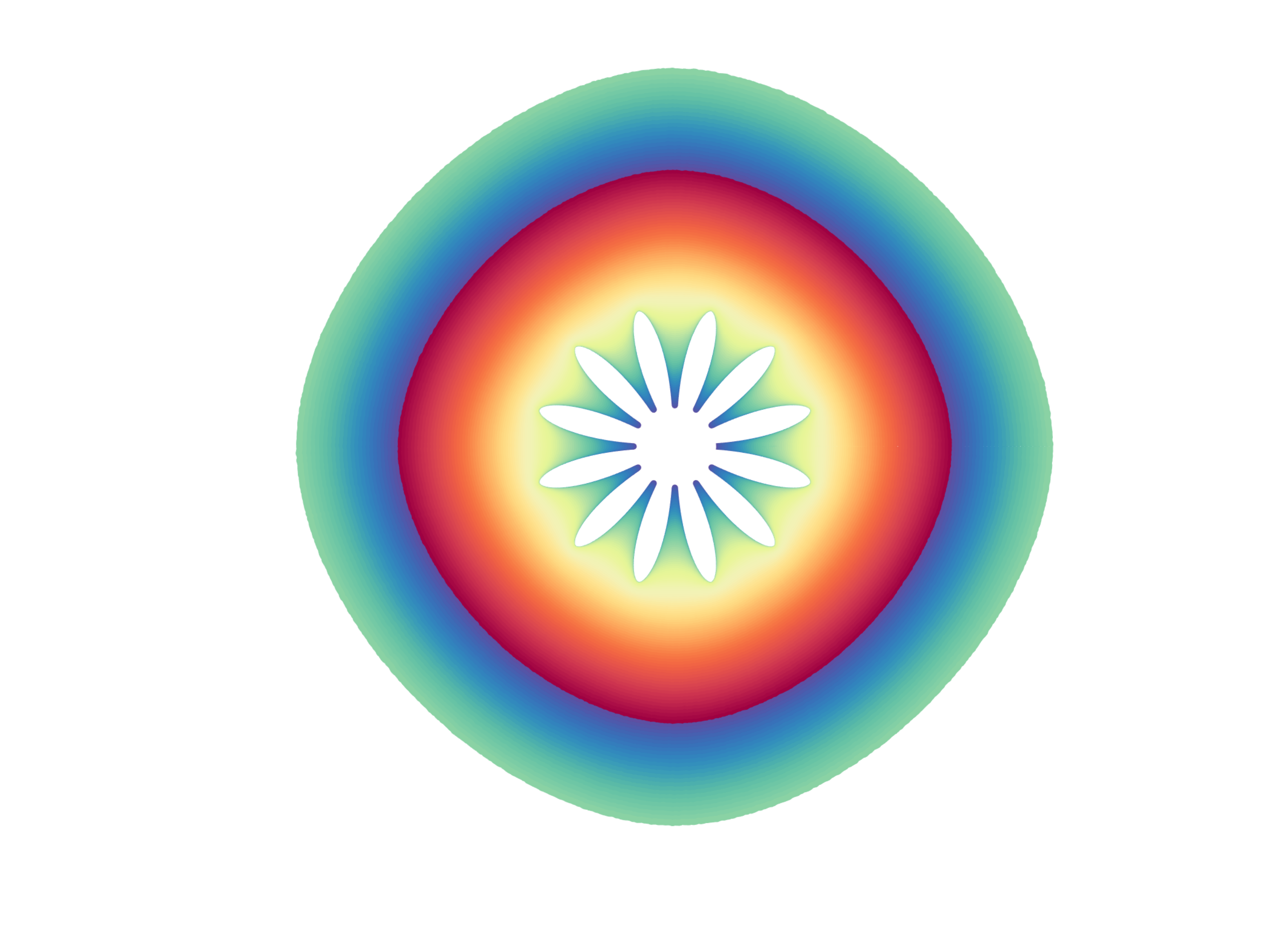}\\
 \end{array}
 \]
 \caption{The shape process \eqref{slow-component} on $\R^2$ with $F(r,x,\cdot)$ given by the harmonic measure on $r$ from $x$ and different transportation rules $H$: {\bf First row} $\ep=10^{-4}$. {\bf Left}: $H(r,\xi)= (r(\xi)-1)_+\xi$. {\bf Middle}: $H(r,\xi)=$ move one unit towards the origin from $r(\xi)\xi$ in the coordinate with larger absolute value. {\bf Right}: $H(r,\xi)=$ move one unit towards the origin in both coordinates. {\bf Second row} $\ep=10^{-6}$. {\bf Left}: $H(r,\xi)= (r(\xi)-1)_+\xi$. {\bf Middle}: $H(r,\xi)= \big(r(\xi)-\frac{|\xi|_\infty}{|\xi|_2}\big)_+\xi$. {\bf Right}: $H(r,\xi)= \big(r(\xi)-\frac{|\xi|_1}{|\xi|_2}\big)_+\xi$. {\bf Third row},  $\ep=10^{-6}$. {\bf Left}: $H(r,\xi)= .9r(\xi)\xi$. {\bf Middle}: $H(r,\xi)= \big(1-\frac{|\xi|_\infty}{10|\xi|_2}\big)r(\xi)\xi$. {\bf Right}: $H(r,\xi)= \big(1-\frac{|\xi|_1}{10|\xi|_2}\big)r(\xi)\xi$. The processes start from the white-colored region in the shape of a sunflower. Different colors represent different time-epochs (proportional to $t^2$). The linear-in-time evolution of these snapshots identifies the asymptotic $O(\sqrt{t})$ for the diameter growth. As time, hence diameter, increases, 
the processes start to ``feel'' the different drifts, tending to different asymptotic shapes: sphere, square or diamond depending on the choice of $H$ (similarly to what we saw for different excitation rules in Figure \ref{fig:2}). The final time is 16. For a dynamic evolution of one instance of the process, the interested reader may consult \url{https://youtu.be/DfsYWZnQA9Q}.}
\label{fig.shapes.process}
\end{figure}
The generator $\cL^\ep$ of the Markov process $(R^\ep_t,x^\ep_t)_{t\ge0}$ is
\begin{equation}\label{generator}
\begin{aligned}
&\big(\cL^\ep \mathsf{f}\big)( r, x)\\
&:=\frac{1}{\ep}\Big[\int_{\bS^{n-1}} \mathsf{f}\Big(r+\ep y^{-1}_{r, x}g_{\eta(\ep,r,x)}(\langle\xi,\cdot\rangle), H(r,\xi)\Big)F(r,x,\xi)d\sigma(\xi)-\mathsf{f}( r, x)\Big],
\end{aligned}
\end{equation}
for any $\mathsf{f}:C_+(\bS^{n-1})\times \R^n\to\R$ in the domain of $\cL^\ep$. 
For $(r,x) \in \mathcal D(F)$ and $\theta\in \bS^{n-1}$ let
\begin{align*}
&b(r,x)(\theta):=\frac{\omega_n}{y_{r,x}} F(r,x, \theta) \,, 
\qquad  
b^\ep(r,x) := b(r,x) \star g_{\eta(\ep,r,x)}\,,\\
&h(r,x):=\int_{\bS^{n-1}} H(r,\xi) F(r,x, \xi)d\sigma(\xi) - x \,.
\end{align*}
Considering \eqref{generator} for
$\mathsf{f}( r, x)= r(\theta)$  the evaluation map at fixed $\theta\in \bS^{n-1}$ and using \eqref{integ-1}, we get for $( R^\ep_t(\theta))_{t\ge0}$
the decomposition 
\begin{align}\label{real}
 R^\ep_t(\theta)&= R^\ep_0(\theta)+\int_0^t b^\ep ( R_s^\ep,x_s^\ep)(\theta)ds+\Sigma^\ep_t(\theta) , \qquad \theta\in \bS^{n-1} \,,  
\end{align}
where $\Sigma_t^\ep(\theta)$ is an $\cF_t$-martingale.
Similarly, taking $\mathsf{f}( r, x)= x\cdot \vec{e}_i$, $i=1,\ldots,n$, 
in \eqref{generator} yields   
\begin{align}
x^\ep_t&= x^\ep_0+\int_0^t\ep^{-1}h( R_s^\ep, x_s^\ep)ds+ M^\ep_t \,,
\label{bump}
\end{align}
for some $\R^n$-valued, $\cF_t$-martingale $M_t^\ep$. For
$r \in C_+(\bS^{n-1})$ let $(x^{\ep,r}_t)_{t \ge 0}$ 
denote the $\R^n$-valued Markov jump process 
evolving by \eqref{bump} in the frozen domain 
$R^\ep_s\equiv r$. Its generator is thus
\begin{align}\label{frozen-dyn}
(\cL^{\ep,r} \mathsf{f}) (x) & =  \ep^{-1}\Big[\int_{\bS^{n-1}} \mathsf{f}(H(r,\xi))F(r, x, \xi)d\sigma(\xi) -  \mathsf{f}(x)\Big]
\end{align}
for a suitable collection of functions $\mathsf{f}:\R^n\to\R$.
Consider also the deterministic dynamics 
$t \mapsto  r_t  \in C_+(\bS^{n-1})$ given 
by
\begin{equation}\label{b-bar}
\begin{aligned}
 r_t (\theta)&= r_0 (\theta)+\int_0^t\ovl{b}( r_s )(\theta)ds\,, \\
\quad \ovl{b}(r)(\theta) &:=\int_{\R^n}b(r,x)(\theta)d\nu_r(x)\,,
\qquad \theta\in \bS^{n-1} \,.   
\end{aligned}
\end{equation}
The probability measures $\nu_r$ on $\R^n$ for $r \in C_+(\bS^{n-1})$ 
will be specified in Assumption (E), with Proposition \ref{ppn:implic}
establishing the existence and uniqueness of the solution for the 
infinite-dimensional \abbr{ODE} \eqref{b-bar}.
For every $a\in(0,1)$ and $p \ge 1$, we define the collections
\begin{align*}
\cA_p (a) := \big\{ r \in C_+(\bS^{n-1}) \, : \; & \inf_\theta \{ r(\theta) \} \ge a, \, \|r\|_p\le a^{-1} \big\}
\,, \nonumber \\ 
\cAe_p(a) := \big\{(r,x)\in\cD(F): & \; r \in \cA_p(a) \,, \\
& x\in \text{Image}(H(r',\cdot)) \text{ for some } r'\in\cA_p(a), r'\le r \big\}
\end{align*} 
and assume the following Lipschitz properties of $F$, $H$ and $\ovl{b}$ 
throughout $\cAe_{\infty}(a)$.
\begin{assumption}[L]
For any $a\in(0,1)$, there exists $K=K(a)$ finite such that 
uniformly for $(r,x), (r',x')\in \cAe_\infty(a)$, 
$z,z'\in \bS^{n-1}$, we have that
\begin{align}
\|F(r, x,\cdot)-F(r', x',\cdot)\|_p&\le K\big(\|r-r'\|_p+|x-x'|\big),  \quad p=2, \label{lip-F}\\
|H(r, z)-H(r', z')|&\le K\big(\|r-r'\|_2+|z-z'|\big), \label{lip-H}  \\
\|\ovl{b}(r)-\ovl{b}(r')\|_p&\le K\|r-r'\|_p\,, \quad \mbox{ both } p=2 \text{ and } p=\infty.\label{lip-bar}
\end{align}
Moreover, for any $a \in (0,1)$,  
\begin{equation}\label{F-ubd}
\ovl F(a):= \sup \{  F(r,x,\theta) : (r,x) \in \cAe_2(a), \theta \in \bS^{n-1} \} < \infty \,.
\end{equation}
\end{assumption}
Our other two assumptions concern the ergodicity of the particle process in a frozen domain and
the convergence to $b(\cdot,\cdot)$ of the drift of $R^\ep_t$ when $\ep \to 0$.
\begin{assumption}[E] 
For any $r \in C_+(\bS^{n-1})$ 
the process $(x_t^{1,r})_{t\ge0}$ of generator \eqref{frozen-dyn} has a unique invariant probability measure $\nu_r$, such that  
\begin{align}\label{rate}
\sup_{(r,x) \in \cAe_\infty(a)} \sup_{t_0\ge 0}
\E_x\left[ \Big\|\frac{1}{t}\int_{t_0}^{t_0+t}[b(r,x^{1,r}_s)-\ovl{b}(r)]ds  \Big\|_2^2 \right]\le \lambda(t,a) \,,
\end{align}
where $\lambda(t,a)\to 0$ as $t\to\infty$, for any fixed $a \in (0,1)$.
\end{assumption}
\begin{assumption}[C]
For the $\cF_t$-stopping times 
\begin{align}\label{zeta}
\zeta^\ep(\delta)  :=\inf\big\{t>0:\, \|R^\ep_t \|_\infty> \delta^{-1} \big\}\,,
\end{align}
any fixed $t \ge 0$ and $\delta>0$,
 \begin{equation}
\lim_{\ep\to 0}  \|b^\ep(R^\ep_{t\wedge \zeta^\ep(\delta)}, x_{t\wedge \zeta^\ep(\delta)}^\ep) - 
b(R^\ep_{t\wedge \zeta^\ep(\delta)}, x_{t\wedge \zeta^\ep(\delta)}^\ep) \|_2 = 0, 
\;\; \text{in probability\,.}
 \end{equation}
\end{assumption}

\begin{rmk}\label{rem:assumpC} 
From Definition \ref{def:apx-iden} we know that 
$\|b^\ep(r, x) - b(r, x) \|_2 \to 0$ as $\ep \to 0$, 
for any fixed $(r,x) \in \cD(F)$. For Assumption (C) we
need this to hold at the $\ep$-dependent 
$(R^\ep_{t\wedge \zeta^\ep},x^\ep_{t \wedge \zeta^\ep})$, 
but see parts (c) and (d) of Proposition \ref{ppn:implic} for simple sufficient
conditions for Assumptions (E) and (C), respectively.
\end{rmk}

Equipped with these assumptions, we next state our main result. For technical reasons, we need to introduce a stopping time $\sigma^\ep(\delta)$ giving a lower bound to $F$, as needed to apply the crucial Lemma \ref{H_-1}.
\begin{thm}[Averaging principle]\label{thm1}
Under Assumptions {\rm (L)}, {\rm (E)} and {\rm (C)}, starting at
$R^\ep_0= r_0 \in C_+(\bS^{n-1})$, for the $\cF_t$-stopping time 
\begin{align}\label{pdf-lbd}
& \sigma^\ep(\delta)  :=\inf\big\{t\ge0: \min_\theta \{F(R_t^\ep,x^\ep_t,\theta)\}<\delta\big\},
\end{align}
and any $T<\infty$, $\delta>0$, we have that for any $\iota>0$
\begin{align}\label{hydrodyn}
\lim_{\ep\to0} \P\Big[\sup_{0\le t\le T\wedge \sigma^\ep(\delta)}\|R_t^\ep- r_t \|_2>\iota\Big]=0\,
\end{align}
where $\{r_t\}_{t\ge 0}$ denotes the unique $C_+(\bS^{n-1})$-solution of the \abbr{ode} \eqref{b-bar} starting at $r_0$ (see Proposition \ref{ppn:implic}(b)).
\end{thm}

\begin{rmk}\label{st-remov}
With minor modifications of the proof, we can  
accommodate in Theorem \ref{thm1} any random initial data such that 
$R_0^\epsilon \to  r_0$ in probability. 
It is crucial to have $ r_0$ strictly positive, 
since the function $b(r,x)$ blows up when $y_{r,x} \to 0$, 
hence \eqref{lip-bar} fails near $r \equiv 0$. Of course, if for any $a \in (0,1)$,  
\begin{equation}\label{F-lbd}
\udl F(a):= \inf \{  F(r,x,\theta) : \; (r,x) \in \cAe_2(a), \; \theta \in \bS^{n-1} \} > 0 \,,
\end{equation}
then we can dispense with the stopping time $\sigma^\ep(\delta)$ in \eqref{hydrodyn}.
\end{rmk}

The next proposition, whose proof is deferred to the appendix,
clarifies the implications of our assumptions.

\begin{ppn}\label{ppn:implic}$~\,$\newline
(a) If condition \eqref{lip-F} of Assumption (L) holds, then
for every $a\in(0,1)$ there exists $C=C(a,K)=C(a)<\infty$ such that for all 
$(r,x),(r',x')\in\cAe_\infty(a)$, 
\begin{equation}\label{lip-b} 
\|b(r, x)-b(r',x')\|_p \le C\big(\|r-r'\|_p+|x-x'|\big)\,, \quad p=2\,.  
\end{equation}
Similarly, if \eqref{lip-F} holds with $p=\infty$, then the same applies for \eqref{lip-b}.

Further, conditions \eqref{lip-F} and \eqref{lip-H} imply that
for all $(r,x),(r',x')\in \cAe_\infty (a)$,
\begin{equation}\label{lip-h}
|h(r,x)-h(r',x')| \le C(\|r-r'\|_2+|x-x'|)\,,
\end{equation}
while if \eqref{lip-F} holds for both $p=2$ and $p=\infty$, with \eqref{lip-H}, \eqref{F-ubd} and
\eqref{F-lbd} holding as well, then \eqref{lip-bar} must also hold.\\
(b) Condition \eqref{lip-bar} at $p=\infty$, together with Condition \eqref{F-ubd} imply that starting at 
any $r_0 \in C_+(\bS^{n-1})$  
the \abbr{ODE} \eqref{b-bar} admits a unique $C_+(\bS^{n-1})$-solution 
on $[0,\infty)$. Further, up 
to time $T$ the solution of \eqref{b-bar} is in $\cA_\infty(a_1)$ for some $a_1(r_0,T)>0$. \\
(c) To verify Assumption (E), it suffices to show that for any 
$a\in(0,1)$ there exist $n_0(a)\in\N$ and $\delta=\delta(a)>0$, 
such that the jump transition probability measure $\mathsf{P}_r$ 
of the embedded Markov chain $\{x_{T_i}^{1,r}\}$
satisfies the uniform minorisation condition
\begin{align}\label{minor}
(\mathsf{P}_r)^{n_0}(x,\cdot)\ge \delta m_r(\cdot)\,,
\end{align}
for any $(r,x)\in \cAe_\infty(a)$ and some probability measure $m_r(\cdot)$ on $\R^n$.\\
(d) Assumption (C) holds if for any $a \in (0,1)$, 
\begin{equation}\label{F-ubd-Lip}
\sup \{  \|F(r,x,\cdot)\|_{\rm Lip} : (r,x) \in \cAe_\infty(a) \} < \infty \,.
\end{equation}
(e) Condition \eqref{F-ubd} implies that for some positive $\delta_n(r_0,T)$ and any
$0<\delta< \delta_n$
\[
\lim_{\ep\to0}\P(\zeta^\ep(\delta)<T\wedge\kappa^\ep)=0,
\]
with the $\cF_t$-stopping times $\zeta^\ep(\cdot)$ of \eqref{zeta} and  
\begin{align}\label{kappa-ep}
\kappa^\ep:=\inf\big\{t\ge0:\, \|R^\ep_t\|_2>1+\|r_T\|_2\big\}.
\end{align}
\end{ppn}

Recall \eqref{eq:Leb-add} that
the random dynamics \eqref{slow-component} has expected volume increase of 
$\ep(1+o(1))$ at each Poisson jump, (irrespective of the precise choice of  
$\eta(\ep,r,x) \to 0$ as $\ep \to 0$).
We thus expect the following result (whose proof is also deferred 
to the appendix), about the linear growth of 
the volume of the deterministic dynamics \eqref{b-bar}.
\begin{ppn}\label{vol-growth}
If the solution $( r_t )_{t\ge0}$ to the \abbr{ODE} \eqref{b-bar} belongs to $C_+(\bS^{n-1})$ for all $t\ge0$, then $Leb( r_t )= Leb( r_0 ) + t$. Further, \eqref{eq:Leb-add} holds throughout $\cD(F)$.
\end{ppn}

Under the following scaling invariance of $F$ and $H$, we will 
deduce a shape theorem for the process $(R_t^1)_{t \ge 0}$, from the averaging principle of Theorem \ref{thm1}.

\begin{assumption}[I]
For any scalar $c>0$, if $(r,x)\in\cD(F)$ then $(cr,cx)\in\cD(F)$ and
\begin{align}\label{scaling-inv}
F(r,x,\cdot)&=F(cr,cx,\cdot),\\
cH(r,\cdot)&=H(cr,\cdot). \label{scaling-H}
\end{align}
\end{assumption}

\begin{defn}\label{def:inv-sol}
(a) A function $\psi\in C_+(\bS^{n-1})$ is called invariant (shape) for the \abbr{ODE} \eqref{b-bar}, if starting at $ r_0 =\psi$ yields 
\begin{align*}
r_t =(1+t/Leb(\psi))^{1/n} \psi, \quad t\ge 0.
\end{align*}
(b) A function $\psi\in C_+(\bS^{n-1})$ is called attractive (shape) for the \abbr{ODE} \eqref{b-bar} and a collection $\cC$ of initial data, if starting at any $ r_0 \in \cC$, the solution $t \mapsto  r_t  \in C_+(\bS^{n-1})$ exists, with
\begin{align}\label{cond:attr}
\lim_{t\to\infty}\Big \|
(Leb( r_0 )+t)^{-1/n}  r_t - Leb(\psi)^{-1/n} \psi \Big \|_2=0.
\end{align}
\end{defn}

In general, invariant shapes may not be unique, nor are they necessarily attractive. See Example \ref{non-uniq}.

\begin{thm}[Shape theorem]\label{thm2}
Suppose Assumption (I) holds and \eqref{hydrodyn} 
applies without the stopping time $\sigma^\ep(\delta)$ 
(see Remark \ref{st-remov}).\\
(a) If a function $\psi$ with $Leb(\psi)=1$ is invariant for the \abbr{ODE} \eqref{b-bar}, then for any $c>0$, $1\le T<\infty$ and $\iota>0$, 
\begin{align}\label{shape-thm-1-a}
\lim_{N\to\infty}\P\Big(\sup_{1\le s\le T}\left \|
(N(c+s))^{-1/n} R^1_{sN}-\psi\right \|_2>\iota\, \Big|\, R^1_0=(cN)^{1/n}\psi\Big)=0\,.
\end{align}
(b) If a function $\psi$ with $Leb(\psi)=1$ is attractive for the \abbr{ODE} \eqref{b-bar} and a collection $\cC$ of initial data, then for any $\iota>0$ and $r_0\in\cC$,
\begin{align*}
\lim_{t\to\infty}\lim_{N\to\infty}\P\Big(\big \|
\big(N(Leb(r_0)+t)\big)^{-1/n} R^1_{tN}-\psi\big\|_2>\iota\, \big|\, R^1_0=N^{1/n}r_0\Big)=0\,.
\end{align*}
\end{thm}

Our main application is a model of random growth on $\R^n$ 
motivated by the expected mesoscopic behavior of \abbr{ORRW} 
and \abbr{OERW} on $\Z^n$, where to gain regularity we consider 
$F$ and $H$ defined via a smoothed version of the evolving domain. 
Specifically, fix $\eta>0$ 
and $g=g_\eta$ as in \eqref{def:g} for some $\phi_\eta \in C^3([-1,1])$. 
Then, $\wt r:=r\star g\in C^3(\bS^{n-1})$ 
for every $r\in L^2(\bS^{n-1})$ (see \eqref{integ-1}). We set
\begin{align}\label{def-conv-F}
F(r,x,\theta):=\frac{\partial}{\partial\mathbf{n}}G_{\wt{r}}(x,y)\Big|_{y=\wt{r}(\theta)\theta} \qquad \theta\in \bS^{n-1} \,,
\end{align}
where $G_{\wt{r}}(x,y)$ denotes
the Green's function of the Laplacian $-\Delta$ 
on star-shaped domain $D\subseteq\R^n$ 
with Dirichlet boundary conditions at $\wt{r}=\partial D$
and $\frac{\partial}{\partial\mathbf{n}}$ is the inward normal  
derivative on $\partial D$. Similarly, fix a 
locally Lipschitz function $\alpha \colon \R_{>0}\times \bS^{n-1} \to \R_{\ge 0}$ such that
\begin{equation}\label{alpha:bdd}
0 \le \alpha(\ell,z) < \ell \;\;\;\; \text{ on } \;\;\;\; \R_{>0}\times \bS^{n-1}\,,
\end{equation}
and set (see Section \ref{sec:app} for the probabilistic interpretation),
\begin{align}\label{def-conv-H}
H(r,z): = \alpha(\wt r(z),z)z \,.
\end{align}

Another natural hitting rule $F$ chooses a boundary point with probability 
``proportional to a function of the distance to the particle''. Specifically, fixing a locally Lipschitz
$\varphi:(0,\infty)\to (0,\infty)$ and with $\wt{r}$ as above, replace \eqref{def-conv-F} by
\begin{align}\label{rule:distance}
F(r,x,\theta)=\frac{\varphi(|\wt r(\theta)\theta-x|)}{\int_{\bS^{n-1}}\varphi(|\wt r(z)z-x|)d\sigma(z)}, \quad \theta\in \bS^{n-1}\,,
\end{align}
while keeping the 
rule $H$ of \eqref{def-conv-H}.
We have the following results for these rules.
\begin{thm}\label{harm:measure}
$~$ \newline
(a) The Averaging Principle of Theorem \ref{thm1} holds under either
\eqref{def-conv-F}-\eqref{def-conv-H} or \eqref{alpha:bdd}--\eqref{rule:distance},
 without the stopping times $\sigma^\ep(\delta)$
of \eqref{pdf-lbd}. \\
(b) In case $\alpha(\ell, z)=\alpha(z)\ell$ and $\vphi(t)=t^{\beta}$,
the Shape Theorem \ref{thm2} also 
holds. In particular, for $\alpha(\ell, z)=\gamma\ell$ with $\gamma\in[0,1)$ 
fixed, the centered Euclidean ball is an invariant shape.\\
(c) For $\gamma = 0$, $\beta<0$ and setting $\wt{r}=r$ in 
\eqref{def-conv-F} and \eqref{rule:distance}, the centered  Euclidean ball is 
uniquely attractive among $C^1_+(\bS^{n-1})$ initial data.
\end{thm}

\begin{rmk} Similar conclusions apply for other transportation rules, such as
\begin{align}\label{rule:stat-center}
H(r) =\int_{\bS^{n-1}}\wt r(z)zd\sigma(z) \,,
\end{align}
which sends the particle to a statistical center of the domain. Note that $H(r)$ may 
possibly be outside the domain, and if start-shaped $D$ is invariant for the 
\abbr{ode} corresponding to \eqref{rule:stat-center} and \eqref{def-conv-F}, then 
so are all translations of $D$ which are star-shaped. Also,
for such $H(\cdot)$ which depends on the domain only, trivially the
invariant measure $\nu_r$ is the Dirac mass at $H(r)$ and Assumption (E) 
holds. 
\end{rmk}

The rest of the article is organized as follows. In Section \ref{sec:proof} we prove Theorem \ref{thm1} and in Section \ref{sec:shape} we deduce the shape result, Theorem \ref{thm2}. In Section \ref{sec:app} we
apply these theorems to concrete growth models and prove Theorem \ref{harm:measure}.

\section{Proof of Theorem \ref{thm1}}\label{sec:proof}
We start with bounding the Wasserstein 2-distance between any
two measures on a compact, connected Riemannian manifold, by 
the $L^2$-distance between their densities with respect to the 
Riemannian measure.
\begin{lem}\label{H_-1}
Let $M\subseteq \R^n$ be a connected Riemannian manifold without boundary compactly embedded in $\R^n$, equipped with its Riemannian distance $d(\cdot,\cdot)$ and measure $\sigma(\cdot)$. Let $\mu,\nu$ be probability distributions on $M$ having densities $p,q$ respectively with respect to $\sigma(\cdot)$, where in addition 
$p(x) \ge c >0$ for all $x \in M$. 
Then, there exists $C=C(M,c)<\infty$ such that
\begin{align*}
W_2(\mu,\nu)\le C\|p-q\|_2\, ,
\end{align*}
where 
\begin{align*}
W_2(\mu,\nu):=\inf\Big\{\big[\E\, d(X,Y)^2\big]^{1/2}: \text{Law }(X)=\mu, \, \text{Law }(Y)=\nu\Big\}
\end{align*}
is the Wasserstein $2$-distance between $\mu$ and $\nu$, and $\|\cdot\|_2:=\|\cdot\|_{L^2(\sigma)}$.
\end{lem}

\begin{proof}[Proof of Lemma \ref{H_-1}. ]
By \cite[Theorem 1]{Pe}, we have the variational representation
\begin{align*}
W_2(\mu,\nu) &\le 2\sup_{\{f\in C^1(M):\, \int_M|\nabla f|^2d\mu \le 1\}}\Big|\int_M
 f d(\mu-\nu)\Big|\\
&\le 2 \sup_{\{f\in C^1(M): \|\nabla f\|_2 \le c^{-1/2}\}}
\Big|\int_M (f-\ovl{f}_M) d(\mu-\nu)\Big|  \\
&\le  2\sup_{\{f\in C^1(M): \|\nabla f\|_2 \le c^{-1/2}\}} \|f-\ovl{f}_M\|_2\|p-q\|_2 
\le 2c^{-1/2}c(M)\|p-q\|_2.
\end{align*}
In the last step, we have used the Poincar\'e inequality $\|f-\ovl{f}_M\|_2\le c(M)\|\nabla f\|_2$, where $\ovl{f}_M$ denotes the $\sigma$-weighted average of $f$ in $M$ and $c(M)$ is the Poincar\'e constant.
\end{proof}

The proof of Theorem \ref{thm1} is based on considering an auxiliary process in which the slow variable is frozen (this is a standard  tool for proving averaging principles, see \cite{Kh,Ve}). 
Set
\begin{align}\label{def:Delta}
\Delta=\Delta(\ep)=\ep\log^{1/3}(\ep^{-1}) \wedge 1.
\end{align}
Given the main process $(R^\ep_t,x_t^\ep)_{t\ge0}$, we consider a family (indexed by $\ep>0$) of auxiliary dynamics $(\wh{R}_t^\ep,\wh{x}_t^\ep)_{t\ge0}$ defined piecewise on each time interval $[k\Delta,(k+1)\Delta)$ with $k\in\N$, on the same probability space $(\Omega, \cF, \P)$ as the main process, as follows. Inductively for every $k\in\N$, take the {\it {same}} Poisson clock $\{T^\ep_i\}_{i\in\N}$ used in constructing the main process, and starting at $\wh{x}_{k\Delta}^\ep=x^\ep_{k\Delta}$, let $(\wh{x}^\ep_t)_{t\in[k\Delta, (k+1)\Delta)}$ have the marginal distribution of the Markov jump process in the frozen domain $R^\ep_{k\Delta}$ defined as in \eqref{frozen-dyn}. That is, $(\wh{x}^\ep_t)_{t\in[k\Delta, (k+1)\Delta)}$ jumps at each $T^\ep_i\in[k\Delta, (k+1)\Delta]$, $i\in\N$ in the {\it{frozen}} domain $R^\ep_{k\Delta}$, by first using probability density $F(R^\ep_{k\Delta}, \wh{x}^\ep_{T_i^{\ep-}},\cdot)$ to choose a spherical angle $\wh{\xi}_i$, then applying the rule $H(R^\ep_{k\Delta},\wh{\xi}_i)$. We further put requirement on the joint law such that at each jump, $\wh{\xi}_i$ and $\xi_i$ of \eqref{bump-loc} achieve 
$\E[\, d(\xi_i,\wh{\xi}_i)^2]^{1/2}$
within twice the Wasserstein $2$-distance $W_2(\mu, \nu)$ on $\bS^{n-1}$, where $\mu=F(R^\ep_{T_i^{\ep-}}, x^\ep_{T_i^{\ep-}}, \cdot)d\sigma$, $\nu=F(R^\ep_{k\Delta}, \wh{x}^\ep_{T_i^{\ep-}}, \cdot)d\sigma$. Inductively the above procedure defines $(\wh{x}_t)_{t\ge 0}$ on $(\Omega, \cF, \P)$.

We then define $(\wh{R}^\ep_t)_{t\ge 0}$ on $(\Omega, \cF, \P)$ as the dynamics driven by the \abbr{ODE}, with $\wh{R}^\ep_0=R^\ep_0$, 
\begin{align}
\wh{R}_t^\ep&=R^\ep_0+\int_0^t b(R^\ep_{\lfloor s/\Delta\rfloor \Delta},\wh{x}^\ep_s)ds, \quad t\ge0.   \label{aux-R}
\end{align}

With the auxiliary processes in place, we proceed to the proof of the theorem. By Proposition \ref{ppn:implic}(b), starting at $ r_0 \in C_+(\bS^{n-1})$, the solution $( r_t )_{t\ge0}$ to the \abbr{ODE} \eqref{b-bar} exists and is unique in $C_+(\bS^{n-1})$. Fixing $\delta>0$, since $R^\ep_t \ge R^\ep_0= r_0$, for any $\ep >0$ 
the stopped at 
$\zeta^\ep(\delta)$ of \eqref{zeta}, sample path 
$R^\ep_{t\wedge\zeta^\ep(\delta)}$ remains within $\cA_\infty(\mathsf{a})$, provided $\mathsf{a}\in (0,\delta \wedge \inf_\theta  r_0 (\theta))$. Hereafter, we only apply Assumptions (L) and (E) with Lipschitz constant $K(\mathsf{a})$, resp. convergence rate $\lambda$ in \eqref{rate}, depending on such fixed $\mathsf{a}$, for the stopped processes.

Next, by \eqref{real}, for any $u\le t$ we have per $\theta$,
\begin{align}\label{R-diff}
(R_t^\ep-R_u^\ep)(\theta)=\int_u^t b^\ep(R_s^\ep,x_s^\ep)(\theta) ds+(\Sigma_t^\ep-\Sigma_u^\ep)(\theta).
\end{align}
By \cite[Proposition 8.7]{DN},  for the stopped martingale $\Sigma^\ep_{s\wedge\zeta^\ep(\delta)}(\theta)$,
\begin{align*}
\E\Big[\sup_{s\in[u,t]}(\Sigma_{s\wedge\zeta^\ep(\delta)}^\ep-\Sigma_{u\wedge\zeta^\ep(\delta)}^\ep)^2(\theta)\Big]\le 4\ep \E\int_{u\wedge\zeta^\ep(\delta)}^{t\wedge\zeta^\ep(\delta)} b^\ep(R^\ep_s,x^\ep_s)^2(\theta)ds\,,
\end{align*}
which together with Fubini, implies that
\begin{align}\label{bound-Sigma}
\E\Big[\sup_{s\in[u,t]}\|\Sigma_{s\wedge\zeta^\ep(\delta)}^\ep-\Sigma_{u\wedge\zeta^\ep(\delta)}^\ep\|^2_2\Big]  & \le \int_{\bS^{n-1}}  \E\Big[\sup_{s\in[u,t]} (\Sigma_{s\wedge\zeta^\ep(\delta)}^\ep-\Sigma_{u\wedge\zeta^\ep(\delta)}^\ep)^2(\theta)\Big]d\sigma(\theta)  \nonumber \\
& \le 4\ep\E\int_{u\wedge\zeta^\ep(\delta)}^{t\wedge\zeta^\ep(\delta)}\|b^\ep(R^\ep_s,x^\ep_s)\|_2^2ds 
\,.
\end{align}
Further, since spherical convolution is a contraction in $L^2(\bS^{n-1})$ (per Definition \ref{def:apx-iden}), and
the Lipschitz bound \eqref{lip-b} holds throughout $\cAe_{\infty}(\mathsf{a})$, hence the norms
$\|b(R^\ep_{s\wedge\zeta^\ep(\delta)}, x^\ep_{s\wedge\zeta^\ep(\delta)})\|_2$ 
are uniformly bounded, 
\begin{align}\label{bound-b-ep}
 \E\int_{u\wedge\zeta^\ep(\delta)}^{t\wedge\zeta^\ep(\delta)}\|b^\ep(R^\ep_s,x^\ep_s)\|_2^2ds 
& \le \E\int_{u\wedge\zeta^\ep(\delta)}^{t\wedge\zeta^\ep(\delta)}\|b(R^\ep_s,x^\ep_s)\|_2^2ds 
\le C (t-u),
\end{align}
for some finite $C=C(K,\delta)$ and all $0\le u\le t$.
Consequently,  by \eqref{R-diff}-\eqref{bound-b-ep} 
and Cauchy-Schwarz, for any $0\le u\le t$,
\begin{align}\label{R-bound}
\frac{1}{2}\E\| R_{t\wedge\zeta^\ep(\delta)}^\ep  - & R_{u\wedge\zeta^\ep(\delta)}^\ep  \|^2_2 
\le\E\big\|\int_{u\wedge\zeta^\ep(\delta)}^{t\wedge\zeta^\ep(\delta)} b^\ep(R_s^\ep,x_s^\ep)ds\big\|^2_2+\E\|\Sigma_{t\wedge\zeta^\ep(\delta)}^\ep-\Sigma_{u\wedge\zeta^\ep(\delta)}^\ep\|^2_2 
\nonumber \\
&\le \E\Big(\int_{u\wedge\zeta^\ep(\delta)}^{t\wedge\zeta^\ep(\delta)}\|b^\ep(R^\ep_s,x^\ep_s)\|_2ds\Big)^2  + 4\ep\E\int_{u\wedge\zeta^\ep(\delta)}^{t\wedge\zeta^\ep(\delta)}\|b^\ep(R^\ep_s,x^\ep_s)\|_2^2ds 
\nonumber \\
&\le (t-u)\E\int_{u\wedge\zeta^\ep(\delta)}^{t\wedge\zeta^\ep(\delta)}\|b^\ep(R^\ep_s,x^\ep_s)\|_2^2ds+
4 \ep C (t-u) \nonumber \\
& \le C(t-u)^2+4 \ep C (t-u).
\end{align}

We let hereafter $\tau = \zeta^\ep(\delta) \wedge \sigma^\ep(\delta)$ 
(for 
$\zeta^\ep$ and $\sigma^\ep$ of \eqref{zeta} and 
\eqref{pdf-lbd}, respectively), and rely on \eqref{R-bound} to establish the following two lemmas.
\begin{lem}\label{lem2}
In the setting of Theorem \ref{thm1}, we have that
\begin{align*}
\lim_{\ep\to 0}\E\Big[\sup_{0\le t\le T}\|R_{t\wedge\tau}^\ep-\wh{R}_{t\wedge\tau}^\ep\|_2^2\Big]=0.
\end{align*}
\end{lem}

\begin{lem}\label{lem3}
In the setting of Theorem \ref{thm1}, we have that
\begin{align}
\lim_{\ep\to0}\E\Big[\sup_{0\le t\le T}\|\wh{R}_{t\wedge\tau}^\ep- r_{t\wedge\tau}  \|_2^2\Big]=0.
\end{align}
\end{lem}

While deferring the proofs of Lemmas \ref{lem2} and \ref{lem3} to the end of the section, we
note that these lemmas together with the definition \eqref{kappa-ep} of $\kappa^\ep$ yield that
\begin{align}\label{L2St}
\lim_{\ep\to0}\P(\kappa^\ep\le T\wedge\tau)=0\,.
\end{align}
In view of Proposition \ref{ppn:implic}(d), this in turn results with 
\begin{align}\label{remove-st}
\lim_{\ep\to0}\P(\zeta^\ep(\delta)< T\wedge\sigma^\ep(\delta))=0
\end{align}
(provided $\delta<\delta_n$). The conclusions of Lemma \ref{lem2} and Lemma \ref{lem3}
are thus strengthened 
to apply with $\sigma^\ep(\delta)$ instead of $\tau$, so combining these 
lemmas with Markov's inequality completes the proof of the theorem.
 
Turning to establish Lemmas \ref{lem2} and \ref{lem3}, we start with the following 
key bound.
\begin{lem}\label{lem1}
In the setting of Theorem \ref{thm1}, for some finite $C=C(K(\mathsf{a}),\delta)$ we have 
\begin{align*}
\sup_{0\le t\le T}\E\Big[{\bf 1}_{\{ t \le \tau \}} |x^\ep_{t}-\wh{x}^\ep_{t}|^2\Big]\le C\ep.
\end{align*}
\end{lem}
\begin{proof}[Proof of Lemma \ref{lem1}. ]
Per \eqref{frozen-dyn}, for each $k\in\N$ the auxiliary process $(\wh{x}^\ep_t)_{t\in[k\Delta, (k+1)\Delta)}$ admits the decomposition
\begin{align}
\wh{x}^\ep_t&=x_{k\Delta}^\ep+\ep^{-1}\int_{k\Delta}^th(R^\ep_{k\Delta},\wh{x}^\ep_s)ds+\wh{M}^\ep_{k\Delta,t}, \quad t\in[k\Delta, (k+1)\Delta), \label{aux-x}
\end{align}
for some $\R^n$-valued, $\cF_t$-martingale $\wh{M}^\ep_{k\Delta,t}$
(setting $\wh{M}^\ep_{k\Delta,t} \equiv 0$ when $t \le k \Delta$).
Taking the difference of \eqref{aux-x} with  \eqref{bump}, we have that for $t\in[k\Delta, (k+1)\Delta)$,
\begin{equation} \label{diff-x}
\begin{aligned}
(x_t^\ep-\wh{x}_t^\ep)-\int_{k\Delta}^t\ep^{-1}[h(R_s^\ep,x_s^\ep)-h(R^\ep_{k\Delta},\wh{x}_s^\ep)]ds=&M^\ep_t-M^\ep_{k\Delta}-\wh{M}^\ep_{k\Delta,t},
\end{aligned}
\end{equation}
is an $\R^n$-valued martingale. 

Considering the generator of $(x_t^\ep-\wh{x}^\ep_t)_{t\in[k\Delta,(k+1)\Delta)}$, we have
by \cite[Proposition 8.7]{DN} and \eqref{lip-H}  that for any $t \in[k\Delta,(k+1)\Delta)$,
\begin{align}\label{use-lip-H}
 \E| \,  & M^\ep_{t\wedge\tau}  -  M^\ep_{k\Delta \wedge \tau} 
-\wh{M}^\ep_{k\Delta, t\wedge\tau}|^2   \nonumber \\ & 
\le \frac{4}{\ep}\E  \, {\bf 1}_{\{k \Delta \le  \tau\}}
\int_{k\Delta}^{t \wedge \tau}
 \mathsf{E} \big|H(R_{s}^\ep,\xi_{s})-x^\ep_{s}
 -H(R^\ep_{k\Delta}, \wh{\xi}_{s})+\wh{x}^\ep_{s}\big|^2ds \nonumber \\
 & \le\frac{12 K^2}{\ep}  \E\, 
{\bf 1}_{\{k \Delta \le  \tau\}}
\int_{k\Delta}^{t \wedge \tau} 
\left(\|R_{s}^\ep-R^\ep_{k\Delta}\|_2^2
+ |x^\ep_{s} - \wh{x}^\ep_{s}|^2  + \mathsf{E}|\xi_{s}-\wh{\xi}_{s}|^2 \right) ds,      
\end{align}
 where the inner {\em conditional} expectation $\mathsf{E}$ is only over $(\xi_{s},\wh{\xi}_{s})$, having marginal densities $F(R^\ep_{s},x^\ep_{s},\cdot)$ and 
 $F(R^\ep_{k\Delta},\wh{x}^\ep_{s},\cdot)$, 
 with respect to $\sigma(\cdot)$ on $\bS^{n-1}$. By the coupling we chose, and Lemma \ref{H_-1} with $M=\bS^{n-1}$, $p=F(R^\ep_{s},x^\ep_{s},\cdot)$ bounded below by $\delta$ when $s < \tau$, 
 we have in \eqref{use-lip-H} for $s < \tau$,
\begin{align*}
\mathsf{E}|\xi_{s\wedge\tau}-\wh{\xi}_{s\wedge\tau}|^2 \le 4 W_2(\xi_{s},\wh{\xi}_{s})^2
&\le C(\delta)\|F(R^\ep_{s},x^\ep_{s},\cdot)-F(R^\ep_{k \Delta},\wh{x}^\ep_{s},\cdot)\|_2^2\\
&\le CK^2(\|R^\ep_{s}-R^\ep_{k  \Delta}\|_2^2+|x^\ep_{s}-\wh{x}^\ep_{s}|^2) \, ,
\end{align*}
using \eqref{lip-F} in the last line. Consequently, we obtain from \eqref{use-lip-H} that
\begin{align}\label{M-bound}
\E|M^\ep_{t\wedge\tau} & -M^\ep_{k\Delta\wedge\tau}  -\wh{M}^\ep_{k\Delta, t\wedge\tau}|^2 \nonumber \\
&\le C(K, \delta)\ep^{-1}\E\Big[{\bf 1}_{\{ k\Delta \le \tau \}}\int_{k\Delta}^{t\wedge\tau}(\|R^\ep_{s}-R^\ep_{k \Delta }\|_2^2+|x^\ep_{s}-\wh{x}^\ep_{s}|^2)ds\Big] \nonumber \\
&\le \wt{C} \ep^{-1}\big(
\Delta^3 + \int_{k\Delta}^{t} \E [{\bf 1}_{\{ s \le \tau \}} |x_{s}^\ep-\wh{x}_{s}^\ep|^2 ] ds\big)\, , 
\end{align}
using in the last line \eqref{R-bound} and that $\Delta \ge \ep$. From \eqref{diff-x} at $t \wedge \tau$, 
Cauchy-Schwarz inequality,
\eqref{M-bound}, \eqref{R-bound} and $\Delta \ge \ep$,
 we have that for some finite $C=C(K,\delta)$ and any $t\in[k\Delta, (k+1)\Delta)$, $k=0,1,\ldots,\lfloor T/\Delta\rfloor$, 
\begin{align*}
& \frac{1}{2}\E [\, {\bf 1}_{\{ t \le \tau \}}  |x^\ep_{t} -\wh{x}^\ep_{t}|^2 ] =
\frac{1}{2}\E [\, {\bf 1}_{\{ t \le \tau \}} |x^\ep_{t\wedge \tau} -\wh{x}^\ep_{t \wedge \tau}|^2 ]
 \\
&\le \ep^{-2}\E\Big[{\bf 1}_{\{ k\Delta \le \tau \}} \int_{k\Delta}^{t\wedge\tau}|h(R_s^\ep,x_s^\ep)-h(R^\ep_{k\Delta},\wh{x}_s^\ep)|ds\Big]^2 +\E|M^\ep_{t\wedge\tau}-M^\ep_{k\Delta\wedge\tau}-\wh{M}^\ep_{k\Delta, t\wedge\tau}|^2\nonumber\\
&\le 2 K^2\ep^{-2}\Delta \, \E \Big[{\bf 1}_{\{ k\Delta \le \tau \}} \int_{k\Delta}^{t\wedge\tau}(\|R_s^\ep-R^\ep_{k\Delta}\|_2^2+|x_s^\ep-\wh{x}_s^\ep|^2)ds \Big]\\
&\quad\quad\quad\quad +\wt{C} \ep^{-1}\big(\Delta^3+\int_{k\Delta}^{t}\E [ {\bf 1}_{\{s \le \tau \}}|x_{s}^\ep-
\wh{x}_{s }^\ep|^2 ] ds\big)\nonumber\\
&\le C \ep^{-2}\Delta^4+C \ep^{-2}\Delta \int_{k\Delta}^{t} \E [ {\bf 1}_{\{s \le \tau \}}|x_{s}^\ep-\wh{x}_{s}^\ep|^2 ] ds \, ,
\end{align*}	
utilizing also that $h(\cdot,\cdot)$ is Lipschitz, as in \eqref{lip-h} of Proposition \ref{ppn:implic}(a).
By Gronwall's inequality for $e(s)=\E[{\bf 1}_{\{s \le \tau \}}|x_{s}^\ep-\wh{x}_{s}^\ep|^2 ]$,
starting at $e(k\Delta)=0$, we get that 
\begin{align}\label{x-approx}
\sup_{0 \le t \le T} \{ e(t) \} \le 2 C \ep^{-2}\Delta^4e^{C\ep^{-2}\Delta^2}.
\end{align}
For our choice \eqref{def:Delta} of $\Delta=\Delta(\ep)$, the \abbr{RHS} of \eqref{x-approx} is bounded by 
$C\ep$ for some (other) finite $C=C(K,\delta)$, as claimed.
\end{proof}

\begin{proof}[Proof of Lemma \ref{lem2}. ]
Per \eqref{real} and \eqref{aux-R}, for any $t\ge0$ we have that
\begin{align*}
R^\ep_t-\wh{R}^\ep_t=\int_0^t[b^\ep(R^\ep_s,x_s^\ep)-b(R^\ep_s,x^\ep_s)]ds+\int_0^t[b(R_s^\ep,x_s^\ep)-b(R^\ep_{\lfloor s/\Delta\rfloor \Delta},\wh{x}_s^\ep)]ds+\Sigma_t^\ep \,.
\end{align*}
Consequently, by Cauchy-Schwartz, 
\begin{align*}
&\frac{1}{3}\E\sup_{0\le t\le T}\|R_{t\wedge\tau}^\ep-\wh{R}_{t\wedge\tau}^\ep\|_2^2 
\le T\E\int_0^T\|b^\ep(R^\ep(_{s\wedge\tau},x_{s\wedge\tau}^\ep)-b(R^\ep_{s\wedge\tau},x^\ep_{s\wedge\tau})\|_2^2ds \\
&\qquad\qquad+T\E\int_0^{T\wedge\tau}\|b(R_s^\ep,x_s^\ep)-b(R^\ep_{\lfloor s/\Delta\rfloor \Delta},\wh{x}_s^\ep)\|_2^2ds+\E\sup_{0\le t\le T}\|\Sigma_{t\wedge\tau}^\ep\|_2^2\, ,
\end{align*}
where the first term tends to zero as $\ep\to0$ by Assumption (C) and the uniform boundedness of the integrand. The remaining two terms are bounded 
via \eqref{lip-b}, \eqref{bound-Sigma}-\eqref{R-bound} and Lemma \ref{lem1} by
\begin{align*}
K^2T\E\int_0^{T} & {\bf 1}_{\{ s \le \tau \}} \big(\|R_s^\ep-R^\ep_{\lfloor s/\Delta\rfloor \Delta}\|_2^2+|x_s^\ep-\wh{x}_s^\ep|^2\big)ds+CT\ep\\
&\le CT\int_0^T(s-\lfloor s/\Delta\rfloor \Delta)^2ds+CT^2\ep+CT\ep
\le CT^2\Delta^2+CT^2\ep \,,
\end{align*}
for some generic constants $C=C(K,\delta)$, hence
tending to zero as well.
\end{proof}

\begin{proof}[Proof of Lemma \ref{lem3}. ]
Per \eqref{rate} and the fact that the event $\{k\Delta\le\zeta^\ep(\delta)\}$ is measurable on 
$\sigma\left(R^\ep_s, s \le k\Delta \right)$, we have that uniformly for 
$k=0,1,\ldots,\lfloor T/\Delta \rfloor-1$, 
\begin{align} \label{ergodic}
\E\Big[\mathbf{1}_{\{k\Delta\le\tau\}} & \big\|\int_{k\Delta}^{(k+1)\Delta}[b(R^\ep_{k\Delta},\wh{x}^\ep_s)-\ovl{b}(R^\ep_{k\Delta})]ds\big \|_2^2\Big] \nonumber \\
&=\E\Big[\mathbf{1}_{\{k\Delta\le\tau\}}\big \|\int_{k\Delta}^{k\Delta+\Delta/\ep}\ep [b(R^\ep_{k\Delta},\wh{x}_t^1)-\ovl{b}(R^\ep_{k\Delta})] dt \big \|_2^2\Big] \nonumber \\
\le \Delta^2\E\Big[\,& \mathbf{1}_{\{k\Delta\le\zeta^\ep(\delta)\}}  \E\Big( \big \|
\frac{1}{\Delta/\ep}
\int_{k\Delta}^{k\Delta+\Delta/\ep} [b(R^\ep_{k\Delta},\wh{x}_t^1) -\ovl{b}(R^\ep_{k\Delta})]dt\big \|_2^2\;\;\Big| R^\ep_s, s \le k\Delta\Big)\Big] \nonumber \\
&\qquad \qquad \le \Delta^2\lambda(\Delta/\ep,\mathsf a) \, .
\end{align}
It then follows from  \eqref{ergodic}, \eqref{lip-bar} and \eqref{R-bound} that for some finite $C=C(K,\delta)$ and any $t\in[0,T]$,

\begin{align}\label{use-erg}
\frac{1}{2}\E\sup_{0\le u\le t} & \Big \|\int_0^{u\wedge\tau}[b(R^\ep_{\lfloor s/\Delta\rfloor \Delta},\wh{x}^\ep_s)-\ovl{b}(R^\ep_s)]ds\Big \|_2^2 \nonumber \\
&\le \E\sup_{0\le u\le t}\Big \|\int_0^{u\wedge\tau}[b(R^\ep_{\lfloor s/\Delta\rfloor \Delta},\wh{x}^\ep_s)-\ovl{b}(R^\ep_{\lfloor s/\Delta\rfloor \Delta})]ds\Big \|_2^2  \nonumber \\
& \qquad\qquad\qquad\qquad + \E\Big(\int_0^{t\wedge\tau}\|\ovl{b}(R^\ep_{\lfloor s/\Delta\rfloor \Delta})-\ovl{b}(R^\ep_s)\|_2ds\Big)^2                      \nonumber \\
&\le \lfloor t/\Delta\rfloor \sum_{k=0}^{\lfloor t/\Delta\rfloor-1}\E\,\mathbf{1}_{\{k\Delta\le\tau\}}\Big \|\int_{k\Delta}^{(k+1)\Delta}[b(R_{k\Delta}^\ep,\wh{x}^\ep_s)-\ovl{b}(R^\ep_{k\Delta})]ds\Big \|_2^2 +C\Delta^2 \nonumber \\
&\qquad \qquad  \qquad\qquad+t\E\int_0^{t\wedge\tau}\|\ovl{b}(R^\ep_{\lfloor s/\Delta\rfloor \Delta})-\ovl{b}(R^\ep_s)\|_2^2ds   \nonumber   \\
&\le t^2\lambda(\Delta/\ep,\mathsf a)+CK^2t^2\Delta^2,  
\end{align}
where $\lambda(\Delta/\ep, \mathsf a)\to0$ as $\ep\to0$ by \eqref{rate} since $\Delta/\ep\to\infty$. We proceed to bound
\begin{align*}
m^\ep(t):=\E\Big[\sup_{0\le u\le t}\|\wh{R}_{u\wedge\tau}^\ep-r_{u\wedge\tau}\|_2^2\Big]
\end{align*}
via Gronwall's inequality. Per \eqref{aux-R} and \eqref{b-bar}, for any $t\ge0$ we have that
\begin{align*}
\wh{R}^\ep_t- r_t =\int_0^t[b(R^\ep_{\lfloor s/\Delta\rfloor \Delta},\wh{x}^\ep_s)-\ovl{b}(R^\ep_s)]ds+\int_0^t[\ovl{b}(R^\ep_s)-\ovl{b}( r_s )]ds\,.
\end{align*}
By \eqref{lip-bar} and  \eqref{use-erg} we have that for any $t \le T$,
\begin{align*}
&\quad \frac{1}{2}m^\ep(t)\\
&\le \E\sup_{0\le u\le t}\Big \| \int_0^{u\wedge\tau}[b(R^\ep_{\lfloor s/\Delta\rfloor \Delta},\wh{x}^\ep_s)-\ovl{b}(R^\ep_s)]ds\Big \|^2_2+\E\sup_{0\le u\le t} \Big \|\int_0^{u\wedge\tau}[\ovl{b}(R^\ep_s)-\ovl{b}( r_s )] ds\Big \|_2^2\\
&\le t^2\lambda(\Delta/\ep, \mathsf a)+CK^2t^2\Delta^2 + t K^2\E\int_0^{t\wedge\tau}\|R^\ep_s- r_s \|^2_2ds\\
&\le T^2\lambda(\Delta/\ep,\mathsf a)+CK^2T^2\Delta^2+2T^2K^2 
\E\sup_{0\le u\le T}\|R_{u\wedge\tau}^\ep-\wh{R}_{u\wedge\tau}^\ep\|_2^2 +2TK^2\int_0^t m^\ep (s)ds \, .
\end{align*}
Gronwall's inequality and Lemma \ref{lem2} yield that
\begin{align*}
m^\ep(T)\le 2 \Big(T^2\lambda(\Delta/\ep,\mathsf a)+C K^2T^2\Delta^2+ 2 T^2K^2\E\sup_{0\le u\le T}\|R_{u\wedge\tau}^\ep-\wh{R}_{u\wedge\tau}^\ep\|_2^2\Big) e^{4 T^2K^2},
\end{align*}
converge to zero when $\ep\to 0$, as required.
\end{proof}

\section{Proof of Theorem \ref{thm2}. }\label{sec:shape}

The following intuitive coupling enables to transfer the Averaging Principle for the family of processes $(R^\ep_t)$ as the scale parameter $\ep\to 0$ on finite time horizons, into a shape result for $(R^1_t)$ of scale $1$ as time $t\to\infty$.

\begin{lem}[coupling]\label{lem:coup}
Fix $\ep>0$. Under Assumption (I), with 
\begin{align*}
(R_0^\ep, x^\ep_0)\overset{d}{=}(\ep^{1/n}R^1_0, \ep^{1/n}x^1_0), 
\end{align*}
there exists a coupling such that
\begin{align}\label{coupling}
(R_t^\ep, x^\ep_t)=(\ep^{1/n}R^1_{t/\ep}, \ep^{1/n}x^1_{t/\ep}), \quad t\ge0.
\end{align}
\end{lem}

\begin{proof}
Let $\{T^\ep_i\}_{i\in\N}$ with $T_0^\ep=0$ denote the sequence of Poisson jump times of rate $\ep^{-1}$ used in constructing $( R_t^\ep, x^\ep_t)_{t\ge0}$, for some fixed $\ep>0$. Set $T^1_i:=\ep^{-1} T^\ep_i$, $i\in\N$. By scaling properties of exponential distribution, $\{T^1_i\}_{i\in\N}$ has the law of a sequence of Poisson arrival times of rate $1$, as such we construct $(R^1_t,x^1_t)_{t\ge0}$ using $\{T^1_i\}_{i\in\N}$, on the same probability space $(\Omega, \cF, \P)$ as $(R^\ep_t, x^\ep_t)_{t\ge0}$. 

Starting with $(R_0^\ep, x^\ep_0)=(\ep^{1/n}R^1_0, \ep^{1/n}x^1_0)$ on $(\Omega, \cF, \P)$, suppose we have succeeded in coupling $( R^\ep_s, x_s^\ep)$ with $(\ep^{1/n}R^1_{s/\ep},  \ep^{1/n}x^1_{s/\ep})$ as in \eqref{coupling} up to time $(T_i^\ep)^-$ for some $i\in\N$. Then by \eqref{scaling-inv}, for any $s\le (T^\ep_i)^-$ we have 
\begin{align*}
F( R_s^\ep,x^\ep_s,\cdot)&=F(\ep^{1/n} R_{s/\ep}^1,\ep^{1/n}x^1_{s/\ep},\cdot)=F( R_{s/\ep}^1,x_{s/\ep}^1,\cdot)\, ,\\
y_{ R_s^\ep,x_s^\ep}&=\omega_n\int R_s^\ep(z)^{n-1}F( R^\ep_s,x_s^\ep,z)d\sigma(z)\\
&= \ep^{\frac{n-1}{n}}\omega_n\int R^1_{s/\ep}(z)^{n-1} F( R^1_{s/\ep},x^1_{s/\ep},z)d\sigma(z)=\ep^{\frac{n-1}{n}} y_{ R^1_{s/\ep},x^1_{s/\ep}}\, ,\\
\eta(\ep, R_s^\ep, x^\ep_s)&=\frac{\ep^{1/n}}{y^{1/(n-1)}_{ R^\ep_s, x^\ep_s}}=y^{-1/(n-1)}_{ R^1_{s/\ep}, x^1_{s/\ep}}=\eta(1, R^1_{s/\ep}, x^1_{s/\ep})\, .
\end{align*}
The induction hypotheses and the construction \eqref{bump-loc}, \eqref{slow-component} yield at $t=T_i^\ep$, per $\theta$
\begin{align*}
 \xi^\ep_t&\overset{d}{\sim} F( R_{t^-}^\ep, x_{t^-}^\ep,\cdot)=F( R^1_{t^-/\ep}, x_{t^-/\ep}^1,\cdot) \overset{d}{\sim} \xi^1_{t/\ep}\, ,\\
 R_t^\ep(\theta)&= R_{t^-}^\ep(\theta)+\ep y^{-1}_{ R_{t^-}^\ep, x^\ep_{t^-}}g_{\eta(\ep, R^\ep_{t^-}, x^\ep_{t^-})}(\langle \xi^\ep_t, \theta\rangle)\\
 &=\ep^{1/n} R_{t^-/\ep}^1(\theta)+\ep^{1/n}y^{-1}_{ R^1_{t^-/\ep}, x^1_{t^-/\ep}}g_{\eta(1, R^1_{t^-/\ep}\, , x^1_{t^-/\ep})}(\langle \xi_t^\ep, \theta\rangle)\, ,\\
R_{t/\ep}^1(\theta)&= R^1_{t^-/\ep}(\theta)+ y^{-1}_{ R^1_{t^-/\ep}, x^1_{t^-/\ep}}g_{\eta(1, R^1_{t^-/\ep}\, , \xi^1_{t/\ep})}(\langle \xi^1_{t/\ep}, \theta\rangle)\, .
\end{align*}
By coupling the jumps of $x^\ep_{t^-}$ and $x^1_{t^-/\ep}$ at $t=T^\ep_i$ such that $\xi^\ep_t=\xi^1_{t/\ep}$, we deduce $R^\ep_t=\ep^{1/n} R^1_{t/\ep}$, $\xi^\ep_t=\xi^1_{t/\ep}$ for $t=T^\ep_i$, and by \eqref{fast-component}, \eqref{scaling-H} also $x^\ep_t=\ep^{1/n}x^1_{t/\ep}$. During $t\in[T^\ep_i,T^\ep_{i+1})$, all processes stay put, hence continuing extends the coupling to all $t\ge0$. 
\end{proof}

\begin{proof}[Proof of Theorem \ref{thm2}.]
We only prove part (b), whereas the proof of part (a) is similar. By Theorem \ref{thm1} and Lemma \ref{lem:coup}, we firstly have for any $t<\infty$ and $\iota>0$,
\begin{equation}\label{fixed-t-conv}
\begin{aligned}
&\lim_{\ep\to0}\P(\|\ep^{1/n}R^1_{t/\ep}-r_t\|_2>\iota/2\, |\, R^1_0=\ep^{-1/n}r_0) \\
&=\lim_{\ep\to0}\P(\|R^\ep_t-r_t\|_2>\iota/2\, |\, R^\ep_0=r_0)=0,   
\end{aligned}
\end{equation}
where $(r_t)_{t\ge 0}$ is the continuous solution of \eqref{b-bar} with initial data $r_0\in\cC\cap C_+(\bS^{n-1})$, and $Leb(r_t)=Leb(r_0)+t$. By the triangle inequality, we have that
\begin{align*}
\lim_{\ep\to0}\P&\Big(\left \|(Leb(r_0)+t)^{-1/n}\ep^{1/n}R^1_{t/\ep}-\psi\right \|_2>\iota\, \Big|\, R^1_0=\ep^{-1/n}r_0\Big)\\
\le \lim_{\ep\to0}&\P\Big(\left \|(Leb(r_0)+t)^{-1/n}(\ep^{1/n}R^1_{t/\ep}-r_t)\right \|_2>\iota/2\, \Big|\, R^1_0=\ep^{-1/n}r_0\Big)    \\
&+\mathbf{1}\left\{\left \|(Leb(r_0)+t)^{-1/n}r_t-\psi\right \|_2>\iota/2\right\} \, .
\end{align*}
By \eqref{fixed-t-conv}, the first term vanishes for any $t\ge 1$, and upon taking another limit as $t\to\infty$, the second term vanishes as well by \eqref{cond:attr}. We obtain the claims upon setting $N=\ep^{-1}$.
\end{proof}

\begin{problem}
It remains open to remove the strict positivity of initial condition in Theorem \ref{thm1}, hence to be able to take $c=0$ in \eqref{shape-thm-1-a}, which would correspond to a genuine shape theorem. 
\end{problem}

We have the following general characterization of invariant shapes.

\begin{ppn}\label{inverse}
Under Assumption (I), $\psi\in C_+(\bS^{n-1})$ is invariant for the \abbr{ODE} \eqref{b-bar} if and only if $\ovl{b}(\psi)=n^{-1}\psi/Leb(\psi)$.
\end{ppn}

\begin{proof}
We prove the ``only if" part, while the converse ``if" direction can be checked directly.
Assumption (I) implies that for any $c>0$, $y_{cr, cx}=c^{n-1}y_{r,x}$ and hence $b(cr,cx)=c^{-(n-1)}b(r,x)$. From \eqref{frozen-dyn} the value of 
$(\cL^{1,cr} (\mathsf{f} \circ c^{-1})) (cx)$ is independent of $c$. The natural coupling
$x_t^{1,cr}=cx_t^{1,r}$ implies that $\nu_{cr}(\cdot)$ of Assumption (E) are merely
the push-back of $\nu_r(\cdot)$ under scaling $x \mapsto cx$, resulting with
$\ovl{b}(cr)=c^{-(n-1)}\ovl{b}(r)$. Per Definition \ref{def:inv-sol}(a), an invariant solution $( r_t )_{t\ge 0}$ starting at $ r_0 =\psi$ is such that $ r_t =c_t\psi$ with $(c_t^n-1)Leb(\psi)=t$. From the \abbr{ODE} \eqref{b-bar} it is not hard to infer that $\ovl{b}(\psi)\propto\psi$. Further, by taking derivative of \eqref{b-bar} in $t$ we identify the proportionality constant to be $n^{-1}/Leb(\psi)$. 
\end{proof}

However, invariant shapes may not be unique.
\begin{ex}\label{non-uniq}
Consider $H\equiv 0$ (the origin) and $F(r,0,\cdot)=r(\cdot)/\int r d\sigma\in C_+(\bS^{n-1})$. Then it is easy to check that 
\begin{align*}
\bar{b}(r)=b(r,0)=F(r, 0, \cdot)/\int r^{n-1}F(r, 0, \cdot)d\sigma= n^{-1}r/Leb(r). 
\end{align*}
Since this choice of $F$ and $H$ satisfies Assumption (I), by Proposition \ref{inverse}, any $r\in C_+(\bS^{n-1})$ is invariant for \eqref{b-bar}, and not attractive except when starting from itself.
\end{ex}

We provide sufficient condition for the centered Euclidean ball $\B$ to be attractive for \eqref{b-bar}, where we denote henceforth by $\B$ the constant $1$ function on $\bS^{n-1}$. Unfortunately, the condition \eqref{suff-attr} is rather hard to check.

\begin{ppn}\label{ppn:attr-sol}
Suppose the \abbr{ODE} \eqref{b-bar} has $C^1(\bS^{n-1})$-solution $(r_t)_{t\ge0}$ for any  $r_0\in\cC\subset C^1(\bS^{n-1})$, and that for any $r\in C^1(\bS^{n-1})$, it holds 
\begin{align}\label{suff-attr}
\ovl{b}(r)(\argmax_\theta r)\le \ovl{b}(r)(\argmin_\theta r)\, .
\end{align}
Then $\B$ is attractive for \eqref{b-bar} for the collection $\cC$ of initial data.
\end{ppn}

\begin{proof}
Set $\mathsf{osc}(r)=\max_\theta r(\theta)-\min_\theta r(\theta)$. Since $r_t$ is $C^1$ for all $t\ge 0$, we have 
\begin{align*}
\frac{d}{dt}\{r_t(\argmax_\theta r_t)\}&=\ovl{b}(r_t)(\argmax_\theta r_t)+\frac{d}{dz}r_t(z)\Big|_{z=\argmax_\theta r_t}\cdot\frac{d}{dt}\{\argmax_\theta r_t\} \\
&=\ovl{b}(r_t)(\argmax_\theta r_t)\, ,
\end{align*}
and similarly for $\argmin_\theta r_t$. Therefore, combined with \eqref{suff-attr} we have that
\begin{align*}
\frac{d}{dt}\mathsf{osc}(r_t)=\ovl{b}(r_t)(\argmax_\theta r_t)-\ovl{b}(r_t)(\argmin_\theta r_t)\le 0.
\end{align*}
Set $\mathbf{r}_t:=r_t/(Leb(r_0)+t)^{1/n}$. Then we have that $\mathsf{osc}(\mathbf{r}_t)=\mathsf{osc}(r_t)/(Leb(r_0)+t)^{1/n}$ and
\begin{align*}
\frac{d}{dt}\mathsf{osc}(\mathbf{r}_t)&=\frac{1}{(Leb(r_0)+t)^{1/n}}\frac{d}{dt}\mathsf{osc}(r_t)-\frac{1}{n(Leb(r_0)+t)^{1+1/n}}\mathsf{osc}(r_t)\\
&\le -\frac{1}{n(Leb(r_0)+t)}\mathsf{osc}(\mathbf{r}_t).
\end{align*}
This yields 
\[
\mathsf{osc}(\mathbf{r}_t)\le \mathsf{osc}(\mathbf{r}_0)\left(\frac{Leb(r_0)+t}{Leb(r_0)}\right)^{-1/n}\to 0  \, , 
\]
as $t\to\infty$, for any $r_0\in\cC$. Equivalently, for some constant $c_n$ such that $Leb(c_n\B)=1$,
\[
||\mathbf{r}_t-c_n\B||_2\le \omega_n^{1/2}||\mathbf{r}_t-c_n\B||_\infty \to 0.
\]
This is exactly the definition \eqref{cond:attr} of attractive shapes with $\psi=\B$. 
\end{proof}

\section{Applications: Proof of Theorem \ref{harm:measure}}
\label{sec:app}

We consider the two applications of Theorems \ref{thm1} and \ref{thm2}, 
introduced previously in Theorem \ref{harm:measure}, with the main one being a simplified model for the growth of the range of \abbr{OERW} (with $F(r,x,\cdot)$ the density of the harmonic measure). Indeed, our choice of \eqref{def-conv-H} is motivated by basic features of \abbr{ORRW} and \abbr{OERW} in the mesoscopic scale. The ideal choice of $H$ to be closer to these models would be 
$H(r,z, \ep)=(r(z)-\ep^{1/n})z$, with \eqref{def-conv-H} an independent of $\ep$,
rule of the same type. Similarly, our choice of $F$ for \eqref{def-conv-F} corresponds to taking a 
simplified continuous model, where the random walk is replaced by a Brownian motion (see the
probabilistic interpretation provided in Section \ref{smooth-harmonic}).
An advantage of using such continuous model is that our proofs work verbatim 
when instead of Brownian motion, the particle follows an elliptic diffusion whose 
generator is a uniformly elliptic second-order divergence form operator $\cL=-\text{div}A\nabla$ (so
the Green's function used in the definition \eqref{def-conv-F} be the one for $\cL$). Indeed, 
recall
Dahlberg's theorem \cite[Theorem 3 and remark]{Da}, that for a Lipschitz domain $D\subset\R^n$, harmonic measures from any point $x\in D$ are mutually absolutely continuous with respect to the $(n-1)$-dimensional Hausdorff measure on $\partial D$, hence their Radon-Nikodym derivative which is the Poisson kernel $P(D,x,\cdot)$ exists and belongs to $L^2_{\text{loc}}(\partial D)$. If the domain is more regular, so is the Poisson kernel.
Per \cite[page 547]{Je}, if $\partial D$ belongs to H\"older space $C^{k+1,\gamma}$ for some $k\in\N, \gamma\in(0,1)$, then $P(D,x,\cdot)\in C^{k,\gamma}(\partial D)$. Since our domains are star-shaped, by an abuse of terminology we will call $F(r,x,\cdot)$ the Poisson kernel of $r$, if it is a probability density on $\bS^{n-1}$ corresponding to $P(D,x,\cdot)$ with $r=\partial D$ up to a change of variables.

As for the reasoning behind the regularization $r \mapsto \wt{r}= r \star g$ in our applications, note that
even for smooth domains, one cannot expect their Poisson kernel to be Lipschitz in $L^2$-norm with respect to boundary perturbations as \eqref{lip-F}, or in any other norm. Indeed, as explained in \cite{Je}, one expects the regularity of $P(D,x,\cdot)$ to be one derivative order less than that of the domain $D$. However, if one forms the kernel based on a regularized domain, then the Lipschitz property can be true (as shown below in Proposition \ref{ppn:local-lip}). Though to  a lesser degree, similar issue 
arises also in the more explicit hitting rule $F$ of \eqref{rule:distance}, where regularization is 
still the key to verifying Assumption (C) via condition \eqref{F-ubd-Lip}.

Recall the $C^k(\Omega)$ and $C^{k,\gamma}(\Omega)$, $\gamma \in (0,1]$
norms of functions, taking in this article only $\Omega=\bS^{n-1}$
or $\Omega=\ovl D$ (the closure of a bounded domain $D \subset\R^n$), and using $C(\Omega)$ for $C^0(\Omega)$.
Letting $\partial^\alpha u$ for multi-index $\alpha$ denote any $|\alpha|$-th order 
derivative of $u$,  we equip the collection of 
$k$-times continuously differentiable functions, with the norm
\begin{align*}
||u||_{C^k(\Omega)}:=\sum_{|\alpha|\le k}\sup_{x\in \Omega}|\partial^\alpha u(x)|.
\end{align*}
Further denoting the $\gamma$-H\"older semi-norm of a function $w$ by
\begin{align*}
[w]_{C^{0,\gamma}(\Omega)}:=\sup_{x\neq y\in \Omega}\frac{|w(x)-w(y)|}{\mathsf d(x,y)^\gamma}\, ,
\end{align*}
if finite, where $\mathsf d(\cdot, \cdot)$ is the geodesic distance when $\Omega=\bS^{n-1}$, and the Euclidean distance when $\Omega=\ovl D$, we define the $C^{k, \gamma}(\Omega)$-H\"older norm 
of any $u \in C^k(\Omega)$, by
\begin{align*}
||u||_{C^{k, \gamma}(\Omega)}:=||u||_{C^k(\Omega)}+\sum_{|\alpha|=k}[\partial^\alpha u]_{C^{0, \gamma}(\Omega)}\, .
\end{align*}
We proceed to verify the conditions needed for
Theorems \ref{thm1} and \ref{thm2}, starting
with the following Lipschitz control on the regularization map $r \mapsto \wt{r}$, the proof of which is deferred to the appendix.
\begin{lem}\label{molify}
For any $r,r'\in L^2(\bS^{n-1})$, we have that
\begin{align}\label{C3L2}
\|\wt{r}-\wt{r}'\|_{C^3(\bS^{n-1})}\le C\|r-r'\|_2 \, ,
\end{align}
for some $C=C(g)<\infty$ that depends only on the convolution kernel $g$. In particular, for any
$a \in (0,1)$ there exists $\delta=\delta(a,\alpha,g) > 0$ so the image of $\cAe_2(a)$ under  
$(r,x) \mapsto (\wt{r},x)$ is within the set
\begin{equation}\label{def:K-a-c-delta}
 \cK_{a,C,\delta} :=
\{ (\wt{r},x) :  \, \inf_\theta \{ \wt{r}(\theta)\} \ge a \,, \; 
\| \wt{r} \|_{C^3(\bS^{n-1})} \le C a^{-1} \,, \, \ovl{\B}(x,5\delta)  \subseteq   D_{\wt{r}} \} \,,
\end{equation}
where $\B(x,\delta)$ denotes the open Euclidean ball of radius $\delta$ centered at $x \in \R^n$ 
and $D_{\wt{r}}$ denotes the domain of boundary $\wt{r}$.
\end{lem}

Equipped with Lemma \ref{molify}, we prove 
property \eqref{lip-H} for $H$ of \eqref{def-conv-H}.
\begin{ppn}\label{ppn:lip-H}
For every $a\in(0,1)$, the map $(r,z)\mapsto H(r,z)$ is Lipschitz from $(\cA_2(a)\times \bS^{n-1}, 
\|\cdot\|_2\times|\cdot|)$ to $\R^n$.
\end{ppn}

\begin{proof} With $\alpha(\cdot,\cdot)$ Lipschitz on compacts, we have
by Lemma \ref{molify} that for $H$ of \eqref{def-conv-H},
some finite $\wt{C}=\wt{C}(a,\alpha)$ and all $r,r'\in\cA_2(a)$, 
\begin{align*}
|H(r,z)-H(r',z)| & =|\alpha(\wt r(z),z) -\alpha(\wt{r}'(z),z)|\\
&\le \wt{C} |\wt r(z) -\wt{r}'(z)|\le \wt{C} \, C(g) \|r-r'\|_2 \, .
\end{align*}
Also, for any $r \in \cA_2(a)$, $z,z'\in \bS^{n-1}$, 
\begin{align*}
|H(r,z)-H(r,z')|\le \wt{r}(z) |z-z'|+|\alpha(\wt{r}(z),z)-\alpha(\wt{r}(z'),z')| \,.
\end{align*}
The uniform on $\cA_2(a)$ control $\|\wt{r}\|_{C^3(\bS^{n-1})} \le C(g) a^{-1}$ 
from Lemma \ref{molify}, translates into the same control over the bounded Lipschitz
norm of $\wt{r}(\cdot)$. Hence, with 
$\alpha(\cdot,\cdot)$ Lipschitz on compacts, we deduce that $|H(r,z)-H(r,z')|\le C |z-z'|$ for some finite $C=C(a,\alpha,g)$ and all $r,z,z'$ as above.
\end{proof}

\subsection{Smoothed harmonic measure}\label{smooth-harmonic}
Recall the construction  above \eqref{def-conv-F} of the smoothed domain $\wt r\in C^3(\bS^{n-1})$ for every $r\in L^2(\bS^{n-1})$. Due to the preceding discussion on Poisson kernels, it is clear that the regularized (as in \eqref{def-conv-F}) Poisson kernel $F(r,x,\cdot)$ belongs to $C_+(\bS^{n-1})$ for any $(r,x)\in\cD(F)$, 
where the probabilistic meaning of the definitions of $F$ in \eqref{def-conv-F} and $H$  in 
\eqref{def-conv-H} is as follows.

If the process $(R^\ep,x^\ep)$ is defined up to time $s$ and the state at that time is given by domain with boundary $R^\ep_s$ and particle position $x^\ep_s$, we wait for the next jump mark, 
denoted by $t>s$, and given by an 
independent Exponential($\ep^{-1}$) random variable. To choose a point at the current boundary $R^\ep_s$, the particle follows the law of a Brownian motion in $\R^n$, starting at $x^\ep_s$ 
till its first exit from the smoothed domain $\wt R^\ep_s = R^\ep_s\star g$. 
We record its exit angle $\xi_t\in \bS^{n-1}$ and define the location for the center of the new bump on the {\it original} domain by $R^\ep_s(\xi_t)\xi_t$. Hence the updated domain is formed by
\[
R^\ep_t(\theta)=R^\ep_s(\theta) + \frac{\ep}{y_{R^\ep_s,x^\ep_s}}  g_{\eta(\ep, R^\ep_s,x^\ep_s)}(\langle \xi_{t}, \theta\rangle).
\]
Observe that the bump is added to the original domain and not the smoothed one. Next, the particle is pushed towards the origin by a strictly positive quantity, along the radius, still in the {\it smoothed} domain $\wt R^\ep_s$, namely $x^\ep_t=\alpha(\wt R^\ep_s(\xi_t), \xi_t)\xi_t$ and there it waits for the next jump mark. Continuing in this way we define the process at all times. Note that we have 
omitted the travel time of the Brownian motion inside the smoothed domain and only deal with its exit distribution.
Further, since $R^\ep_t(\theta) \ge R^\ep_s(\theta)$, $t>s$ for all $\theta\in \bS^{n-1}$, necessarily also $\wt R^\ep_t(\theta) \ge \wt R^\ep_s(\theta)$, with the particle $(x^\ep_t)$ always contained in the smoothed domain, once we assume it is the case for $(\wt {R}_0^\ep, x^\ep_0)$. 


We postpone to the appendix the proof of 
the following key proposition about Lipschitz regularity for Poisson kernels in our regularized domains 
(where the precise value of $\gamma\in(0,1)$ is unimportant).
\begin{ppn}\label{ppn:local-lip}
For any $a \in  (0,1)$, $C<\infty$, $\delta>0$, the 
map $(\wt r,x) \mapsto F(r,x,\cdot)$ of \eqref{def-conv-F} is Lipschitz from 
the set $\cK_{a,C,\delta}$ of \eqref{def:K-a-c-delta}, equipped with
the $C^{2,\gamma}(\bS^{n-1})$-norm for the first variable and the Euclidean norm for the second one, 
to $C^1 (\bS^{n-1})$. 
\end{ppn}

The map $(r,x)\mapsto F(r,x,\cdot)$ of \eqref{def-conv-F} is a composition of 
$(r,x) \mapsto (\wt r,x)$ and $(\wt r,x) \mapsto F(r,x, \cdot)$. The former map is
globally Lipschitz per Lemma \ref{molify}, whereas the latter is per Proposition \ref{ppn:local-lip}
globally Lipschitz on some $\cK_{a,C,\delta}$ that contains the 
image of $\cAe_2(a)$ under the first map, yielding the following corollary.
\begin{cor}\label{cor:F-lip}
For every $a\in(0,1)$, the map $(r,x) \mapsto F(r,x,\cdot)$ is globally Lipschitz from  $(\cAe_2 (a),
\|\cdot\|_2 \times|\cdot|)$ to $C^1(\bS^{n-1})$.
\end{cor}

\begin{rmk}\label{verify-C} From Corollary \ref{cor:F-lip} we get that  
$\|F(r,x,\cdot)\|_{\rm Lip} \le C(a)$ for all $(r,x)\in\cAe_\infty(a)$,
hence Assumption (C) holds in our setting (see Proposition \ref{ppn:implic}(d)).
\end{rmk}

We next show that $F(\cdot)$ of \eqref{def-conv-F} is bounded above and below, uniformly in
$\cAe_2(a) \times \bS^{n-1}$, thereby verifying \eqref{F-ubd} and as explained 
in Remark \ref{st-remov}, allowing us also to dispense 
of the stopping time $\sigma^\ep(\delta)$ in Theorem \ref{thm1}.
\begin{ppn}\label{ex:erg}
For each $a\in (0,1)$ there exist  $c=c(a, \alpha, g)>0$ such that 
\begin{align}\label{prob-doblin}
c \le \inf_{\cAe_2(a)\times\bS^{n-1}} \{F(r,x,\theta)\} \le \sup_{\cAe_2(a)\times\bS^{n-1}} 
\{F(r,x,\theta)\} \le c^{-1} \, .
\end{align}
\end{ppn}
\begin{proof}
For each $r \in \cA_2(a)$, the state space of $(x^{1,r}_t)_{t\ge0}$ is contained in the star-shaped domain enclosed by $\wt r\in C^3(\bS^{n-1})$. In such domains, with $x$ in the interior, the Poisson kernel $F(r,x,\cdot)$ is pointwise positive (a consequence of Hopf lemma). By Proposition \ref{ppn:local-lip},
$(\wt r, x)\mapsto F(r, x, \theta)$ is continuous per fixed $\theta$ with 
$\theta\mapsto F(r, x, \theta)$ continuous per fixed $(\wt r, x) \in \cK_{a,C,\delta}$. 
Thanks to the compactness of $\bS^{n-1}$,  the joint continuity of 
$(\wt r, x, \theta)\mapsto F(r, x, \theta)$ follows and
we get \eqref{prob-doblin} from Lemma \ref{molify} and 
the compactness of $\cK_{a,C,\delta} \times \bS^{n-1}$ under the 
$C^{2, \gamma}(\bS^{n-1})\times \R^n\times\bS^{n-1}$ norm 
(see proof of Proposition \ref{ppn:local-lip}).
\end{proof}
\begin{rmk}\label{verify-e}
Propostion \ref{ex:erg} further establishes Assumption (E) for our model. Indeed, the sufficient condition \eqref{minor} of Proposition \ref{ppn:implic}(c), 
then holds with $n_0=1$, $\delta=c$ and $m_r(\cdot)$ the push-forward under $H(r,\cdot)$ 
of the uniform measure on $\bS^{n-1}$. 
\end{rmk}

\begin{rmk}\label{rmk-b}
With $\omega_n^{-1/2}\|\cdot\|_2 \le \|\cdot\|_\infty \le \|\cdot\|_{C^{1}(\bS^{n-1})}$, hence also 
$\cAe_\infty(a) \subset \cAe_2(a)$, Corollary \ref{cor:F-lip} implies that 
\eqref{lip-F} holds for both $p=2$ and $p=\infty$. Recall Proposition \ref{ppn:lip-H} that 
\eqref{lip-H} holds here, and Proposition \ref{ex:erg} that so do both \eqref{F-ubd} and \eqref{F-lbd}.
Combined, these in turn yield by Proposition \ref{ppn:implic}(a) the last remaining 
Lipschitz condition required (namely, property \eqref{lip-bar} of $\ovl{b}$), 
and thereby complete the verification of Assumption (L) in our setting.
\end{rmk}

Having verified Assumptions (C), (E) and (L) (see Remarks \ref{verify-C}, \ref{verify-e} and 
\ref{rmk-b},  respectively),
we can apply Theorem \ref{thm1} to this model without the stopping time
$\sigma^\ep(\delta)$ of \eqref{pdf-lbd}. This amounts to proving
Theorem \ref{harm:measure}(a) for this model, whereas our next proposition, considering 
special cases where we have explicit descriptions, constitutes the proof 
of Theorem \ref{harm:measure}(b) in this setting.
\begin{ppn}\label{ppn:poisson-ker} $~$ \\
(a) If the function $\alpha(\ell,z)=\alpha(\ell)$ does not separately depend on $z$, then the centered Euclidean ball $\B$ is an invariant solution to \eqref{b-bar}. \\
(b) If the function $\alpha(\ell,z)=\alpha(z)\ell$ depends linearly on $\ell$, then Assumption (I) is satisfied.\\
(c) If $\alpha(\ell,z)=\gamma \ell$ for some fixed number $\gamma\in[0,1)$, then the unique invariant measure $\nu_r$ is explicitly given by the harmonic measure from the origin in the domain enclosed by $\gamma\wt r$, for every $r\in C_+(\bS^{n-1})$.\\
(d) If $\alpha(\ell, z)\equiv 0$ and we force $\wt{r}=r$, then $\B$ is the unique attractive solution 
of \eqref{b-bar}, whenever $r_0\in \cC=C_+^1(\bS^{n-1})$.
\end{ppn}
\begin{proof}
{\em (a)}. First note that $\B\star g=\B$, so regularization by $g$ has no effect here. Further, by the 
rotational invariance of the Brownian law, the uniform measure on $\alpha(1) \bS^{n-1}$ 
is invariant for the process $(x_t^{1,\B})$ (i.e. a Brownian motion on $\B$ starting at 
$\alpha(1) \bS^{n-1}$ and radially projected back to that set upon hitting $\partial \B = \bS^{n-1}$). In
addition, recall \eqref{def:y} that $y_{\B,x}=\omega_n$ is constant, hence so is
$\theta\mapsto \ovl{b}(\theta)$, from which it directly follows that $\B$ is invariant for \eqref{b-bar}. 

{\em (b)}. The identity \eqref{scaling-inv} is due to the scaling invariance of the Brownian motion, while \eqref{scaling-H} is satisfied by our choice.

{\em (c)}. By the scaling invariance of Brownian motion, the harmonic measure from the origin on $\gamma \wt r$ and on $\wt r$, viewed as measures on spherical angles, are equal. Since the transition kernel of the Brownian motion from $\gamma \wt r$ to $\wt r$ is exactly given by $F(r,x,\cdot)$, we see that the harmonic measure from the origin is the unique (thanks to Assumption (E)), invariant measure for $(x_t^{1,r})$.

{\em (d)}. Since $H\equiv 0$ (the origin), here $\ovl{b}(r)=F(r,0,\cdot)/y_{r,0}$. By Corollary \ref{cor:F-lip}, 
$F(r, 0, \cdot)\in C^1(\bS^{n-1})$. Hence, for any $r_0\in C_+^1(\bS^{n-1})$, the \abbr{ODE} \eqref{b-bar} admits $C^1$ solution $(r_t)_{t\ge0}$. By Proposition \ref{ppn:attr-sol}, for $\B$ to be attractive for \eqref{b-bar} it suffices to show that for any $r\in C^1(\bS^{n-1})$, the Poisson kernel from the origin to $\wt{r}$,
is at angle $\argmin_\theta r$ no smaller than at $\argmax_\theta r$. To this end, 
upon forcing $\wt{r}=r$
consider two standard Brownian motions in $\R^n$, one in the domain enclosed by $r$, the other in the Euclidean ball $\B(0, \min_\theta r)$. Couple them to move together starting from the origin until the first hitting time by both of $\partial\B(0, \min_\theta r)$, where one Brownian motion is stopped and the other can continue to move until hitting $r$. This coupling yields that $F(r, 0, \argmin_\theta r)\ge 1/\omega_n$. An analogous coupling, between a Brownian motion in the domain enclosed by $r$ and another in 
$\B(0, \max_\theta r)$, yields that $F(r, 0, \argmax_\theta r)\le 1/\omega_n$, thus verifying our claim.
\end{proof}

For anisotropic $\alpha(\ell,z)$ (that do not satisfy the condition
of Proposition \ref{ppn:poisson-ker}(a)), one may obtain other limiting shapes as  invariant solutions to the \abbr{ODE} \eqref{b-bar}, such as diamond, square etc (see Figure \ref{fig.shapes.process}), implicitly determined as in Proposition \ref{inverse} (and in the anisotropic case, the Euclidean ball is 
typically not an invariant shape).

\subsection{Distance to particle}

Following the approach of Section \ref{smooth-harmonic}, in particular as in
Remarks \ref{verify-C}, \ref{verify-e} and \ref{rmk-b}, our next proposition establishes 
Theorem \ref{harm:measure}(a) (namely, shows that Theorem \ref{thm1} applies 
without $\sigma^\ep(\delta)$ of \eqref{pdf-lbd}), for $F$ of \eqref{rule:distance}, where the
probability of choosing a boundary point is a fixed function of its distance to the particle.
\begin{ppn}\label{lip-F-dist}
For every $a\in(0,1)$, the map $(r,x)\mapsto F(r,x,\cdot)$ of \eqref{rule:distance}
is Lipschitz from $(\cAe_\infty(a), \|\cdot\|_p\times|\cdot|)$ to $(L^p(\bS^{n-1}), \|\cdot\|_p)$, for both $p=2$ and $p=\infty$. 
In addition, $\|F(r,x,\cdot)\|_{\rm Lip}$ is bounded, uniformly over $\cAe_\infty(a)$, 
while $F(r,x,\theta)$ is bounded above and away from zero, uniformly over $\cAe_2(a) \times \bS^{n-1}$.
\end{ppn}

\begin{proof} Starting with the boundedness above and away from zero 
of $F(r,x,\cdot)$, we have in view of Lemma \ref{molify}
that $\|H(r,\cdot)\|_\infty \le \| \wt r \|_\infty \le C(g) \|r\|_2 
\le C(g) a^{-1}$. The argument of $\varphi(\cdot)$ in 
\eqref{rule:distance} is thus bounded, uniformly over $\cAe_2(a)$. 
Now recall from Lemma \ref{molify} that for some $\delta_\star(a)>0$ 
and \emph{all} $(r,x) \in  \cAe_2(a)$, the point $x$ is at least $\delta_\star$ 
away from $\wt{r}$, and thereby the argument of $\varphi(\cdot)$ in \eqref{rule:distance}
is bounded away from zero uniformly over $\cAe_2(a)$. The uniform over $\cAe_2(a) \times \bS^{n-1}$ 
bound above and away from zero on $F$ is then a direct consequence of our assumption that 
$\varphi(\cdot)$ is positive and continuous on $(0,\infty)$.
 
Next, with the numerator of \eqref{rule:distance} uniformly bounded above 
and the denominator uniformly bounded below, it  suffices to separately 
prove the Lipschitz property for the numerator and denominator in \eqref{rule:distance}.
To this end, note that  $\varphi(|\cdot|)$ is globally Lipschitz on compact subsets of $(0,\infty)$.
Dealing simultaneously with $p=2$ and $p=\infty$, it follows that for some $C=C(\vphi)$ finite
and all $(r,x), (r', x')\in\cAe_\infty(a)$, 
\begin{align*}
 \big \|\varphi(|\wt r(\theta)\theta-x|)-\vphi(|\wt r'(\theta)\theta-x'|)\big \|_p
& \le C \big \|\left(\wt r(\theta)\theta- \wt r'(\theta)\theta\right)-(x-x')\big \|_p \\
\le C\big(\|\wt r-\wt r'\|_p+ \omega_n^{1/p} |x-x'|\big) 
&\le  C \big(\|r-r'\|_p+ \omega_n^{1/p} |x-x'|\big),
\end{align*}
where 
all $L^p$-norms are with respect to $\theta$. 
The difference between the denominator values at $(r,x)$ and $(r',x')$ is at most the
$L^1$-norm of the difference between the corresponding numerators, hence its
Lipschitz property follows from the preceding bound. 

Finally, note that $\|\wt{r}\|_{\rm Lip} \le \|g'\|_\infty \|r\|_1$ is uniformly bounded on $\cAe_\infty(a)$,
hence the same applies for $\theta \mapsto \wt{r}(\theta)\theta-x$.
With $\vphi(|\cdot|)$ Lipschitz on compacts, it thus 
follows that $\|F(r,x,\cdot)\|_{\rm Lip}$ is uniformly bounded on $\cAe_\infty(a)$.
\end{proof}

We turn to get the analog of Proposition \ref{ppn:poisson-ker} for this model, thereby 
establishing Theorem \ref{harm:measure}(b) in this setting.
\begin{ppn}\label{ppn:dist-rule-inv} $~$ \\
(a) If the function $\alpha(\ell,z)=\alpha(\ell)$ does not separately depend on $z$, then the centered Euclidean ball $\B$ is an invariant solution to \eqref{b-bar}. \\
(b) Assumption (I) holds whenever $\alpha(\ell,z)=\alpha(z)\ell$,
$\vphi(t)=t^\beta$ for some $\beta \in \R$.\\
(c) If $\alpha(\ell, z)\equiv 0$, $\beta<0$ and we force $\wt{r}=r$, then $\B$ is the unique attractive solution 
of \eqref{b-bar}, whenever $r_0\in C_+^1(\bS^{n-1})$.
\end{ppn}
\begin{proof} {\em (a)}. Following the proof of Proposition \ref{ppn:poisson-ker}(a), 
the invariance of  $|\theta - \alpha(1) z|$ to a \emph{common rotation} of 
$(z,\theta) \in \bS^{n-1} \times \bS^{n-1}$ yields that the uniform measure on $\alpha (1) \bS^{n-1}$
is invariant for $(x_t^{1,\B})$. 
   
{\em  (b).} It is easy to check that having $\vphi(t)=t^{ \beta}$ in \eqref{rule:distance} results with
\eqref{scaling-inv} holding, whereas \eqref{scaling-H} is satisfied by the assumed linearity in $\ell$
of $\alpha(\ell,z)$. Thus,  Assumption (I) holds in this case and from part (a) we know that when
$\alpha(\ell,z) \equiv 0$ the ball $\B$ is an invariant solution  to \eqref{b-bar}. 

{\em (c)}. For $H\equiv 0$ and $\vphi(t)=t^\beta$, one has that 
\begin{equation}\label{def:b-bar-H-0}
\ovl{b}(r)(\theta) =y_{r,0}^{-1} F(r,0,\theta) = \|\wt{r}\|_{q}^{-q} \; \wt{r}(\theta)^\beta \,,
\quad \mbox{ for } \quad q:= n-1+\beta \,.
\end{equation}
In particular, for any $r_0\in C_+^1(\bS^{n-1})$, the \abbr{ODE} \eqref{b-bar} with $\ovl{b}(r)$ of
\eqref{def:b-bar-H-0} admits $C^1$ solution $(r_t)_{t\ge0}$ and in view of Proposition \ref{ppn:attr-sol}, 
for $\B$ to be attractive for \eqref{b-bar} when $\beta<0$,
it suffices to have along such solution $r_t$, the regularized 
$\wt{r}_t$ no smaller at angle $\argmax_\theta r_t$ than at 
$\argmin_\theta r_t$. This  indeed holds once we force $\wt{r}=r$.
\end{proof}

        

\appendix
\section{}

\noindent
\begin{proof}[Proof of Proposition \ref{ppn:implic}] $~$
\\ {\em (a)} Since $b(r,x)$ is a ratio of $F(r,x,\cdot)$ and the scalar $y_{r,x}$, with  $y_{r,x}\ge \omega_n a^{n-1}$ on $\cA_\infty(a)$ and 
$F$ assumed Lipschitz in $L^p$-norm on this set (by \eqref{lip-F}), it suffices to show that in addition 
$(r,x)\mapsto y_{r,x}$ is Lipschitz in that norm (on $\cA_\infty(a)$). 
To this end, note that for any $(r,x), (r', x')\in \cAe_\infty(a)$ and $1/p+1/q=1$,

\begin{align*}
\omega_n^{-1} |y_{r,x}-y_{r', x'}|\le \int_{\bS^{n-1}}&\left|r^{n-1}(\theta)F(r,x,\theta)-(r')^{n-1}(\theta)F(r',x',\theta)\right|d\sigma(\theta)\\
\le \Big[\int_{\bS^{n-1}}&|r^{n-1}-(r')^{n-1}|(\theta)F(r,x,\theta)d\sigma(\theta) \\
&+\int_{\bS^{n-1}}(r')^{n-1}(\theta)|F(r,x,\theta)-F(r',x',\theta)|d\sigma(\theta)\Big]\\
\le n  a^{-(n-2)} \|r-&r'\|_p\|F(r,x,\cdot)\|_q+ a^{-(n-1)} \|F(r,x,\cdot)-F(r',x',\cdot)\|_1\, ,
\end{align*}
by H\"older's inequality and expanding $|r^{n-1}-(r')^{n-1}|$. Considering either $p=q=2$, or 
$p=\infty$, $q=1$,  our claim follows from \eqref{lip-F} at $p$, since 
$p \mapsto \omega_n^{-1/p} \|\cdot\|_p$ is non-decreasing for functions on $\bS^{n-1}$.

Turning next to $h$, for any $(r,x), (r', x')\in\cAe_\infty(a)$, by \eqref{lip-F}, \eqref{lip-H},
\begin{align*}
|h(r,x)-h(r',x')|- |x-x'|&\le \int_{\bS^{n-1}}\left|H(r,\theta)F(r,x,\theta)-H(r',\theta)F(r',x',\theta)\right|d\sigma(\theta)  \\
\le \int_{\bS^{n-1}}|H(r,\theta)(F(r,x,\theta)&-F(r',x',\theta))|d\sigma(\theta)\\
&+\int_{\bS^{n-1}}|H(r,\theta)-H(r',\theta)|F(r',x',\theta)d\sigma(\theta)\\
\le \|H(r,\cdot)\|_2\|F(r,x,\cdot)-&F(r',x',\cdot)\|_2+K\|r-r'\|_2\le C(a)(||r-r'||_2+|x-x'|) 
\end{align*}
as needed. 

Fixing $a \in (0,1)$, and dealing simultaneously with $p=2$ and $p=\infty$,  
we proceed to show that $r\mapsto \ovl{b}(r)$
is Lipschitz from $(\cA_\infty(a), \|\cdot\|_p)$ 
to $(L^p(\bS^{n-1}), \|\cdot\|_p)$. To this end note that 
the Markov chain on $\bS^{n-1}$ of transition kernel density 
$\mathsf{K}_r(\xi,\theta) = F(r, H(r,\xi),\theta)$ with respect to the uniform measure, 
is, by \eqref{F-lbd}, uniformly ergodic throughout $r\in\cA_2(a)$, 
with ergodicity coefficient depending only on $a$ (via $\udl{F}(a)$), and as such, 
has a unique invariant measure $\mu_r$ on $\bS^{n-1}$. Further,
as the rule $H$ is non-random, the invariant measure $\nu_r$ of the Markov chain 
$\{x^{1,r}_{T_i}\}$ is merely the push-forward by $\xi \mapsto H(r,\xi)$ of 
the aforementioned $\mu_r$ with $\ovl{b}$ of \eqref{b-bar} represented alternatively as
\[
\ovl{b} (r) 
= \int b_\star(r,\xi) d\mu_r(\xi) \,, \quad \mbox{ for } \quad b_\star(r,\xi) := b(r,H(r,\xi)) \,.
\]
Next,  recall the preceding where we saw that \eqref{lip-F} results with \eqref{lip-b} holding
at the corresponding value of $p$. Combining the latter with \eqref{lip-H},
we deduce that the analogous to \eqref{lip-b}, uniform over $\cAe_\infty(a)$ Lipschitz bounds, 
apply also for $b_\star(\cdot,\cdot)$. Thus, 
utilizing the characterization of total variation norm of 
finite signed measures (cf. \cite[page 124]{Ha}),
we have for some finite constant $C=C(a)$ and all $r,r'\in\cA_\infty(a)$, 
\begin{align*}
\|\ovl{b}(r)-\ovl{b}(r')\|_p
& \le \big \|\int b_\star (r,\xi) d (\mu_{r} - \mu_{r'})(\xi)\big \|_p + \int \|b_\star (r,\xi) - b_\star(r',\xi)\|_p 
d\mu_{r'}(\xi) \nonumber \\
&\le 2\sup_{\cA_\infty(a) \times \bS^{n-1}} \{ \| b_\star (r,\xi,\cdot)\|_\infty \} \, \|\mu_{r} - \mu_{r'}\|_{{\rm TV}} 
 + C(a)||r-r'||_p\, .  
\end{align*}
Turning to the first term, recall that $b_\star(r,\cdot) \le a^{-(n-1)} F(r,\cdot)$ throughout $\cA_\infty(a)$, so 
in view of \eqref{F-ubd} it suffices to bound 
$\|\mu_{r} - \mu_{r'}\|_{{\rm TV}}$.
Denoting by
$\mathcal M_1$ the space of signed Borel measures on $\bS^{n-1}$ with total variation one, we find that (per notation in \cite[(2.1)]{Mi}), for some $K=K(a)$ and $C=C(a)$,
\begin{align*}
\| \mathsf{K}_r - \mathsf{K}_{r'}\|_{\rm op} & := 
 \frac{1}{2}\sup_{\mu \in \mathcal M_1 }\sup_{\{g 
:\, |g|\le 1\}}\Big|\int_{\bS^{n-1}}\int  g(\theta) (\mathsf{K}_r
(\xi,\theta)- \mathsf{K}_{r'} (\xi,\theta)
) d\mu(\xi)  d\sigma(\theta)   \Big|\\
& \le  \frac{\omega_n}{2} \sup_{\xi \in \bS^{n-1}} \Big\{
\| F(r, H(r,\xi), \cdot) - 
F(r', H(r',\xi), \cdot)\|_p \Big\}\\  
& \le K\frac{\omega_n}{2} \big(\|r-r'\|_p +\sup_{\xi \in \bS^{n-1}}|H(r,\xi)-H(r',\xi)| \big)\\
&\le  K\frac{\omega_n}{2} (\|r-r'\|_p+ C\| r-r'\|_2) \le C\|r-r'\|_p
\end{align*}
(using in the above, also \eqref{lip-F} and \eqref{lip-H}).
Consequently, by \cite[Corollary 3.1]{Mi}, for some 
$\wt C=\wt{C} (a)$ which depends 
on the uniform ergodicity coefficient,
\[
 \|\mu_{r} - \mu_{r'}\|_{{\rm TV}} \le  \wt C \| \mathsf{K}_r - \mathsf{K}_{r'}\|_{\rm op}  \le 
\tilde{C} \, C \|r-r'\|_p\,,
\]
thereby completing the proof.

\noindent
{\em (b)} Consider the Banach space $(C(\bS^{n-1}), \|\cdot\|_\infty)$, and the subset $\cA_\infty(a)$ for some $a<\inf_\theta r_0(\theta)\wedge \|r_0\|^{-1}_\infty$.
Since per \eqref{lip-bar}, the map $r\mapsto\ovl b(r)$ is Lipschitz from $L_\infty$ to $L_\infty$ in $\mathcal A_\infty(a)$, there exists a unique continuous solution $(r_t)$ to \eqref{b-bar} locally in time, defined up to the first exit time of the set $\mathcal A_\infty(a)$, cf. \cite[Theorem 7.3]{Am}. 
Hence, Proposition \ref{vol-growth} holds as long as the solution is defined, and we can bound the growth of its $L^2$-norm by H\"older's inequality
\[
\|r_t \|_2 \le \omega_n^{\frac{n-2}{2n}}\|r_t\|_n 
\le \omega_n^{\frac{n-2}{2n}}n^{\frac{1}{n}}(Leb( r_0 )+t)^{\frac{1}{n}}=:p(t) \,.
\]
This in turn implies the following bound on the growth of $F(r_t, \cdot, \cdot)$, 
\[
\sup_{\substack{x\in \text{Image}(H(r_t, \cdot))\\ \theta\in \bS^{n-1}}} \{F(r_t, x, \theta)\} \le \ovl F(p(t))\,,
\]
in terms of $\ovl F(p(t))$ of \eqref{F-ubd}, thereby leading to
\[
\sup_{\theta\in\bS^{n-1}}\bar{b}(r_t)(\theta)\le \ovl F(p(t))(\inf_\theta r_0(\theta))^{-(n-1)}.
\]
Thus, by the \abbr{ODE} \eqref{b-bar}, as long as the solution is defined,
we have control on the growth of its $L_\infty$-norm, 
\[
\|r_t\|_\infty\le \|r_0\|_\infty+\int_0^t \ovl F(p(s))) ds (\inf_\theta r_0(\theta))^{-(n-1)}=:q(t).
\]
Given any $T<\infty$, taking in the beginning $a<q(T)^{-1}\wedge \inf_\theta r_0(\theta)$, ensures the existence and uniqueness of a continuous solution up to time $T$.

\noindent
{\em (c)} 
The minorisation \eqref{minor} implies by standard theory 
of general state space Markov chains (see \cite[Theorem 8]{RR}), that for 
any $r \in C_+(\bS^{n-1})$ the embedded chain $\{x_{T_i}^{1,r}\}$ has a  
unique invariant measure $\nu_r$, with the uniform on $\cAe_\infty(a)$
convergence 
\[
\sup_{(r,x) \in \cAe_2(a)} \{ \|\mathsf{P}_r^{n}(x,\cdot)- \nu_r(\cdot)\|_{\text{\rm TV}} \} \le (1-\delta)^{\lfloor n/n_0 \rfloor}  \, .
\]
The proof is by coupling, which extends to the 
process $(x^{1,r}_t)_{t\ge0}$ with $x^{1,r}_0=x$ 
and its stationary version 
$(\mathsf{x}^{1,r}_t)_{t\ge 0}$ (i.e. starting at distribution $\nu_r$
and using the same jump times $\{T_i\}$ for both processes). It follows that 
the processes coalesce at the coupling time $\cT_x$ with
\[
\|\P_x(x^{1,r}_t \in \cdot) - \nu_r(\cdot)\|_{\text{\rm TV}} \le \P_x(x^{1,r}_t \ne \mathsf{x}^{1,r}_{t}) = \P(\cT_x >t) \le e^{-ct} , 
\]
for some positive constant $c=c(\delta, n_0)=c(a)$, 
any $t \ge 0$ and all $(r,x) \in \cAe_\infty(a)$. 
By the triangle inequality, employing this coupling 
for proving \eqref{rate}, we separately bound 
\begin{align}\label{bd-1}
\sup_{t_0\ge 0}\E\Big \|\frac{1}{t}\int_{t_0}^{t_0+t}[b(r,x^{1,r}_s)-b(r,\mathsf{x}^{1,r}_s)]ds \Big \|_2^2
\end{align}
and 
\begin{align}\label{bd-2}
\sup_{t_0\ge 0}\E\Big \|\frac{1}{t}\int_{t_0}^{t_0+t}[b(r,\mathsf{x}^{1,r}_s)-\ovl{b}(r)]ds \Big \|_2^2 \,.
\end{align}
There is no contribution to \eqref{bd-1} from $s \ge \cT_x$ and a-priori
$\|b(r,x)\|_2 \le C(a)<\infty$  
for all $(r,x) \in \cAe_\infty(a)$. Hence \eqref{bd-1} is at most
$4 C(a)^2 \E \cT_x/t \le 4 C(a)^2/(c t)$. By stationarity
the expectation in \eqref{bd-2} is independent of $t_0$ and utilizing 
the Markov property, it equals
\begin{equation}\label{bd-3}
\frac{2}{t} \int d\nu_r(x) \int_{0}^{t} \big(1-\frac{u}{t}\big) \Delta_{r,x} (u) du \,,
\end{equation}
where by Fubini
\[
\Delta_{r,x}(u) :=  \E_x \Big[ \int_{\bS^{n-1}} 
[ b(r,x^{1,r}_u)(\theta) - \ovl{b}(r)(\theta) ] 
b(r,x)(\theta) 
 d\sigma(\theta) \Big]
\,.
\]
Using the preceding coupling per value of $x$ in \eqref{bd-3}, we
deduce that  
\[
|\Delta_{r,x}(u)| \le \Gamma_{r,x} \, \P(\cT_x > u) \,,
\]
where by Cauchy-Schwarz, for the pre-compact 
$\cK_r = \{x\in\R^n \colon (r,x) \in \mathcal D(F), \, x \in {\rm Image}(H(r,\cdot)) \}$,
\begin{align*}
\Gamma_{r,x} := &
\sup_{y,y' \in \cK_r} \;
\int_{\bS^{n-1}} 
| b(r,y)(\theta) - b(r,y')(\theta) |
b(r,x)(\theta) 
 d\sigma(\theta) \\ 
 &\le 2 \sup_{x' \in \cK_r} \|b(r,x')\|_2^2
\le 2 C(a)^2 \,.
\end{align*}
Plugging into \eqref{bd-3} this uniform bound on $\Gamma_{r,x}$ 
and the uniform tail bound on $\cT_x$, bounds the term \eqref{bd-2} by
$4 C(a)^2/(ct)$, thereby completing the proof.

\noindent
{\em (d)} The processes $(R^\ep_{t\wedge \zeta^\ep},x^\ep_{t \wedge \zeta^\ep})$
are, for some $a=a(r_0,\delta)$, within $\cAe_\infty(a)$, whereby $y_{r,x} \ge \omega_n a^{n-1}$. 
Thus, given \eqref{eta}, the definition of $b(r,x)$ and the bound \eqref{approx-id} on 
$\| f \star g_\eta - f \|_2$,
it suffices for Assumption (C) to bound the rate of convergence to zero as $t \to 1$, 
for the \abbr{rhs} of \eqref{approx-id} at $f=F(r,x,\cdot)$, uniformly over $(r,x) \in \cAe_\infty(a)$.
In particular, recall from \eqref{def:Tt} that $\|T_t f-f\|_2 \le 2 \|f\|_{\rm Lip} (1-t)^{1/2}$ (since
$\|x-y\|_2 = \sqrt{2(1-\langle x,y \rangle)}$ when $x,y \in \bS^{n-1}$), hence
\eqref{F-ubd-Lip} suffices for Assumption (C) to hold.

\noindent
{\em(e)} By our construction and \eqref{eta}, any bump added to some $R^\ep_t$ 
is supported on a spherical cap in $\bS^{n-1}$ whose radius is at most 
\begin{align}\label{eta-ubd}
2 \eta(\ep, R^\ep_t, x) \le 2 \omega_n^{-1/(n-1)} \ep^{1/n}  (\inf_\theta r_0(\theta))^{-1}=: \etas \,.
\end{align}
Proceeding to cover $\bS^{n-1}$ by spherical caps $\{S_j\}_{j=1}^{N_\ep}$ of radius $\etas$, 
in view of \eqref{eta-ubd} the growth of $R^\ep_{T^\ep_i}$ 
\emph{somewhere} within $S_j$ requires that $x^\ep_{T^\ep_i}$ hit the concentric cap  
of radius $2 \etas$. 
Further, for $\kappa^\ep$ of \eqref{kappa-ep} and 
$\ga=\ga(r_0,T):=(1+\|r_T\|_2)^{-1}\wedge \inf_\theta r_0(\theta)$, 
\begin{align}\label{A2-bd}
(R^\ep_{t\wedge\kappa^\ep})_{t\in[0, T]}\in \cA_2(\ga)\,.
\end{align}
Thus, for any $T^\ep_i \le \kappa^\ep$, the probability of $x^\ep_{T^\ep_i}$ hitting a spherical cap
of radius $2 \etas$, is by  \eqref{F-ubd} and \eqref{A2-bd},  
at most $\ovl{F}(\ga) \omega_n (2\etas)^{n-1}$. Consequently, the total number of changes in $R_t^\ep$ 
restricted to $S_j$ and time interval $[0,T\wedge \kappa^\ep]$, is stochastically dominated 
by a Poisson variable of mean
\begin{equation}\label{lambda-ep}
\lambda_\ep := \ep^{-1} T \ovl{F}(\ga) \omega_n (2\etas)^{n-1} = C_\star \ep^{-1/n} \,,
\end{equation}
with $C_\star=C_\star(r_0,T,\ga)$ finite. Recall \eqref{eta} that we add to the domain the (local) bump
$\ep y^{-1}_{r,x} g_\eta = \ep^{1/n} \eta^{n-1} g_\eta$,
whose radial height is at most
\[
\ep^{1/n}\eta^{n-1}\|g_\eta\|_\infty\le \ep^{1/n}\Lambda,
\]
for some absolute finite constant $\Lambda$ and any $(r,x)$. Consequently, by a union bound 
over the $N_\ep$ spherical caps in our covering of $\bS^{n-1}$, 
\[
\P\Big(\|R^\ep_{T \wedge \kappa^\ep}\|_\infty \ge \|r_0\|_\infty + 2 C_\star \Lambda\Big)
\le  N_\ep \, \P( \; \hbox{ Poisson}(\lambda_\ep) \ge 2 \lambda_\ep) \,.
\]
By volume considerations $N_\ep \le c_n \etas^{-(n-1)}$ for some universal constant $c_n$.
In view of \eqref{lambda-ep} and the definition \eqref{eta-ubd} of $\etas$, 
this translates to $N_\ep \le C(n,r_0) \lambda_\ep^{n-1}$ for some finite 
$C(n,r_0)$. Thus, from the (super) exponential in $\lambda_\ep$ tail probabilities for 
a Poisson($\lambda_\ep)$ law, we deduce that for $\delta_n = (\|r_0\|_\infty+2 C_\star \Lambda)^{-1}$,
\[
\lim_{\ep\to0} \P\big( \|R^\ep_{T\wedge\kappa^\ep}\|_\infty > \delta_n^{-1}\big)=0
\]
and our claim then follows from
the definition \eqref{zeta} of $\zeta^\ep(\cdot)$.
\end{proof}

\begin{proof}[Proof of Proposition \ref{vol-growth}]
For $C_+(\bS^{n-1})$-solutions $( r_t )_{t\ge 0}$, \eqref{b-bar} is valid in pointwise sense and we can compute
\begin{align*}
\frac{d}{dt}\{Leb( r_t )\}&=\frac{d}{dt}\left\{\int_{\bS^{n-1}}n^{-1} r_t (\theta)^nd\sigma(\theta)\right\}=\int_{\bS^{n-1}} r_t (\theta)^{n-1}\frac{d}{dt}\{ r_t (\theta)\}d\sigma(\theta)\\
&=\int_{\bS^{n-1}} r_t (\theta)^{n-1}\int_{\R^n}b( r_t , x)(\theta)d\nu_{ r_t }(x)d\sigma(\theta)\\
&=\int_{\R^n}\left(\int_{\bS^{n-1}} r_t ^{n-1}(\theta)b( r_t , x)(\theta)d\sigma(\theta)\right)d\nu_{ r_t }(x)\\
&=\int_{\R^n}d\nu_{ r_t }(x)=1\, ,
\end{align*}
yielding $Leb( r_t )=Leb( r_0 )+t$, for any $t\ge 0$. 

We proceed to similarly verify \eqref{eq:Leb-add} for any $(r,x)\in\cD(F)$. Indeed, 
\begin{align*}
&\E[Leb(r+\ep y_{r,x}^{-1}g_\eta(\langle\xi, \cdot\rangle))-Leb (r)]\\
&=\int_{\bS^{n-1}}n^{-1}\int_{\bS^{n-1}}[r(\theta)+\ep y^{-1}_{r,x}g_\eta(\langle z,  \theta\rangle)]^nF(r,x,z)d\sigma(z)d\sigma(\theta)\\
&   \quad\quad\quad\quad -n^{-1}\int_{\bS^{n-1}}r(\theta)^nd\sigma(\theta)\\
&=\ep\int_{\bS^{n-1}}\int_{\bS^{n-1}}r(\theta)^{n-1}y^{-1}_{r,x}g_\eta(\langle z, \theta\rangle)F(r,x,z)d\sigma(z)d\sigma(\theta)+o(\ep)\\
&=\ep\int_{\bS^{n-1}}r(\theta)^{n-1}(b(r,x) \star g_\eta)(\theta)d\sigma(\theta)+o(\ep)\\
&=\ep\int_{\bS^{n-1}}r(\theta)^{n-1}[b^\ep(r,x)(\theta)-b(r,x)(\theta)]d\sigma(\theta)\\
&\quad\quad\quad\quad +\ep\int_{\bS^{n-1}}r(\theta)^{n-1}b(r,x)(\theta)d\sigma(\theta)+o(\ep).
\end{align*}
The second term gives exactly $\ep$. Upon applying Cauchy-Schwarz inequality to the first term and using the $L^2$-approximation property \eqref{approx-id} of the spherical approximate identity as $\ep\to 0$, we see that the whole expression is $\ep+o(\ep)$.
\end{proof}

\begin{proof}[Proof of Lemma \ref{molify}]
Recall the definition \eqref{integ-1} of spherical convolution. For any multi-index $\alpha$
with $|\alpha|=k\in \{0,1,2,3\}$ and any $z\in \bS^{n-1}$, we have that
\begin{align*}
|\partial^\alpha(\wt{r}-\wt{r}')(z)|&=\frac{1}{\omega_n}\Big| \int_{\bS^{n-1}}(r-r')(\theta)\partial^\alpha g(\langle  z, \theta \rangle ) d\sigma(\theta)\Big|\\
&\le \frac{1}{\omega_n}\|r-r'\|_2\sup_{z\in \bS^{n-1}}\|\partial^\alpha g(\langle z, \cdot \rangle)\|_2
\end{align*}
where $\partial^\alpha$ is any $k$-th order derivative with respect to $z$ variable on $\bS^{n-1}$. 
Since $\bS^{n-1}$ is compact, hence the supremum in the last line is finite and depends only on $g$, 
we arrive at the claimed bound \eqref{C3L2} 
on $\|\wt{r}-\wt{r}'\|_{C^3(\bS^{n-1})}$. In particular, from \eqref{C3L2} with 
$r'=0$ (hence $\wt{r}'=0$), we have that
$\|\wt{r}\|_{C^3(\bS^{n-1})} \le C \|r\|_2 \le C a^{-1}$ for any $r \in \cA_2(a)$. 
Since convolution with $g \ge 0$ does not lower the minimal value of $r \in \cA_2(a)$, 
it further follows that then $\inf_\theta \wt{r}(\theta) \ge a$. Finally, from \eqref{alpha:bdd}
the continuous and positive  $\ell-\alpha(\ell,z)$ is bounded away from zero 
on the compact $[a,C a^{-1}] \times \bS^{n-1}$. In particular, for some $\eta=\eta(a,C)>0$
and all $(r,x) \in \cAe_2(a)$ the radial distance of $x$ from $\partial D_{\wt{r}}$ must be
at least $\eta$. The uniform over $\cA_2(a)$ control on $\|\wt{r}\|_{C^3(\bS^{n-1})}$ 
implies a uniform bound on the 
Hessian of $\wt r (\cdot)$ and hence that $\ovl{\B}(x,5\delta)$ be within the open, 
star-shaped domain $D_{\wt{r}}$ 
for some $\delta=\delta(C a^{-1},\eta)>0$. In conclusion, if $(r,x) \in \cAe_2(a)$ 
then $(\wt{r},x) \in \cK_{a,C,\delta}$ of \eqref{def:K-a-c-delta}, as claimed.
\end{proof}

\begin{proof}[Proof of Proposition \ref{ppn:local-lip}] Fixing $a \in  (0,1)$, $C<\infty$ and
$\delta>0$, we 
write for brevity $\|(\wt{r},x)\| = \|\wt{r}\|+|x|$
and $\|\wt{r}\|$ for the $C^{2,\gamma}(\bS^{n-1})$ norm of
$\wt{r}$. By the Arzel\`a-Ascoli theorem, a $C^3(\bS^{n-1})$ closed ball is $C^{2,\gamma}(\bS^{n-1})$-compact, 
which clearly extends to the $(\wt{r},x)$-pairs with $\ovl{\B}(x,5\delta) \subseteq D_{\wt{r}}$,
hence also to $\cK_{a,C,\delta}$. It thus 
suffices to prove that $(\wt{r},x) \mapsto F(r,x,\cdot)$ is locally Lipshitz on $\cK_{a,C,\delta}$. Specifically,
fixing $(\wt{r}_\star,x_\star) \in \cK_{a,C,\delta}$, we denote by $\cN_\eta$ 
the set of $(\wt{r},x) \in \cK_{a,C,\delta}$ with
$\|(\wt{r},x)-(\wt{r}_\star,x_\star)\| < \eta$ and
proceed to show that 
\begin{align}\label{C-1-bd-F}
\|F(r',x',\cdot)-F(r,x',\cdot)\|_{C^{1,\gamma}(\bS^{n-1})} \le \kappa_0\|\wt{r}'-\wt{r}\| \,, \qquad
\forall (\wt{r}',x'), (\wt{r},x') \in \cN_\eta \,,
\end{align}
where $\eta>0$ and $\kappa_i<\infty$ depend only on 
$\|\wt{r}_\star\|$ and $\delta$.
To this end, let $D=D_{\wt r}$ and $D'=D_{\wt r'}$, noting  
that for $\eta<\delta$  
necessarily $\ovl{\B}(x',3\delta)\subset D \cap D'$. 
One can then construct a $C^3$-diffeomorphism $\phi:\R^n\to\R^n$ that maps $D$ 
to $D'$ such that $\wt r' (\theta) \theta = \phi(\wt r (\theta)\theta)$,
while being the identity map on $\ovl{\B}(x',\delta)$, and such that for some $\kappa_1<\infty$,
\begin{align}\label{phi-defor}
\|\phi - {\rm Id}\|_{C^{2, \gamma}(\ovl D)} \le \kappa_1\|\wt r'-\wt r\|  \le 2 \kappa_1 \eta \,.
\end{align}
Here ${\rm Id}(y)=y$ is the identity map, whose Jacobian matrix is $I_n$. 
Since Green's function is harmonic (in its second argument), away from its pole, 
\begin{align}\label{rem-sing}
\Delta G_{\wt{r}'}(x',\phi(y))-\Delta G_{\wt{r}}(x',y)=0 \,, \qquad \forall y \in \ovl D \,,
\end{align}
where at $y=x'$ this identity still holds, albeit in the distributional sense, since 
evaluating the \abbr{lhs} at any smooth function 
$f(\cdot)$ gives $f(\phi^{-1}(x'))-f(x')=0$
(by definition of the map $\phi$). Denoting by $\mathsf{D}\phi(y)$
the Jacobian matrix of $\phi$ at $y$ and 
\[
\Gamma(x',y) :=G_{\wt{r}'}(x',\phi(y)) \,, \quad {\rm for } \quad  y \in \ovl D 
\]  
(with a 
pole singularity at $y=x'$), we have that
\begin{equation}\label{Lap-Gamma}
\begin{aligned}
\Delta G_{\wt{r}'} (x',\phi(y)) &= \text{div}( \nabla G_{\wt{r}'}(x',\phi(y))) \,, \\    
\Delta\Gamma(x',y) &= \text{div}(\mathsf{D}\phi(y)\nabla G_{\wt{r}'}(x',\phi(y))) \,.
\end{aligned}
\end{equation}
Thus, upon combining \eqref{rem-sing} and \eqref{Lap-Gamma} we arrive at 
\begin{align}\label{schauder}
\Delta(\Gamma-G_{\wt r})(x',y) =\text{div}((\mathsf{D}\phi-I_n)(y)\nabla G_{\wt{r}'}(x',\phi(y)))\,.
\end{align}

We show next that for some $\kappa_2$ finite,
\begin{align}\label{holder-bd}
\|(\Gamma-G_{\wt r})(x',\cdot)\|_{C^{2, \gamma}(\ovl{D})}  \le 
\kappa_2\|\wt{r}'-\wt{r}\| \,.
\end{align}
Indeed, since $\mathsf{D}\phi=I_n$ throughout $\B (x',\delta)$,  there is no singularity on the \abbr{RHS} of \eqref{schauder}. With the boundary condition $(\Gamma-G_{\wt r})(x',\cdot) \equiv 0$ on
$\partial D$, we have by the global Schauder estimate of \cite[Theorem 5.26]{HL}, applied to the Poisson equation \eqref{schauder}, combined with the maximum principle  for the same equation 
(see \cite[Proposition 2.15]{HL}), that for some $\kappa_3$ finite,
\begin{align*}
 \|(\Gamma-G_{\wt r})(x',\cdot)\|_{C^{2, \gamma}(\ovl{D})}
\le \kappa_3 \|\text{div}((\mathsf{D}\phi-I_n)\nabla G_{\wt{r}'}(x',\phi(\cdot)))\|_{C^{0,\gamma}(\ovl{D})} \,.
\end{align*}
We arrive at \eqref{holder-bd} upon further bounding the preceding \abbr{rhs} by 
\begin{align*} 
\kappa_4 ||\phi-{\rm Id}||_{C^{2, \gamma}(\ovl D)} \|G_{\wt{r}'}(x',\phi(\cdot))\|_{C^{2, \gamma}
(\ovl{D}\backslash \B(x',\delta))}   \le \kappa_2 \|\wt{r}'-\wt{r}\| \,.
\end{align*}
The latter inequality is due to \eqref{phi-defor}, since
$G_{\wt{r}'}(x',\phi(\cdot))$ is $C^{2, \gamma}$ away from its pole at $x'$,
up to the $C^3$-boundary (at least for $\eta$ small, thanks to \eqref{phi-defor}).
Recall \eqref{def-conv-F}, that 
\begin{align*}
F(r',x',\theta)=\frac{\partial}{\partial\mathbf{n}}G_{\wt{r}'}(x', y)\big|_{y=\wt{r}'(\theta)\theta}, 
\quad F(r,x',\theta)=\frac{\partial}{\partial\mathbf{n}}G_{\wt{r}}(x', y)\big|_{y=\wt{r}(\theta)\theta} \,.
\end{align*}
We thus consider \eqref{holder-bd} at $\wt r = \partial D$, and get \eqref{C-1-bd-F} upon 
using 
\eqref{phi-defor} and the fact that the inward normal unit vectors of $\wt{r}$ and $\wt{r}'$ at points $y$ 
and $\phi(y)$ respectively, are close to each other (since $\|\wt{r}'-\wt{r}\| < 2 \eta$, possibly 
reducing the value of $\eta>0$ as needed).

In view of \eqref{C-1-bd-F}, we get the local Lipschitz property 
of $(\wt{r},x) \mapsto F(r,x,\cdot)$, upon showing that for $\delta>\eta>0$ as above and 
some $\kappa_5$ finite 
\begin{align}\label{lip-F-x-variable}
\| F(r,x',\cdot)-F(r,x,\cdot)\|_{C^1(\bS^{n-1})}\le \kappa_5 |x'-x| \,, \qquad 
\forall (\wt r,x'), (\wt r,x) \in \cN_\eta \,.
\end{align}
Indeed, by the preceding construction, 
$\ovl \B(x'', 3\delta) \subset \ovl \B(x, 5\delta) \subset D$ for any $x''$ on  
the line segment connecting $x$ and $x'$. Note that
$\| \partial_{z_i} G_{\wt r}(z,\cdot) \|_{C^2({\ovl D} \backslash(\B(z,\delta))}$ are
bounded uniformly over $z \in D$ of distance at least $3\delta$ from $\partial D$. 
Thus,
applying the mean value theorem to $x\mapsto G(x,y)$, $x\mapsto \partial_{y_i} G(x, y)$ 
and $x\mapsto \partial^2_{y_iy_j} G(x, y)$, we get for some finite $\kappa_6$,   
\begin{align*}
\|G(x',\cdot)-G(x,\cdot)\|_{C^2(\ovl{D}\backslash(\B(x',2\delta)\cup \B(x,2\delta)))}\le \kappa_6 |x'-x|,
\end{align*}
from which \eqref{lip-F-x-variable} follows (thanks to \eqref{def-conv-F}).
\end{proof}

\begin{lem}\label{lem:apx-iden}
Fixing $n \ge 2$, a collection $g_\eta$ forms a 
 local spherical approximate identity (as in Definition \ref{def:apx-iden}), if and only if
\begin{align}\label{def:g}
g_\eta(t):=\frac{c_\eta}{\omega_{n-1}}\eta^{-(n-1)}\phi_\eta 
\Big(1-\frac{1-t}{\eta^2}\Big) \quad t\in[-1,1],
\end{align} 
for some continuous $\phi_\eta(s) \ge 0$, supported on  $[-1,1]$, with $\|\phi_\eta\|_\infty=1$ and for $\eta \in (0,1)$,
\begin{align}\label{def:c-eta}
c_\eta^{-1} := \frac{2^{n-2}}{\omega_n}\int_0^1\phi_\eta (1-2s)s^{\frac{n-3}{2}}(1-\eta^2 s)^{\frac{n-3}{2}}ds \,,
\end{align}
are bounded away from zero. In particular, this applies whenever 
$\phi_\eta(\cdot)=\phi(\cdot)$ is independent of $\eta$, and not identically zero.
\end{lem}
\begin{proof} 
The mapping from $g_\eta(\cdot)$ to $\phi_\eta(\cdot)$ is merely a change of argument,
with $g_\eta(\cdot)$ continuous and non-negative iff $\phi_\eta(\cdot)$ are. Further, having the arbitrary
constant $c_\eta>0$ allows us to set \abbr{wlog} $\|\phi_\eta\|_\infty=1$. Our requirement that  
$g_\eta(\langle z, \theta \rangle)$ be supported on the 
spherical cap $|\theta-z| \le 2\eta$, $\theta \in \bS^{n-1}$ translates into 
$g_\eta(\cdot)$ supported on $[1-2\eta^2,1]$, or equivalently, to $\phi_\eta(\cdot)$ supported on $[-1,1]$,
whereas under \eqref{def:g} the uniform boundedness of $\eta^{n-1} \|g_\eta\|_\infty$ amounts 
to the same for $c_\eta$.
 
Next, by a change of variable (see \cite[(2.1.8)]{DX}),
\begin{align}
(1 \star g_\eta)(z) =\frac{\omega_{n-1}}{\omega_n}\int_{1-2\eta^2}^1g_\eta(t)(1-t^2)^{\frac{n-3}{2}}dt = 1 \,,
\quad \forall z \in \bS^{n-1} \,,
\label{integ=1}
\end{align}
iff $c_\eta^{-1}$ are given by \eqref{def:c-eta}. Further,
recall \cite[(2.1.8)]{DX}, that for every $z\in \bS^{n-1}$ and $f \in L^2(\bS^{n-1})$,
\begin{align}\label{rewrite-conv}
(f \star g_\eta)(z)=\frac{\omega_{n-1}}{\omega_n}\int_{1-2\eta^2}^1g_\eta(t)T_t f(z)(1-t^2)^{\frac{n-3}{2}}dt,
\end{align} 
where $\{T_t\}_{t\in[-1,1]}$ is a family of translation operators \cite[(2.1.6)]{DX} defined by
\begin{equation}\label{def:Tt}
T_t f(x) :=\frac{(1-t^2)^{\frac{1-n}{2}}}{\omega_{n-1}}\int_{\langle x,y \rangle=t} f(y) \, d\ell_{x,t}(y).
\end{equation}
Here $d\ell_{x,t}$ denotes Lebesgue measure on 
$\{y \in \bS^{n-1} \colon \langle x,y \rangle=t\}$. These operators satisfy
\begin{align}\label{per-t-bd}
\forall t, \;\; \|T_t f\|_2\le \|f\|_2 \, ,\qquad \lim_{t\to 1^-}\|T_t f-f\|_2= 0, 
\end{align}
see \cite[Lemma 2.1.7]{DX}. Hence, by \eqref{integ=1}-\eqref{per-t-bd} and the convexity of the norm, for all $\eta \in (0,1)$,
\begin{align}\label{contraction}
\|f \star g_\eta\|_2 &\le \frac{\omega_{n-1}}{\omega_n}\int_{1-2\eta^2}^1g_\eta(t)\|T_tf\|_2(1-t^2)^{\frac{n-3}{2}}dt\le \|f\|_2 \,, \\
\|f \star g_\eta-f\|_2 &\le \frac{\omega_{n-1}}{\omega_n}\int_{1-2\eta^2}^1g_\eta(t)\|T_tf-f\|_2(1-t^2)^{\frac{n-3}{2}}dt \nonumber \\
& \le \sup_{t \in [1-2\eta^2,1]} \{ \|T_t f - f \|_2 \} 
\label{approx-id}
\end{align}
with the \abbr{rhs} of \eqref{approx-id} converging to zero 
as $\eta\to 0$ (see \eqref{per-t-bd}). Finally, note that when $\phi_\eta(\cdot)=\phi(\cdot)$ is
independent of $\eta$, we have from \eqref{def:c-eta} that $\eta \mapsto c_\eta^{-1}$ is monotone,
with $c_0^{-1},c_1^{-1}>0$, for any non-zero $\phi$ and $n \ge 2$.
\end{proof}

\medskip
\noindent
{\bf{Acknowledgments.}}
We thank Juli\'an Fern\'andez Bonder and Luis Silvestre for useful conversations and Mart\'in Arjovsky for pointing out the plausibility of Lemma \ref{H_-1} and its proof.

\bibliographystyle{plain}
\bibliography{biblio}

\end{document}